\documentclass[a4paper,10pt,final]{article}

\pdfoutput=1
\usepackage{graphicx}
\usepackage{amssymb,amsmath,amsthm,mathrsfs}
\usepackage{enumerate,color}
\usepackage{a4wide}
\usepackage{hyperref}

\numberwithin{equation}{section}
\allowdisplaybreaks[4]

\newtheorem{proposition}{Proposition}[section]
\newtheorem{theorem}[proposition]{Theorem}
\newtheorem{lemma}[proposition]{Lemma}

\newtheorem{definition}[proposition]{Definition}
\newtheorem{remark}[proposition]{Remark}

\renewenvironment{proof}{\smallskip\noindent\emph{\textbf{Proof.}}%
  \hspace{1pt}}{\hspace{-5pt}{\nobreak\quad\nobreak\hfill\nobreak%
    $\square$\vspace{2pt}\par}\smallskip\goodbreak}

\newenvironment{proofof}[1]{\smallskip\noindent{\textbf{Proof~of~#1.}}%
  \hspace{1pt}}{\hspace{-5pt}{\nobreak\quad\nobreak\hfill\nobreak%
    $\square$\vspace{2pt}\par}\smallskip\goodbreak}

\newcommand{\pint}[1]{\mathaccent23{#1}}
\newcommand{\C}[1]{\mathbf{C}^{#1}}
\newcommand{\Cc}[1]{\mathbf{C}_c^{#1}}

\newcommand{\BV}{\mathbf{BV}}

\renewcommand{\L}[1]{{\mathbf{L}^#1}}

\newcommand{\W}[2]{{\mathbf{W}^{#1,#2}}}

\newcommand{\modulo}[1]{{\left|#1\right|}}
\newcommand{\norma}[1]{{\left\|#1\right\|}}
\newcommand{\caratt}[1]{{\chi_{\strut#1}}}
\newcommand{\reali}{{\mathbb{R}}}
\newcommand{\naturali}{{\mathbb{N}}}

\renewcommand{\epsilon}{\varepsilon}
\renewcommand{\phi}{\varphi}
\renewcommand{\theta}{\vartheta}
\renewcommand{\O}{\mathcal{O}(1)}
\newcommand{\tv}{\mathinner{\rm TV}}

\newcommand{\sgn}{\mathop{\rm sgn}}

\newcommand{\Lip}{\mathinner\mathbf{Lip}}
\renewcommand{\d}[1]{\mathinner{\mathrm{d}{#1}}}
\renewcommand{\div}{\mathinner{\mathop{{\rm div}}}}

\newcommand{\f}{\boldsymbol{f}}

\DeclareMathOperator{\Div}{Div}
\DeclareMathOperator{\grad}{grad}
\DeclareMathOperator{\tr}{tr}

\makeatletter
\let\@fnsymbol\@arabic
\makeatother

\title{Rigorous Estimates on Balance Laws in Bounded Domains}

\author{Rinaldo M.~Colombo\footnote{Unit\`a INdAM, Universit\`a di
    Brescia, Via Branze 38, 25123 Brescia, Italy
    (\texttt{rinaldo.colombo@unibs.it}).}  \and Elena
  Rossi\footnote{Dipartimento di Matematica ed Applicazioni,
    Universit\`a di Milano-Bicocca, Via Cozzi 55, 20125 Milano, Italy
    (\texttt{e.rossi50@campus.unimib.it}).}}
\date{}
\begin{document}

\maketitle




\begin{abstract}

  \noindent The initial -- boundary value problem for a general
  balance law in a bounded domain is proved to be well posed. Indeed,
  we show the existence of an entropy solution, its uniqueness and its
  Lipschitz continuity as a function of time, of the initial datum and
  of the boundary datum. The proof follows the general lines
  in~\cite{BardosLerouxNedelec}, striving to provide a rigorous
  treatment and detailed references.

  \medskip

  \noindent\textit{2000~Mathematics Subject Classification:} 35L50,
  35L65.

  \medskip

  \noindent\textit{Keywords:} Balance Laws, Initial -- Boundary Value
  Problem for Balance Laws.
\end{abstract}

\section{Introduction}
\label{sec:Intro}

This paper is devoted to the well posedness of a general scalar
balance law in an $n$-dimensional bounded domain, that is of
\begin{equation}
  \label{eq:1}
  \left\{
    \begin{array}{l@{\qquad}r@{\;}c@{\;}l}
      \partial_t u + \Div f (t,x,u) = F (t,x,u)
      &
      (t,x) & \in & I \times \Omega
      \\
      u (0,x) = u_o (x)
      & x & \in & \Omega
      \\
      u (t,\xi) =  u_{b} (t,\xi)
      & (t,\xi) & \in & I \times \partial\Omega \, .
    \end{array}
  \right.
\end{equation}
A key reference in this context is the classical paper by Bardos,
Leroux and N\'ed\'elec~\cite{BardosLerouxNedelec}. There, the
\emph{``correct''} definition of solution to~\eqref{eq:1} is selected,
in the spirit of the definition given by Kru\v zkov in the case
$\Omega = \reali^n$, see~\cite[Definition~1]{Kruzkov}. A proof of the
existence, uniqueness and continuous dependence of the solution from
the initial data is described in~\cite{BardosLerouxNedelec} in the
case $u_b = 0$.

For its relevance, since its publication, the well posedness
of~\eqref{eq:1} proved in~\cite{BardosLerouxNedelec} was refined or
explained in various text books, mostly in particular cases. For
instance, the case $f = f (u)$, $F = 0$ and $u_b = 0$ is detailed
in~\cite[Section~6.9]{DafermosBook}, while non homogeneous boundary
conditions are considered in~\cite[Section~15.1]{Serre2}, always in
the case $f = f (u)$, $F = 0$. A different type of boundary condition is considered, for instance, in~\cite{AndreianovSbihi}.

Below, we aim at a presentation which covers the general
case~\eqref{eq:1}, which is self contained and with precise references
to the elliptic or parabolic results required. Where possible, we also
seek to underline which regularity is necessary at which step. As a
result, we also obtain further estimates on the solution
to~\eqref{eq:1}.

As in~\cite{BardosLerouxNedelec} and~\cite{Kruzkov}, existence of
solution is obtained through the vanishing viscosity technique. The
usual term $\epsilon \, \Delta u$ is added on the right hand side of
the equation in~\eqref{eq:1}, turning it into the parabolic problem
\begin{equation}
  \label{eq:visco}
  \left\{
    \begin{array}{l@{\qquad}r@{\;}c@{\;}l}
      \partial_t u_\epsilon + \Div f (t,x,u_\epsilon) = F (t,x,u_\epsilon) + \epsilon \; \Delta u_\epsilon
      &
      (t,x) & \in & I \times \Omega
      \\
      u_\epsilon (0,x) = u_o (x)
      & x & \in & \Omega
      \\
      u_\epsilon (t,\xi) =  u_{b} (t,\xi)
      & (t,\xi) & \in & I \times \partial\Omega \,,
    \end{array}
  \right.
\end{equation}
which is first considered under stricter conditions (regularity and
compatibility of the data). Classical results from the parabolic
literature~\cite{FriedmanParabolic, LadysenkajaSolonnikovUralceva,
  LadysenkajaUralceva1961} can then be applied to ensure the existence
of a solution $u_\epsilon$ to~\eqref{eq:visco} in the case $u_b =
0$. To pass to the limit $\epsilon \to 0$, suitable bounds on
$u_\epsilon$ are necessary. First, the $\L\infty$
bound~\eqref{eq:Linfty} is fundamental. In this connection, we note
that the similar bound~\cite[Formula~(9)]{BardosLerouxNedelec} lacks a
term that should be present also in the case $u_b = 0$ considered
therein, refer to Section~\ref{sec:PP} for more details. Then, a
tricky $\BV$ bound allows to prove that the family of solutions
to~\eqref{eq:visco} is relatively compact in $\L1$, so that the limit
of any convergent subsequence of the $u_\epsilon$ solves~\eqref{eq:1}.

The next step is the problem with $u_b \neq 0$, a situation hardly
considered in the literature. To rigorously extend the existence of
solutions to the non homogeneous case, a time and space dependent
translation in the $u$ space of the solution to~\eqref{eq:1} is
necessary. This leads on one side to the need of solving an elliptic
problem and, on the other side, to prove that this translation does
indeed give a solution to~\eqref{eq:1}. An \emph{ad hoc} adaptation of
the doubling of variables technique from~\cite{Kruzkov} makes this
latter proof possible. However, to get the necessary estimates on the
translated balance law~\eqref{eq:23}, strict regularity requirements
on the elliptic problem are necessary, see
Lemma~\ref{lem:elliptic}. All this leads to keep, in the present work,
strict regularity assumptions on $u_b$ and the condition that $u_b
(0,\xi) = 0$ for all $\xi \in \partial\Omega$.

At this stage, the existence of solutions to~\eqref{eq:1} is proved,
under rather strict conditions on initial and boundary data. A further
use of the doubling of variables technique allows to prove the
Lipschitz continuous dependence of the solution form the initial and
boundary data. Remarkably, this technique allows to obtain a proof
that essentially relies only on the definition of solution, in a
generality wider than that available for the existence of
solutions. Finally, we thus obtain at once also the uniqueness of
solutions to~\eqref{eq:1} and to relax the necessary condition on the
initial datum.

The bounds on the total variation of the solution have a key role
throughout this work. First, they are obtained in the case $u_b = 0$,
similarly to what is done in~\cite{BardosLerouxNedelec},
see~\eqref{eq:tb_original}--\eqref{eq:21}. This bound depends on the
total variation of the initial datum and on various norms of the flow
$f$ and of the source $F$. The translation that allows to pass to the
non homogeneous problem leads to consider a translated balance law,
where the translated flow and source depend on an extension of the
boundary data, see~\eqref{eq:23}. Therefore, the bound on the total
variation of the solution to the translated problem depends on high
norms of the boundary datum, see~\eqref{eq:tv_bound}, which in the end
imposes to keep the condition $u_b (0,\xi) = 0$ for all $\xi
\in \partial\Omega$.

The paper is organized as follows. The next section is devoted to the
main result, which is obtained through estimates on the parabolic
approximation~\eqref{eq:visco} to~\eqref{eq:1}, presented in
Section~\ref{sec:PP}. Then, Section~\ref{sec:HP} accounts for the
hyperbolic results. All proofs are deferred to sections~\ref{sec:PPP},
\ref{sec:PHP} and~\ref{sec:PNDA}. A final appendix gathers useful
information on the trace operator.

\section{Notations, Definitions and Main Result}
\label{sec:NDA}

Throughout, $\reali^+ = \left[0, +\infty\right[$, $B (x,r)$ denotes
the open ball centered at $x$ with radius $r>0$. The closed real
interval $I = [0,T]$ is fixed, $T$ being completely arbitrary. For the
divergence of a vector field, possibly composed with another function,
we use the notation
\begin{displaymath}
  \Div f\left(t,x,u (t,x)\right)
  =
  \div f\left(t,x,u (t,x)\right)
  +
  \partial_u f\left(t,x,u (t,x)\right)
  \cdot
  \grad u (t,x) \,.
\end{displaymath}
The Lebesgue $n$ dimensional measure of $\Omega$ is denoted
$\mathcal{L}^n (\Omega)$, while the Hausdorff $n-1$ dimensional
measure of $\partial\Omega$ is $\mathcal{H}^{n-1} (\partial\Omega)$.

\noindent We use below the following standard assumptions, where $\ell
\in \naturali$ and $\alpha \in \left[0, 1\right[$:
\begin{description}
\item[($\boldsymbol{\Omega_{\ell,\alpha}}$)] $\Omega$ is a bounded
  open subset of $\reali^n$ with piecewise $\C{\ell,\alpha}$ boundary
  $\partial \Omega$ and exterior unit normal vector $\nu$.
\item[(f)] $f \in \C2 (\Sigma; \reali^n)$, $\partial_u f \in
  \L\infty(\Sigma; \reali^n)$, $\partial_u \div f \in \L\infty(\Sigma;
  \reali)$.
\item[(F)] $F \in \C2 (\Sigma; \reali)$, $\partial_u F \in \L{\infty}
  (\Sigma; \reali)$.
\item[(C)] $u_o \in (\BV \cap \L\infty) (\Omega ; \reali)$ and $u_{b}
  \in (\BV \cap \L\infty) (I \times \partial \Omega; \reali)$.
\end{description}

\noindent Above and in the sequel, we denote
\begin{displaymath}
  \Sigma = I \times \bar{\Omega} \times \reali \,.
\end{displaymath}
Above, we followed the choice in~\cite[Chapter~10]{Serre2} of a
boundary data with bounded total variation. Refer
to~\cite[Section~2.6]{MalekNecasRokyta} and~\cite{Otto1996} for a
generalization to $\L\infty$ boundary data. For a definition of
functions of bounded variation on a manifold, refer for instance
to~\cite[Definition~3.1]{KronerMullerStrehlau}.

Our starting point is the definition of solution, which originates in
the work of Vol'pert~\cite{Volpert1967}, see
also~\cite{BressanLectureNotes, DafermosBook, Kruzkov, Serre2}

\begin{definition}[{\cite[p.1028]{BardosLerouxNedelec}}]
  \label{def:sol}
  Let $\Omega$ satisfy~{\bf($\boldsymbol{\Omega_{2,0}}$)}. Fix $u_o$
  and $u_b$ satisfying~{\bf{(C)}}. A \emph{solution} to~\eqref{eq:1}
  on $I$ is a map $u \in (\L\infty \cap \BV) (I \times \Omega;
  \reali)$ such that for any test function $\phi \in \Cc2
  (\left]-\infty, T\right[ \times \reali^n; \reali^+)$ and for any $k
  \in \reali$,
  \begin{equation}
    \label{eq:4}
    \!\!\!\!
    \begin{array}{@{}rcl@{}}
      \displaystyle
      \int_I \!\int_\Omega
      \Big\{
      \modulo{u (t,x) -k} \, \partial_t\phi (t,x)
      +
      \sgn (u (t,x) -k) \,
      \left(f (t,x,u) - f (t,x,k)\right) \cdot \grad \phi (t,x)
      \\
      \displaystyle
      +
      \sgn (u (t,x) -k) \,
      \left( F (t,x,u) - \div f (t,x,k) \right)\, \phi (t,x)
      \Big\} \d{x} \d{t}
      \\
      \displaystyle
      +
      \int_\Omega
      \modulo{u_o (x) - k} \; \phi (0,x) \d{x}
      \\
      \displaystyle
      -
      \int_I \!\int_{\partial\Omega}  \sgn (u_b (t,\xi) - k)
      \left(
        f\left(t,\xi,\left(\tr u\right) (t,\xi)\right) - f (t,\xi,k)
      \right) \cdot \nu (\xi) \,
      \phi (t,\xi) \,
      \d{\xi} \d{t}
      & \geq & 0
    \end{array}
    \!\!\!\!\!\!\!\!\!
  \end{equation}
\end{definition}

\noindent Above, $\tr u (t,\xi)$ denotes the trace of the map $x \to u
(t,x)$ on $\partial\Omega$ evaluated at $\xi$. More information and
references on the trace operator are collected in the Appendix.

Now, we recall consequences of Definition~\ref{def:sol} specifying the
sense in which the initial datum is attained.

\begin{proposition}
  \label{prop:sol0}
  Let~{\bf($\boldsymbol{\Omega_{2,0}}$)}, {\bf(f)}, {\bf(F)}
  and~{\bf(C)} hold. Let $u \in (\BV \cap \L\infty) (I \times \Omega;
  \reali)$ be a solution to~\eqref{eq:1} in the sense of
  Definition~\ref{def:sol}. Then, there exists a set $\mathcal{E}
  \subset I$ of Lebesgue measure $0$ such that
  \begin{displaymath}
    \lim_{t\to 0+, \, t \in I \setminus \mathcal{E}}
    \int_\Omega \modulo{u (t,x) - u_o (x)} \d{x}
    =
    0 \,.
  \end{displaymath}
\end{proposition}

\noindent The proof is deferred to Section~\ref{sec:PNDA}.

A further similar consequence of the above definition of solution and
of the properties of the trace operator is the following
Proposition. It gives information on the way in which the values of
the boundary data are attained by the solution.

\begin{proposition}
  \label{prop:min}
  Let~{\bf($\boldsymbol{\Omega_{2,0}}$)}, {\bf(f)}, {\bf(F)}
  and~{\bf(C)} hold. Let $u \in (\BV \cap \L\infty) (I \times \Omega;
  \reali)$ be a solution to~\eqref{eq:1} in the sense of
  Definition~\ref{def:sol}. Then, for all $k \in \reali$ and for
  almost every $(t,\xi) \in \pint{I} \times \partial\Omega$,
  \begin{equation}
    \label{eq:6}
    \!\!\!
    \left[
      \sgn\left(\tr u (t,\xi)  - k\right)
      -
      \sgn\left(u_{b} (t,\xi)  - k\right)
    \right]
    \left[f \left(t,\xi,\left(\tr u\right) (t,\xi)\right) - f (t,\xi,k)\right]
    \cdot \nu (\xi)
    \geq 0.
  \end{equation}
  Moreover, for almost every $(t,\xi) \in
  \pint{I}\times\partial\Omega$
  \begin{equation}
    \label{eq:alive}
    \min_{k \in \mathcal{I} (t,\xi)}
    \sgn\left( \tr u (t,\xi) - u_b (t,\xi) \right)
    \left[
      f \left( t, \xi, \tr u (t,\xi) \right)
      -
      f \left( t,\xi, k \right)
    \right]
    \cdot \nu (\xi) = 0,
  \end{equation}
  where $\mathcal{I} (t,\xi) = \left\{ k \in \reali \colon \left( u_b
      (t,\xi) - k \right) \left( k - \tr u (t,\xi) \right) \geq 0
  \right\} $.
\end{proposition}

\noindent The proof is deferred to Section~\ref{sec:PNDA}. In other
words, \eqref{eq:alive} states that $\tr u$ and $u_b$ may differ
whenever the jump between them gives rise to waves exiting $\Omega$.

Recall now the classical concept of entropy -- entropy flux pair, in
the general case~\eqref{eq:1}.

\begin{definition}
  \label{def:pair}
  An entropy -- entropy flux pair for equation $\partial_t u + \Div
  f(t,x,u) = F (t,x,u)$ is a pair of functions
  $(\mathcal{E},\mathcal{F})$ such that:
  \begin{enumerate}
  \item $\mathcal{E} \in \C2 (\reali; \reali)$ and $\mathcal{F} \in
    \C2 (I \times \Omega \times \reali; \reali^n)$;
  \item $\mathcal{E}$ is convex;
  \item for all $(t,x,u) \in \Sigma$, $\mathcal{E}' (u) \, \partial_u
    f (t,x,u) = \partial_u \mathcal{F} (t,x,u)$.
  \end{enumerate}
\end{definition}

\noindent In the case of the general balance law~\eqref{eq:1}, the
differential form of the entropy inequality is
\begin{align*}
  \partial_t \mathcal{E} \left(u (t,x)\right) + \Div
  \mathcal{F}\left(t,x,u (t,x)\right) \leq \ & \mathcal{E}'\left(u
    (t,x)\right) \left( \mathcal{F}\left(t,x,u (t,x)\right) - \div
    f\left(t,x,u (t,x)\right) \right)
  \\
  & \qquad + \div \mathcal{F}\left(t,x,u (t,x)\right) .
\end{align*}
Particular cases of this expression are considered, for instance,
in~\cite{BressanLectureNotes, BulicekGwiazda, DafermosBook,
  GagneuxBook, KarlsenRisebro2003, Serre2}.

\begin{definition}
  \label{def:e_sol}
  Let $\Omega$ satisfy~{\bf($\boldsymbol{\Omega_{2,0}}$)}. Fix $u_o$
  and $u_b$ satisfying~{\bf{(C)}}.  An \emph{entropy solution}
  to~\eqref{eq:1} is a map $u \in (\L\infty \cap \BV) (I \times
  \Omega; \reali)$ such that for any entropy -- entropy flux pair
  $(\mathcal{E}, \mathcal{F})$ and for any $\phi \in \Cc2
  (\left]-\infty,T\right[ \times \reali^n; \reali^+)$, the following
  inequality holds:
  \begin{equation}
    \label{eq:dis_entro}
    \!\!\!\!\!\!\!\!\!
    \begin{array}{@{}l@{}}
      \displaystyle
      \int_I \int_\Omega
      \Big\{
      \mathcal{E}\left(u (t,x)\right) \partial_t \phi (t,x)
      +
      \mathcal{F}\left(t,x,u (t,x)\right) \cdot \grad \phi (t,x)
      \\
      \qquad
      +
      \left[
        \mathcal{E}'\left(u (t,x)\right)
        \left(
          F\left(t,x,u (t,x)\right)
          -
          \div f \left(t,x,u (t,x)\right)
        \right)
        +
        \div \mathcal{F}\left(t,x,u (t,x)\right)
      \right]
      \phi (t,x)
      \Big\}
      \d{x} \d{t}
      \\
      \displaystyle
      +
      \int_\Omega \mathcal{E}\left(u_o (x)\right) \; \phi (0,x)\d{x}
      \\
      \displaystyle
      -
      \int_I \int_{\partial\Omega}
      \Big[
      \mathcal{F}\left(t, \xi, u_b (t,\xi)\right)
      \\
      \qquad
      -
      \mathcal{E}'\left(u_b (t,\xi)\right)
      \left(
        f\left(t,\xi,u_b (t,\xi)\right)-
        f\left(t, \xi, \tr u (t,\xi)\right)
      \right)
      \Big]
      \!\!\cdot \nu (\xi) \phi (t,\xi)
      \d\xi \d{t}
      \geq 0 \, .
    \end{array}
    \!\!\!\!\!\!\!\!\!\!\!\!
  \end{equation}
\end{definition}

Formally, Definition~\ref{def:sol} is a \emph{``particular''} case of
Definition~\ref{def:e_sol}, obtained choosing as entropy -- entropy
flux pair the maps
\begin{displaymath}
  \mathcal{E} (u) = \modulo{u-k}
  \quad \mbox{ and } \quad
  \mathcal{F} (t,x,u) = \sgn (u-k) \left(f (t,x,u) - f (t,x,k)\right) ,
\end{displaymath}
for $k \in \reali$. However, the two definitions actually coincide.

\begin{proposition}
  \label{prop:equi}
  Definitions~\ref{def:sol} and~\ref{def:e_sol} are equivalent for
  bounded solutions.
\end{proposition}

\noindent This Proposition is well known and its proof is briefly
sketched in Section~\ref{sec:PNDA}.

We are now ready to state the main result of this paper.

\begin{theorem}
  \label{thm:Dream}
  Let $T>0$, $\alpha \in \left]0, 1\right[$, and assume
  that~{\bf($\boldsymbol{\Omega_{3,\alpha}}$)}, {\bf(f)} and~{\bf(F)}
  hold. Fix an initial datum $u_o \in (\L\infty \cap \BV) (\Omega;
  \reali)$ and a boundary datum $u_b \in \C{3,\alpha} (I
  \times \partial\Omega; \reali)$ with $u_b (0,\xi) = 0$ for all $\xi
  \in \partial\Omega$. Then, problem~\eqref{eq:1} admits a unique
  solution $u \in \C{0,1}\left(I; \L1
    (\Omega;\reali)\right)$. Moreover, the following estimates hold:
  \begin{eqnarray}
    \nonumber
    \norma{u (t)}_{\L\infty (\Omega;\reali)}
    & \leq &
    \left(
      \norma{u_o}_{\L\infty (\Omega;\reali)}
      +
      \norma{u_b}_{\L\infty ([0,t]\times\Omega;\reali)}
    \right)
    e^{c_1 t}
    \\
    \label{eq:L_bound_final}
    & &
    \qquad
    +
    \frac{c_2 + \norma{\partial_t u_b}_{\L\infty ([0,t]\times\Omega;\reali)}}{c_1}
    \left(e^{c_1 t}-1\right)
    \\
    \nonumber
    \tv\left(u (t)\right)
    & \leq &
    \mathcal{C} (\Omega,f,F,t)
    \left(
      \norma{u_b}_{\C{3,\alpha} ([0,t]\times\partial\Omega;\reali)}
      +
      \norma{u_b}_{\C{3,\alpha}([0,t]\times\partial\Omega;\reali)}^2
    \right)
    \\
    \label{eq:tv_bound_final}
    & &
    \times
    \left(
      1
      +
      t
      +
      \norma{u_o}_{\L\infty (\Omega;\reali)}
      +
      \tv (u_o)
    \right)
    \\
    \nonumber
    & &
    \times
    \exp
    \left(
      \mathcal{C} (\Omega,f,F,t)
      (1+\norma{u_b}_{\C{2,\alpha} ([0,t]\times\partial\Omega;\reali)})\,t
    \right)
    \\
    \label{eq:t_Lip_final}
    \norma{u (t) - u (s)}_{\L1 (\Omega;\reali)}
    & \leq &
    \left(\sup_{\tau \in [s,t]} \tv\left(u (\tau)\right)\right)
    \modulo{t-s}
  \end{eqnarray}
  for $t,\,s \in I$, where $c_1$, $c_2$ and $\mathcal{C}
  (\Omega,f,F,t)$ are independent of the initial and boundary data,
  see~\eqref{eq:c1c2} and~\eqref{eq:finita}.
\end{theorem}

The proof consists of the lemmas and propositions in the sections
below, together with the final bootstrap procedure presented in
Section~\ref{sec:PNDA}. The Lipschitz continuous dependence of the
solution from the initial and boundary data is stated and proved in
Theorem~\ref{thm:wp}.

\begin{remark}
  \label{rem:BV}
  {\rm The above estimate~\eqref{eq:L_bound_final} shows that the
    solution $u$ is in $\L\infty (I\times \Omega;
    \reali)$. By~\eqref{eq:tv_bound_final}, we also have $u (t) \in
    \BV (\Omega; \reali)$ for every $t \in I$. The Lipschitz
    continuity in time ensured by~\eqref{eq:t_Lip_final} then implies
    that $u \in (\L\infty \cap \BV) (I\times \Omega; \reali)$, as
    required in Definition~\ref{def:sol}. This can be proved using
    exactly the arguments in~\cite[Section~2.5, Proof of
    Theorem~2.6]{BressanLectureNotes}.}
\end{remark}

\section{The Parabolic Problem~(\ref{eq:visco})}
\label{sec:PP}

All proofs of the statements in this Section are deferred to
Section~\ref{sec:PPP}. Note that the results in this section are
obtained without requiring that $u_b=0$.

The next Lemma provides the existence of classical solutions to the
parabolic problem~\eqref{eq:visco}.

\begin{lemma}
  \label{lem:EeUpar}
  Fix $\alpha \in \left]0,1\right[$. Let
  conditions~{\bf($\boldsymbol{\Omega_{2,\alpha}}$)}, {\bf(f)}
  and~{\bf(F)} hold. Assume moreover that there exists a function
  $\bar u \in \C{2,\delta} (I \times \bar\Omega; \reali)$, with
  $\delta \in \left]\alpha, 1\right[$, such that
  \begin{equation}
    \label{eq:smoothData}
    \begin{array}{@{}l}
      \partial_t \bar u (0,\xi)
      +
      \Div f \left(0, \xi, \bar u (0,\xi)\right)
      =
      F\left(0, \xi, \bar u (0,\xi)\right)
      +
      \epsilon \; \Delta \bar u (0,\xi)
      \\[3pt]
      u_o (\xi) =  \bar u (0,\xi) = u_b (0, \xi)
    \end{array}
    \qquad
    \mbox{ for all } \xi \in \partial\Omega.
  \end{equation}
  Then, setting
  \begin{equation}
    \label{eq:compInDat}
    u_o (x) = \bar u (0,x)
    \quad \mbox{ for all } x \in \bar\Omega
    \quad \mbox{ and } \quad
    u_b (t,\xi) = \bar u (t,\xi)
    \quad \mbox{ for all } (t,\xi) \in I \times \partial\Omega \,,
  \end{equation}
  there exists a unique solution $u_\epsilon$ to~\eqref{eq:visco} of
  class $\C{2,\gamma} (I \times \bar{\Omega}; \reali)$, for a suitable
  $\gamma \in \left]\delta, 1\right[$.
\end{lemma}

We now provide an $\L\infty$--estimate for the solution $u_\epsilon$
to~\eqref{eq:visco}. It is important to note that we obtain a bound
that holds uniformly in $\epsilon$, see~\eqref{eq:Linfty}.

\begin{lemma}
  \label{lem:stimaLinf}
  Fix $\alpha \in \left]0,1\right[$. Let
  conditions~{\bf($\boldsymbol{\Omega_{2,\alpha}}$)}, {\bf(f)}
  and~{\bf(F)} hold. Assume moreover there exists a function $\bar u
  \in \C{2,\delta} (I \times \bar\Omega; \reali)$, for $\delta \in
  \left]\alpha, 1\right[$, such that~\eqref{eq:smoothData} holds.  Let
  $u_\epsilon$ be a solution to~\eqref{eq:visco} with $u_o$ and $u_b$
  as in~\eqref{eq:compInDat}. Then, for all $t \in I$,
  \begin{equation}
    \label{eq:Linfty}
    \norma{u_\epsilon}_{\L\infty ([0,t] \times \Omega; \reali)}
    \leq
    \left(
      \norma{u_o}_{\L\infty (\Omega; \reali)}
      +
      \norma{u_b}_{\L\infty ([0,t] \times \partial\Omega;\reali)}
    \right)
    e^{c_1\, t}
    +
    \frac{c_2}{c_1}
    \left(e^{c_1\, t} - 1\right) \,,
  \end{equation}
  where $c_1, \, c_2$ are constants depending on the $\L\infty$ norms
  of $\div f$, $\partial_u \div f$, $F$ and $\partial_u F$, as defined
  in~\eqref{eq:c1c2}.
\end{lemma}

We remark that, also in the case $u_b = 0$, due to the presence of the
second addend in the right hand side, the above
estimate~\eqref{eq:Linfty} significantly differs from the $\L\infty$
bound~\cite[Formula~(9)]{BardosLerouxNedelec}, which can not be
true. Indeed, the estimate~\cite[Formula~(9)]{BardosLerouxNedelec}
implies that the solution to~\eqref{eq:visco} with $u_o= 0$ and $u_b =
0$ is $u = 0$, which is false as, for instance, the case where $f
(t,x,u) = - x$ and $F=0$ clearly shows.

Consider now problem~\eqref{eq:visco} with homogeneous boundary
condition, i.e., $u_b (t,\xi) = 0$ for $(t,\xi) \in I
\times \partial\Omega$. In the next Lemma we partly
follow~\cite[Theorem~1]{BardosLerouxNedelec}, \cite[Chapter~6,
\S~6.9]{DafermosBook} and~\cite[Chapter~4]{GagneuxBook}. Introduce the
notation
\begin{equation}
  \label{eq:Ut}
  \begin{array}{rcl}
    \mathcal{U} (t)
    & = &
    \displaystyle
    [-\mathcal{M} (t), \mathcal{M} (t)]
    \quad \mbox{ with }
    \\
    \mathcal{M} (t)
    & = &
    \displaystyle
    \left(
      \norma{u_o}_{\L\infty (\Omega; \reali)}
      +
      \norma{u_b}_{\L\infty ([0,t] \times \partial\Omega;\reali)}
    \right)
    e^{c_1\, t}
    +
    \frac{c_2}{c_1}
    \left(e^{c_1\, t} - 1\right),
  \end{array}
\end{equation}
as in~\eqref{eq:Linfty} and~\eqref{eq:c1c2}.

\begin{lemma}
  \label{lem:tvParZero}
  Fix $\delta \in \left]0, 1\right[$.  Let
  conditions~{\bf($\boldsymbol{\Omega_{2,\delta}}$)}, {\bf(f)}
  and~{\bf(F)} hold. Assume moreover that $u_o \in \C{2,\delta}
  (\bar{\Omega} ; \reali)$ is such that $u_o (\xi) =0$ for all $\xi
  \in \partial\Omega$ and $u_{b} (t,\xi) = 0$ for $(t,\xi) \in I
  \times \Omega$. Let $u_\epsilon \in \C2 (I \times \bar \Omega;
  \reali)$ be a solution to~\eqref{eq:visco}. Then,
  \begin{align}
    \label{eq:tvPar}
    \tv \left(u_\epsilon (t)\right) \leq \ & \mathcal{L}_\epsilon (t)
    \\
    \label{eq:dipTempoPar}
    \norma{u_\epsilon (t) - u_\epsilon (s)}_{\L1 (\Omega; \reali)}
    \leq \ & \mathcal{L}_\epsilon \left(\max\{t,s\}\right)\modulo{t-s}
  \end{align}
  for $t,s \in I$, where
  \begin{equation}
    \label{eq:30}
    \mathcal{L}_\epsilon (t)
    =
    \left(
      A_1
      +
      A_2 \, t
      +
      A_3 \norma{\grad u_o}_{\L1 (\Omega; \reali^n)}
      +
      \epsilon
      \norma{\Delta u_o}_{\L1 (\Omega; \reali)}
    \right)
    e^{A_4 \, t}
  \end{equation}
  Above, $A_1, A_2, A_3, A_4$ are constants depending on $n$, $\Omega$
  and on norms of $Df$ and $F$, see~\eqref{eq:20}. In particular, they
  are independent of $\epsilon$ and of $u_o$.
\end{lemma}

\section{The Hyperbolic Problem~(\ref{eq:1})}
\label{sec:HP}

In the particular case of homogeneous boundary condition, we study the
convergence of the sequence $(u_\epsilon)$ as $\epsilon$ tends to
$0$. We also prove that the limit function is a solution to
problem~\eqref{eq:1}, with homogeneous boundary condition.

\begin{proposition}
  \label{prop:limiteZero}
  Fix $\delta \in \left]0, 1\right[$. Let
  conditions~{\bf($\boldsymbol{\Omega_{2,\delta}}$)}, {\bf(f)}
  and~{\bf(F)} hold. Assume moreover that $u_o \in \C{2,\delta}
  (\bar{\Omega} ; \reali)$ is such that $u_o (\xi) = 0$ for $\xi
  \in \partial \Omega$, and $u_{b} (t,\xi) = 0$ for $(t,\xi) \in I
  \times \Omega$.

  Then, the family of solutions $u_\epsilon$ to~\eqref{eq:visco} is
  relatively compact in $\L1$. Any cluster point $u_\infty \in \L1 (I
  \times \Omega; \reali)$ of this family is a solution
  to~\eqref{eq:1}, with $u_b = 0$, in the sense of
  Definition~\ref{def:sol}. Moreover, the following estimates hold:
  \begin{eqnarray}
    \nonumber
    \norma{u_\infty (t)}_{\L\infty (\Omega; \reali)}
    & \leq &
    \norma{u_o}_{\L\infty (\Omega; \reali)} \, e^{c_1 \, t}
    +
    \frac{c_2}{c_1} \left(e^{c_1 \, t} - 1 \right),
    \\
    \label{eq:tb_original}
    \tv \left(u_\infty (t)\right)
    & \leq &
    \mathcal{L} (t)
    \\
    \nonumber
    \norma{u_\infty (t) - u_\infty (s)}_{\L1 (\Omega;\reali)}
    & \leq &
    \mathcal{L} \left(\max\left\{t,s\right\}\right) \; \modulo{t-s},
  \end{eqnarray}
  for $t,s \in I$, where
  \begin{equation}
    \label{eq:21}
    \mathcal{L} (t)
    =
    \left(
      A_1
      +
      A_2 \, t
      +
      A_3 \norma{\grad u_o}_{\L1 (\Omega; \reali^n)}
    \right)
    e^{A_4 \, t}
  \end{equation}
  Above, $c_1, c_2, A_1, A_2, A_3, A_4$ are constants depending on
  $n$, $\Omega$ and on norms of $Df$ and $F$, see~\eqref{eq:c1c2}
  and~\eqref{eq:20}, all independent of the initial datum.
\end{proposition}

\noindent Note that $u_\infty \in (\L\infty \cap \BV)
(I\times\Omega;\reali)$, see Remark~\ref{rem:BV}.

\begin{theorem}
  \label{thm:estConComp}
  Fix $\alpha \in \left]0, 1 \right[$. Let
  conditions~{\bf($\boldsymbol{\Omega_{3,\alpha}}$)}, {\bf(f)}
  and~{\bf(F)} hold. Assume moreover that $u_b \in \C{3,\alpha} (I
  \times \partial \Omega; \reali)$ and $u_o \in \C{2,\delta}
  (\bar{\Omega} ; \reali)$, with $\delta \in \left]\alpha, 1 \right[$,
  are such that
  \begin{equation}
    \label{eq:nonchiamarla}
    u_o (\xi) = 0  = u_b (0, \xi)
    \quad \mbox{ for all }
    \xi \in \partial \Omega.
  \end{equation}
  Then, there exists a unique solution $u \in \C{0,1}\left(I; \L1
    (\Omega; \reali)\right)$ to~\eqref{eq:1} in the sense of
  Definition~\ref{def:sol}. Moreover, the following bounds hold:
  \begin{eqnarray}
    \nonumber
    \norma{u (t)}_{\L\infty (\Omega;\reali)}
    & \leq &
    \left(
      \norma{u_o}_{\L\infty (\Omega;\reali)}
      +
      \norma{u_b}_{\L\infty ([0,t]\times\Omega;\reali)}
    \right)
    e^{c_1 t}
    \\
    \label{eq:L_bound}
    & &
    \qquad
    +
    \frac{c_2 + \norma{\partial_t u_b}_{\L\infty ([0,t]\times\Omega;\reali)}}{c_1}
    \left(e^{c_1 t}-1\right)
    \\
    \nonumber
    \tv\left(u (t)\right)
    & \leq &
    \mathcal{C} (\Omega,f,F,t)
    \left(
      \norma{u_b}_{\C{3,\alpha} ([0,t]\times\partial\Omega;\reali)}
      +
      \norma{u_b}_{\C{3,\alpha}([0,t]\times\partial\Omega;\reali)}^2
    \right)
    \\
    \label{eq:tv_bound}
    & &
    \times
    \left(
      1
      +
      t
      +
      \tv (u_o)
    \right)
    \\
    \nonumber
    & &
    \times
    \exp
    \left(
      \mathcal{C} (\Omega,f,F,t)
      (1+\norma{u_b}_{\C{2,\alpha} ([0,t]\times\partial\Omega;\reali)})\,t
    \right)
    \\
    \label{eq:t_Lip}
    \norma{u (t) - u (s)}_{\L1 (\Omega;\reali)}
    & \leq &
    \left(\sup_{\tau \in [s,t]} \tv\left(u (\tau)\right)\right)
    \modulo{t-s}
  \end{eqnarray}
  for $t, \, s \in I$, where $c_1$, $c_2$ and $\mathcal{C}
  (\Omega,f,F,t)$ are independent of the initial and boundary data,
  see~\eqref{eq:c1c2} and~\eqref{eq:finita}.
\end{theorem}

\noindent Above, Remark~\ref{rem:BV} applies and guarantees that $u
\in (\L\infty \cap \BV) (I\times\Omega;\reali)$.

Following~\cite[Theorem~2]{BardosLerouxNedelec}, we
extend~\cite[Theorem~15.1.5]{Serre2} to the case of balance laws with
time and space dependent flow and source.

\begin{theorem}
  \label{thm:wp}
  Let~{\bf($\boldsymbol{\Omega_{2,0}}$)}, {\bf(f)} and~{\bf(F)}
  hold. Set
  \begin{displaymath}
    L_f = \norma{\partial_u f}_{\L{\infty}(\Sigma; \reali^n)}
    \quad \mbox{ and } \quad
    L_F = \norma{\partial_u F}_{\L{\infty}(\Sigma; \reali)} \,.
  \end{displaymath}
  Assume that the initial data $u_o, v_o$ and the boundary data
  $u_{b}, v_{b}$ satisfy~{\bf(C)}.  If $u$ and $v$ are the
  corresponding solutions to~\eqref{eq:1} in the sense of
  Definition~\ref{def:sol}, then, for all $t \in I$, the following
  estimate holds
  \begin{eqnarray*}
    \int_{\Omega} \modulo{u(t,x) - v(t,x)} \d{x}
    & \leq &
    e^{L_F \, t} \int_{\Omega} \modulo{u_o(x) - v_o(x)} \d{x}
    \\
    & &
    +
    L_f
    \int_0^t e^{L_F \, (t- \tau)} \int_{\partial \Omega}
    \modulo{u_{b}(\tau,\xi) - v_{b}(\tau,\xi)} \d{\xi} \d{\tau}.
  \end{eqnarray*}
\end{theorem}

\section{Proofs Related to the Parabolic Problem}
\label{sec:PPP}

\begin{proofof}{Lemma~\ref{lem:EeUpar}}
  To improve the readability, we write $u$ instead of $u_\epsilon$. We
  apply~\cite[Chapter~7, \S~4, Theorem~9]{FriedmanParabolic}. To this
  aim, in the notation of~\cite[\S~4]{FriedmanParabolic}, we verify
  the required assumptions with reference to $Lu = \f (t,x,u,\grad
  u)$, where
  \begin{align*}
    L u = \ & \epsilon \, \Delta u - \partial_t u
    \\
    \f (t,x,u,w) = \ & \div f (t,x,u) + \partial_u f (t,x,u) \cdot w -
    F (t,x,u) \,.
  \end{align*}
  The boundary and initial data $\psi$ in~\cite{FriedmanParabolic}
  corresponds here to the function $\bar u$. The required
  $\C{2,\alpha}$ regularity of $S = I \times \partial\Omega$ is
  ensured by the hypothesis. The parabolicity
  condition~\cite[Chapter~7, \S~2, p.191, (A)]{FriedmanParabolic}
  holds with $H_o = \epsilon$. The condition~\cite[Chapter~7, \S~4,
  p.204, (B')]{FriedmanParabolic} on the coefficients of $L$ is
  immediately satisfied: the only non-zero coefficient is the constant
  $\epsilon$. By hypothesis, the function $\bar u$ is in
  $\C{2,\delta}$, for $\alpha < \delta < 1$. The H\"older continuity
  of $\f$ follows from~\textbf{(f)} and~\textbf{(F)}.
  Concerning~\cite[Chapter~7, \S~2, p.203,
  Formula~(4.10)]{FriedmanParabolic}, it reads:
  \begin{align*}
    u \, \f (t,x,u,0) = \ & u \left(\div f (t,x,u) - F (t,x,u)\right)
    \\
    \leq \ & \modulo{u} \left( \norma{\div f (\cdot, \cdot,
        0)}_{\L\infty (I \times \bar\Omega; \reali)} + \norma{F
        (\cdot, \cdot, 0)}_{\L\infty (I \times \bar\Omega; \reali)}
    \right)
    \\
    & + u^2 \left( \norma{\partial_u \div f}_{\L\infty (I \times
        \bar\Omega \times \reali; \reali)} + \norma{\partial_u
        F}_{\L\infty (I \times \bar\Omega \times \reali; \reali)}
    \right)
    \\
    \leq \ & A_1 u^2 + A_2
  \end{align*}
  for suitable positive $A_1, A_2$, by~\textbf{(f)} and~\textbf{(F)}.

  \noindent Passing to~\cite[Chapter~7, \S~2, p.205,
  Formula~(4.17)]{FriedmanParabolic}
  \begin{align*}
    \modulo{\f (t,x,u,w)} \leq \ & \modulo{\div f (t,x,u)} +
    \norma{\partial_u f (t,x,u)} \norma{w} + \modulo{F (t,x,u)}
    \\
    \leq \ & \left( \norma{\div f (\cdot, \cdot, 0)}_{\L\infty (I
        \times \bar\Omega; \reali)} + \norma{F (\cdot, \cdot,
        0)}_{\L\infty (I \times \bar\Omega; \reali)} \right)
    \\
    & + \left( \norma{\partial_u \div f}_{\L\infty (I \times
        \bar\Omega \times \reali; \reali)} + \norma{\partial_u
        F}_{\L\infty (I \times \bar\Omega \times \reali; \reali)}
    \right) \modulo{u}
    \\
    & + \norma{\partial_u f}_{\L\infty (I\times\bar\Omega \times
      \reali; \reali^n)} \norma{w}
    \\
    \leq \ & A (\modulo{u}) + \mu \, \norma{w}
  \end{align*}
  for a non decreasing function $A$ and a positive scalar $\mu$,
  by~\textbf{(f)} and~\textbf{(F)}.  Lastly, the compatibility
  condition $L\bar u (0,x) = \f (0,x,\bar u, \grad \bar u)$ on
  $\partial \Omega$ holds by~\eqref{eq:smoothData}.

  We can thus apply~\cite[Chapter~7, \S~4,
  Theorem~9]{FriedmanParabolic}, obtaining the existence of a solution
  $u_\epsilon$ to~\eqref{eq:visco} in the class $\C{2,\gamma} (I
  \times \bar \Omega; \reali)$ for $0 < \gamma <
  1$. Moreover,~\cite[Chapter~7, \S~4, Theorem~6]{FriedmanParabolic}
  ensures the uniqueness of the solution. The verification that the
  necessary assumptions are satisfied is here immediate.
\end{proofof}

\begin{proofof}{Lemma~\ref{lem:stimaLinf}}
  For the sake of readability, we write $u$ instead of $u_\epsilon$.
  We use~\cite[Chapter~1, \S~2,
  Theorem~2.9]{LadysenkajaSolonnikovUralceva}, which refers to
  \begin{displaymath}
    \partial_t u
    -
    \sum_{i,j=1}^n
    a_{ij} (t,x,u,\grad u) \, \partial^2_{ij} u
    +
    a (t,x,u,\grad u)
    =
    0
  \end{displaymath}
  where, in the present case,
  \begin{displaymath}
    \begin{array}{@{}rcl}
      a_{ij} (t, x, u, p)
      & = &
      \epsilon \, \delta_{ij}
      \quad \mbox{ for } i,j = 1, \ldots, n \,,
      \\
      a (t, x, u, p)
      & = &
      \div f (t,x,u)
      +
      \partial_u f (t,x,u) \cdot p
      -
      F (t,x,u) \,.
    \end{array}
  \end{displaymath}
  Condition~{\bf($\boldsymbol{\Omega_{2,\alpha}}$)} ensures the
  necessary regularity of the domain.  By~\textbf{(f)}
  and~\textbf{(F)}, the regularity requirements on $a_{ij}$ and $a$
  are met. Moreover,
  \begin{displaymath}
    \begin{array}{@{}lrcl@{}}
      \mbox{\cite[Formula~(2.29)]{LadysenkajaSolonnikovUralceva}: }
      &
      \displaystyle
      \sum_{i,j=1}^n a_{ij} (t,x,u,0) \xi_i \xi_j
      & = &
      \epsilon \, \norma{\xi}^2  \, \geq \, 0
      \\[12pt]
      \mbox{\cite[Formula~(2.32)] {LadysenkajaSolonnikovUralceva}: }
      & u \, a (t,x,u,0)
      & = &
      u \, \div f (t,x,u) - u \, F (t,x,u)
      \\
      & &\geq &
      - \Phi (\modulo{u}) \, \modulo{u}
    \end{array}
  \end{displaymath}
  where $b_2=0$ in~\cite[Chapter~1, \S~2,
  Formula~(2.32)]{LadysenkajaSolonnikovUralceva},
  \begin{equation}
    \label{eq:c1c2}
    \Phi (\modulo{u}) =  c_1 \, \modulo{u} + c_2
    \quad \mbox{ and } \quad
    \begin{array}{rcl}
      c_1 & = &
      1 + \norma{\partial_u \div f}_{\L\infty
        (I\times\bar\Omega\times\reali; \reali)} + \norma{\partial_u
        F}_{\L\infty (I\times\bar\Omega\times\reali; \reali)} \,,
      \\
      c_2
      & = & \norma{\div f(\cdot, \cdot, 0)}_{\L\infty
        (I\times\bar\Omega; \reali)} + \norma{F (\cdot, \cdot,
        0)}_{\L\infty (I\times\bar\Omega; \reali)} \,.
    \end{array}
  \end{equation}
  Note that $\Phi$ is nondecreasing, positive and
  condition~\cite[Chapter~1, \S~2,
  Formula~(2.32)]{LadysenkajaSolonnikovUralceva} holds.  Hence,
  \cite[Chapter~1, \S~2, Theorem~2.9]{LadysenkajaSolonnikovUralceva}
  applies and the solution $u$ to~\eqref{eq:visco}
  satisfies~\cite[Formula~(2.34)]{LadysenkajaSolonnikovUralceva} with
  $\phi (\xi) = (c_2/c_1) \left( \xi^{c_1} - 1 \right)$, so that
  \begin{displaymath}
    \norma{u}_{\L\infty ([0,t] \times \Omega; \reali)} \leq \
    \left( \norma{u_o}_{\L\infty (\Omega; \reali)} +
      \norma{u_b}_{\L\infty ([0,t] \times \partial\Omega;\reali)}
    \right) e^{c_1\, t} + \frac{c_2}{c_1} \left(e^{c_1\, t} - 1\right) \,.
  \end{displaymath}
  completing the proof.
\end{proofof}

\begin{proofof}{Lemma~\ref{lem:tvParZero}}
  First, define $w_\epsilon \in \C{2,\delta} (\bar \Omega; \reali)$ as
  solution to the elliptic problem
  \begin{displaymath}
    \left\{
      \begin{array}{rcl@{\qquad}r@{\;}c@{\;}l}
        \Delta w_\epsilon
        & = &
        - \Delta u_o
        + \frac{1}{\epsilon} \, \div f (0, x, 0)
        +  \frac{1}{\epsilon} \, \partial_u f (0,x, 0) \cdot \grad u_o (x)
        -  \frac{1}{\epsilon} \, F (0, x, 0)
        &
        x & \in & \Omega \,,
        \\
        w_\epsilon (\xi)
        & = &
        0
        & \xi & \in & \partial\Omega \,.
      \end{array}
    \right.
  \end{displaymath}
  The elliptic problem above admits a unique solution $w_\epsilon \in
  \C{2,\delta} (\bar\Omega; \reali)$ thanks to~\cite[Chapter~3, \S~1,
  Theorem~1.3]{LadysenkajaUralcevaElliptic}. Indeed, with reference to
  the equation $L w_\epsilon(x) = \f(x)$ where
  \begin{align*}
    L w_\epsilon = \ & \sum_{i,j=1}^n a_{i,j} (x) \, \partial_{ij}^2
    w_\epsilon + \sum_{i=1}^n a_i (x) \, \partial_i w_\epsilon + a(x)
    \, w_\epsilon = \Delta w_\epsilon
    \\
    \f(x) = \ & - \Delta u_o + \frac{1}{\epsilon} \, \div f (0, x, 0)
    + \frac{1}{\epsilon} \, \partial_u f (0,x, 0) \cdot \grad u_o (x)
    - \frac{1}{\epsilon} \, F (0, x, 0),
  \end{align*}
  the hypotheses of~\cite[Chapter~3, \S~1,
  Theorem~1.3]{LadysenkajaUralcevaElliptic} are all satisfied: the
  coefficients of $L$ belong to $\C{\delta} (\bar\Omega; \reali)$ and
  satisfy the ellipticity condition; we have $a(x) = 0$; the boundary
  $\partial \Omega$ is of class $\C{2,\delta}$ by hypothesis; the
  function $\f$ is in $\C{\delta} (\bar\Omega; \reali)$ thanks to the
  hypothesis on $u_o$, to~{\bf{(f)}} and~{\bf{(F)}}; the homogeneous
  boundary condition implies that, in the notation of~\cite[Chapter~3,
  \S~1]{LadysenkajaUralcevaElliptic}, $\phi= 0$, which is clearly in
  $\C{2,\delta} (\partial \Omega; \reali)$.

  Define now $\bar u_\epsilon (t,x) = u_o (x) + w_\epsilon (x)$ for
  every $(t,x) \in I \times \Omega$: this function $\bar u_\epsilon$
  belongs to $\C{2,\delta} (I \times \bar \Omega; \reali)$ and it
  satisfies~\eqref{eq:smoothData} and~\eqref{eq:compInDat}, with $u_b
  =0$. Since $\partial \Omega$ is of class $\C{2,\delta}$, it is also
  of class $\C{2,\alpha}$ for any $\alpha \in \left]0,
    \delta\right[$. Hence, Lemma~\ref{lem:EeUpar} yields that there
  exists a unique solution $u_\epsilon \in \C{2,\gamma} (I \times
  \bar\Omega; \reali)$, for a $\gamma \in \left]0,1\right[$,
  to~\eqref{eq:visco} with $u_b = 0$.

  Following~\cite{BardosLerouxNedelec, DafermosBook, GagneuxBook}, for
  $\eta>0$ introduce the functions
  \begin{equation}
    \label{eq:sigma}
    \sigma_\eta (z) = \left\{
      \begin{array}{lr@{\,}c@{\,}l}
        z - \eta/2 & z & > & \eta
        \\
        z^2/(2\eta) & z & \in & [-\eta, \eta]
        \\
        -z + 3\eta/2 & z & < & -\eta \,,
      \end{array}
    \right.
    \qquad\qquad
    \sigma'_\eta (z) = \left\{
      \begin{array}{lr@{\,}c@{\,}l}
        1 & z & > & \eta
        \\
        z/\eta & z & \in & [-\eta, \eta]
        \\
        -1 & z & < & -\eta .
      \end{array}
    \right.
  \end{equation}
  Note that $u_\epsilon$ is of class $\C2$, hence,
  by~\eqref{eq:visco}, $\Delta u_\epsilon$ is of class $\C1$ and we
  can differentiate with respect to $t$ the equation
  in~\eqref{eq:visco}:
  \begin{equation}
    \label{eq:derT}
    \begin{array}{cl}
      &
      \partial^2_{tt} u_\epsilon (t,x)
      +
      \Div \left(
        \partial_t f\left(t,x,u_\epsilon (t,x)\right)
        +
        \partial_u f \left(t,x,u_\epsilon (t,x)\right) \partial_t u_\epsilon (t,x)
      \right)
      \\
      = &
      \partial_t F \left(t,x,u_\epsilon (t,x)\right)
      +
      \partial_u F\left(t,x,u_\epsilon (t,x)\right) \,
      \partial_t u_\epsilon (t,x)
      +
      \partial_t \Delta u_\epsilon (t,x).
    \end{array}
  \end{equation}
  Multiply by $\sigma'_\eta \left(\partial_t u_\epsilon (t,x)\right)$
  and integrate over $\Omega$ each term above to obtain
  \begin{equation}
    \label{eq:derT1}
    \!
    \int_\Omega
    \partial^2_{tt} u_\epsilon (t,x)
    \sigma'_\eta\left(\partial_t u_\epsilon (t,x)\right)
    \d{x}
    =
    \frac{\d{~}}{\d{t}}
    \int_\Omega
    \int_0^{\partial_tu_\epsilon (t,x)}
    \sigma'_\eta (v)
    \d{v}
    \d{x}
    \stackrel{\eta\to 0}{=}
    \frac{\d{~}}{\d{t}}
    \int_\Omega
    \modulo{\partial_t u_\epsilon (t,x)}
    \d{x}.
  \end{equation}
  Concerning the second term on the first line of~\eqref{eq:derT}, we
  have
  \begin{align}
    \nonumber & \int_\Omega \Div
    \partial_t f\left(t,x,u_\epsilon (t,x)\right) \,
    \sigma'_\eta\left(\partial_t u_\epsilon (t,x)\right) \d{x}
    \\
    \nonumber = \ & \int_\Omega \left[ \div
      \partial_t f\left(t,x,u_\epsilon (t,x)\right) \,
      \sigma'_\eta\left(\partial_t u_\epsilon (t,x)\right) +
      \partial_u \partial_t f\left(t,x,u_\epsilon (t,x)\right) \,
      \grad u_\epsilon (t,x) \, \sigma'_\eta\left(\partial_t
        u_\epsilon (t,x)\right) \right] \d{x}
    \\
    \label{eq:derT2}
    \geq \ & - \mathcal{L}^n (\Omega) \norma{\div \partial_t
      f}_{\L\infty ([0,t]\times \Omega \times \mathcal{U}(t);\reali)}
    - \norma{\grad u_\epsilon (t)}_{\L1 (\Omega;\reali^n)}
    \norma{\partial_u \partial_t f}_{\L\infty ([0,t]\times \Omega
      \times \mathcal{U}(t);\reali^n)}
  \end{align}
  and
  \begin{align}
    \nonumber & \int_\Omega \Div \left(
      \partial_u f \left(t,x,u_\epsilon (t,x)\right) \,
      \partial_t u_\epsilon (t,x) \right) \sigma'_\eta\left(\partial_t
      u_\epsilon (t,x)\right) \d{x}
    \\
    \nonumber = \ & \int_{\partial\Omega} \sigma'_\eta\left(\partial_t
      u_\epsilon (t,\xi)\right) \;
    \partial_u f\left(t, \xi, u_\epsilon (t,\xi)\right) \cdot \nu
    (\xi) \,
    \partial_t u_\epsilon (t,\xi) \d\xi
    \\
    \nonumber & - \int_\Omega
    \partial_t u_\epsilon (t,x) \;
    \partial_u f\left(t, x, u_\epsilon (t,x)\right) \cdot
    \grad \partial_t u_\epsilon (t,x) \; \sigma''_\eta\left(\partial_t
      u_\epsilon (t,x)\right) \d{x}
    \\
    \nonumber \geq \ & - \norma{\partial_u f}_{\L\infty ([0,t] \times
      \Omega \times \mathcal{U}(t);\reali^n)} \int_{\partial\Omega}
    \modulo{\partial_t u_\epsilon (t,\xi)} \d\xi
    \\
    \nonumber & - \norma{\partial_u f}_{\L\infty ([0,t] \times \Omega
      \times \mathcal{U}(t);\reali^n)} \int_{\modulo{\partial_t
        u_\epsilon} \leq \eta} \modulo{\partial_t u_\epsilon (t,x)} \,
    \grad \left(\sigma'_\eta\left(\partial_t u_\epsilon
        (t,x)\right)\right) \, \d{x}
    \\
    \nonumber = \ & - \norma{\partial_u f}_{\L\infty ([0,t] \times
      \Omega \times \mathcal{U}(t);\reali^n)} \int_{\partial\Omega}
    \modulo{\partial_t u_b (t,\xi)} \d\xi
    \\
    \nonumber & - \norma{\partial_u f}_{\L\infty ([0,t] \times \Omega
      \times \mathcal{U}(t);\reali^n)} \int_{\modulo{\partial_t
        u_\epsilon} \leq \eta} \grad\partial_t u_\epsilon (t,x) \,
    \d{x}
    \\
    \nonumber \stackrel{\eta \to 0}{=} \ & - \norma{\partial_u
      f}_{\L\infty ([0,t] \times \Omega \times
      \mathcal{U}(t);\reali^n)} \int_{\partial\Omega}
    \modulo{\partial_t u_b (t,\xi)} \d\xi
    \\
    \nonumber \stackrel{u_b= 0}{=} \ & 0,
  \end{align}
  where, in the last limit, we used~\cite[Lemma~2]{BardosLerouxNedelec}.

  To estimate the first two terms on the second line
  of~\eqref{eq:derT}, we compute:
  \begin{equation}
    \label{eq:derT3}
    \begin{aligned}
      & \displaystyle \int_\Omega \left(
        \partial_t F \left(t,x,u_\epsilon (t,x)\right) +
        \partial_u F\left(t,x,u_\epsilon (t,x)\right) \,
        \partial_t u_\epsilon (t,x) \right)
      \sigma'_\eta\left(\partial_t u_\epsilon (t,x)\right) \d{x}
      \\
      \leq \ & \displaystyle \mathcal{L}^n (\Omega) \,
      \norma{\partial_ t F}_{\L\infty ([0,t]\times \Omega \times
        \mathcal{U}(t);\reali)} + \norma{\partial_u F}_{\L\infty
        ([0,t]\times \Omega \times \mathcal{U}(t);\reali)} \,
      \norma{\partial_t u_\epsilon (t)}_{\L1 (\Omega;\reali^n)} \,.
    \end{aligned}
  \end{equation}
  To bound the last term on the second line of~\eqref{eq:derT}, we
  proceed as follows:
  \begin{align*}
    & \epsilon \int_\Omega
    \partial_t \Delta u_\epsilon (t,x) \, \sigma'_\eta\left(\partial_t
      u_\epsilon (t,x)\right) \d{x}
    \\
    = \ & \epsilon \int_{\partial\Omega}
    \partial_t \grad u_\epsilon (t,\xi) \cdot \nu (\xi) \;
    \sigma'_\eta\left(\partial_t u_\epsilon (t,\xi)\right) \d{\xi} -
    \int_\Omega \norma{\partial_t \grad u_\epsilon (t,x)}^2 \;
    \sigma''_\eta\left(\partial_t u_\epsilon (t,x)\right) \d{x}
    \\
    \leq \ & \epsilon \int_{\partial\Omega}
    \partial_t \grad u_\epsilon (t,\xi) \cdot \nu (\xi) \;
    \sigma'_\eta\left(\partial_t u_\epsilon (t,\xi)\right) \d{\xi}
    \\
    = \ & \epsilon \int_{\partial\Omega} \grad \sigma_\eta
    \left(\partial_t u_\epsilon (t,\xi)\right) \cdot \nu (\xi) \,
    \d{\xi}
    \\
    \stackrel{u_b= 0}{=} \ & 0 \,.
  \end{align*}
  Integrate~\eqref{eq:derT} in time over $[0,t]$,
  using~\eqref{eq:derT1}, \eqref{eq:derT2} and~\eqref{eq:derT3} to
  obtain
  \begin{align}
    \nonumber & \norma{\partial_t u_\epsilon (t)}_{\L1
      (\Omega;\reali)}
    \\
    \nonumber \leq \ & \norma{\partial_t u_\epsilon (0)}_{\L1
      (\Omega;\reali)}
    \\
    \nonumber & + \mathcal{L}^n (\Omega) \, t \,
    \norma{\div \partial_t f}_{\L\infty ([0,t]\times \Omega \times
      \mathcal{U}(t);\reali)}
    \\
    \nonumber & + \int_0^t \norma{\partial_u \partial_t f}_{\L\infty
      ([0,\tau]\times \Omega \times \mathcal{U}(t);\reali^n)} \,
    \norma{\grad u_\epsilon (\tau)}_{\L1 (\Omega;\reali^n)} \, \d\tau
    \\
    \nonumber & + \mathcal{L}^n (\Omega) \, t \, \norma{\partial_t
      F}_{\L\infty ([0,t]\times \Omega \times \mathcal{U}(t);\reali)}
    \\
    \nonumber & + \int_0^t \norma{\partial_u F}_{\L\infty
      ([0,\tau]\times \Omega \times \mathcal{U}(t);\reali)} \,
    \norma{\partial_t u_\epsilon (\tau)}_{\L1 (\Omega;\reali)} \,
    \d\tau
    \\
    \nonumber \leq \ & \norma{\partial_t u_\epsilon (0)}_{\L1
      (\Omega;\reali)} + \mathcal{L}^n (\Omega) \, t \left(
      \norma{\div \partial_t f}_{\L\infty ([0,t]\times \Omega \times
        \mathcal{U}(t);\reali)} + \norma{\partial_t F}_{\L\infty
        ([0,t]\times \Omega \times \mathcal{U}(t);\reali)} \right)
    \\
    \label{eq:18}
    & + \left( \norma{\partial_u \partial_t f}_{\L\infty ([0,t]\times
        \Omega \times \mathcal{U}(t);\reali^n)} + \norma{\partial_u
        F}_{\L\infty ([0,t]\times \Omega \times
        \mathcal{U}(t);\reali)} \right)
    \\
    \nonumber & \qquad \times \int_0^t \left( \norma{\partial_t
        u_\epsilon (\tau)}_{\L1 (\Omega;\reali)} + \norma{\grad
        u_\epsilon (\tau)}_{\L1 (\Omega;\reali^n)} \right) \d{\tau} .
  \end{align}
  Using the parabolic equation~\eqref{eq:visco}, we can estimate the
  first term in the right hand side above as follows:
  \begin{align}
    \nonumber \norma{\partial_t u_\epsilon (0)}_{\L1 (\Omega;\reali)}
    \leq \ & \norma{\partial_u f}_{\L\infty ([0,t]\times\Omega\times
      \mathcal{U} (t);\reali^n)} \, \norma{\grad u_o}_{\L1
      (\Omega;\reali^n)}
    \\
    \label{eq:19}
    & + \mathcal{L}^n (\Omega) \left( \norma{\div f}_{\L\infty
        ([0,t]\times\Omega\times\mathcal{U}(t);\reali)} +
      \norma{F}_{\L\infty
        ([0,t]\times\Omega\times\mathcal{U}(t);\reali)} \right)
    \\
    \nonumber & + \epsilon \norma{\Delta u_o}_{\L1 (\Omega;\reali)}
    \,.
  \end{align}

  \smallskip

  As noted above, $\Delta u_\epsilon$ is of class $\C1$ and, for $j=1,
  \ldots, n$, we can differentiate the equation in~\eqref{eq:visco}
  with respect to $x_j$ to obtain
  \begin{equation}
    \label{eq:12}
    \partial_t \partial_j u_\epsilon (t,x)
    +
    \Div \frac{\d{~}}{\d{x_j}} f\left(t,x,u_\epsilon (t,x)\right)
    =
    \frac{\d{~}}{\d{x_j}} F\left(t,x,u_\epsilon (t,x)\right)
    +
    \Delta \partial_j u_\epsilon (t,x) \,.
  \end{equation}
  Multiply by $\sigma'_\eta (\partial_j u_\epsilon)$ and integrate
  each term in~\eqref{eq:12} over $\Omega$:
  \begin{equation}
    \label{eq:derX1}
    \int_\Omega
    \partial_t \partial_j u_\epsilon (t,x) \;
    \sigma'_\eta\left(\partial_j u_\epsilon (t,x)\right)
    \d{x}
    =
    \int_\Omega \frac{\d~}{\d{t}}
    \sigma_\eta\left(\partial_j u_\epsilon (t,x)\right) \d{x}
    =
    \frac{\d~}{\d{t}}
    \int_\Omega
    \sigma_\eta\left(\partial_j u_\epsilon (t,x)\right)
    \d{x} \,.
  \end{equation}

  To estimate the second term in the left hand side of~\eqref{eq:12},
  we follow~\cite[Chapter~6, Proof of Lemma~6.9.5]{DafermosBook}, use
  the equality $\grad \partial_j u_\epsilon (t,x) \, \sigma'_\eta
  \left(\partial_j u_\epsilon (t,x)\right) = \grad \sigma'_\eta
  \left(\partial_j u_\epsilon (t,x)\right)$ and the Divergence
  Theorem:
  \begin{align}
    \nonumber & \int_\Omega \Div \frac{\d{~}}{\d{x_j}}
    f\left(t,x,u_\epsilon (t,x)\right) \sigma'_\eta\left(\partial_j
      u_\epsilon (t,x)\right) \d{x}
    \\
    \nonumber = \ & \int_\Omega \Div \partial_j f\left(t,x,u_\epsilon
      (t,x)\right) \sigma'_\eta\left(\partial_j u_\epsilon
      (t,x)\right) \d{x}
    \\
    \nonumber &+ \int_\Omega \Div \left(
      \partial_u f\left(t,x,u_\epsilon (t,x)\right) \,
      \partial_j u_\epsilon (t,x) \right) \sigma'_\eta\left(\partial_j
      u_\epsilon (t,x)\right) \d{x}
    \\
    \nonumber =\ & \int_\Omega \div \partial_j f\left(t,x,u_\epsilon
      (t,x)\right) \sigma'_\eta \left(\partial_j u_\epsilon
      (t,x)\right) \d{x}
    \\
    \nonumber & + \int_\Omega
    \partial_u \partial_j f\left(t,x,u_\epsilon (t,x)\right) \,
    \partial_j u_\epsilon (t,x) \, \sigma'_\eta \left(\partial_j
      u_\epsilon (t,x)\right) \d{x}
    \\
    \nonumber & + \int_\Omega \Div \left( \partial_u
      f\left(t,x,u_\epsilon (t,x)\right)\right)
    \partial_j u_\epsilon (t,x) \, \sigma'_\eta \left(\partial_j
      u_\epsilon (t,x)\right) \d{x}
    \\
    \nonumber & + \int_\Omega \partial_u f\left(t,x,u_\epsilon
      (t,x)\right) \cdot \grad \sigma_\eta \left(\partial_j u_\epsilon
      (t,x) \right) \d{x}
    \\
    \nonumber \geq \ & - \mathcal{L}^n (\Omega) \, \norma{\grad \div
      f}_{\L\infty ([0,t] \times \Omega \times
      \mathcal{U}(t);\reali^n)}
    \\
    \nonumber & - \norma{\grad \partial_u f}_{\L\infty ([0,t] \times
      \Omega \times \mathcal{U}(t);\reali^{n\times n})} \, \int_\Omega
    \modulo{
      \partial_j u_\epsilon (t,x) \, \sigma'_\eta \left(\partial_j
        u_\epsilon (t,x)\right) } \d{x}
    \\
    \nonumber & + \int_\Omega \Div \left( \partial_u
      f\left(t,x,u_\epsilon (t,x)\right)\right)
    \partial_j u_\epsilon (t,x) \, \sigma'_\eta \left(\partial_j
      u_\epsilon (t,x)\right) \d{x}
    \\
    \nonumber & - \int_\Omega \Div \left( \partial_u
      f\left(t,x,u_\epsilon (t,x)\right)\right) \, \sigma_\eta
    \left(\partial_j u_\epsilon (t,x)\right) \d{x}
    \\
    \nonumber & + \int_{\partial \Omega} \sigma_\eta\left(\partial_j
      u_\epsilon (t,\xi)\right) \,
    \partial_u f\left(t,\xi,u_\epsilon (t,\xi)\right) \cdot \nu (\xi)
    \d\xi
    \\
    \label{eq:derX2}
    \geq \ & - \mathcal{L}^n (\Omega) \, \norma{\grad \div
      f}_{\L\infty ([0,t] \times \Omega \times
      \mathcal{U}(t);\reali^n)}
    \\
    \label{eq:derX3}
    & - \norma{\grad \partial_u f}_{\L\infty ([0,t] \times \Omega
      \times \mathcal{U} (t);\reali^{n\times n})} \int_\Omega \modulo{
      \partial_j u_\epsilon (t,x)} \d{x}
    \\
    \label{eq:derX4}
    & + \int_\Omega \Div \left[ \partial_u f\left(t,x,u_\epsilon
        (t,x)\right)\right] \!  \left[
      \partial_j u_\epsilon (t,x) \sigma'_\eta \left(\partial_j
        u_\epsilon (t,x)\right) - \sigma_\eta \left(\partial_j
        u_\epsilon (t,x)\right) \right] \!  \d{x}
    \\
    \label{eq:derX5}
    & + \int_{\partial \Omega} \sigma_\eta\left(\partial_j u_\epsilon
      (t,\xi)\right) \,
    \partial_u f\left(t,\xi,u_\epsilon (t,\xi)\right) \cdot \nu (\xi)
    \d\xi.
  \end{align}
  For later use, note that, for $\xi \in \partial\Omega$,
  \eqref{eq:visco} is the equality
  \begin{displaymath}
    \partial_u f \left(t, \xi, u_\epsilon (t,\xi)\right) \cdot \nu (\xi) \;
    \partial_\nu u_\epsilon (t,\xi)
    =
    \epsilon \, \Delta u_\epsilon (t,\xi)
    + F\left(t, \xi, u_\epsilon (t,\xi)\right)
    - \div f\left(t, \xi, u_\epsilon (t,\xi)\right).
  \end{displaymath}
  Hence, thanks also to the fact that
  \begin{equation}
    \label{eq:13}
    \partial_\nu u_\epsilon \; \nu_j = \partial_j u_\epsilon \,,
  \end{equation}
  we can now elaborate~\eqref{eq:derX5} as follows:
  \begin{align}
    \nonumber & \int_{\partial \Omega} \sigma_\eta\left(\partial_j
      u_\epsilon (t,\xi)\right) \,
    \partial_u f\left(t,\xi,u_\epsilon (t,\xi)\right) \cdot \nu (\xi)
    \d\xi
    \\
    \label{eq:derX5bis}
    = \ & \int_{\partial\Omega} \frac{\sigma_\eta\left(\partial_j
        u_\epsilon (t,\xi)\right)} {\partial_j u_\epsilon (t,\xi)}
    \left( \epsilon \Delta u_\epsilon (t,\xi) + F\left(t, \xi,
        u_\epsilon (t,\xi)\right)- \div f\left(t, \xi, u_\epsilon
        (t,\xi)\right) \right)\nu_j (\xi) \d\xi .
  \end{align}
  Here we used the fact that $\sigma_\eta(z) = o (z)$ for $z\to 0$, so
  that the map $z \to \sigma_\eta (z) / z$ is well defined also at
  $z=0$.  Pass now to the first term in the right hand side
  of~\eqref{eq:12}:
  \begin{align}
    \nonumber & \int_\Omega \frac{\d{~}}{\d{x_j}} F (t,x,u_\epsilon)
    \sigma'_\eta \left(\partial_j u_\epsilon(t,x)\right) \d{x}
    \\
    \nonumber = \ & \int_\Omega \left(
      \partial_j F\left(t,x,u_\epsilon (t,x)\right) +
      \partial_u F\left(t,x,u_\epsilon (t,x)\right) \;
      \partial_j u_\epsilon (t,x) \right) \sigma'_\eta
    \left(\partial_j u_\epsilon(t,x)\right) \d{x}
    \\
    \nonumber = \ & \int_\Omega
    \partial_j F\left(t,x,u_\epsilon (t,x)\right) \, \sigma'_\eta
    \left(\partial_j u_\epsilon(t,x)\right) \d{x}
    \\
    \nonumber \ & + \int_\Omega
    \partial_u F\left(t,x,u_\epsilon (t,x)\right) \;
    \partial_j u_\epsilon (t,x) \, \sigma'_\eta \left(\partial_j
      u_\epsilon(t,x)\right) \d{x}
    \\
    \label{eq:derX6}
    \leq \ & \mathcal{L}^n (\Omega) \norma{\grad F}_{\L\infty
      ([0,t]\times \Omega \times \mathcal{U}(t);\reali^n)} +
    \norma{\partial_u F}_{\L\infty ([0,t]\times \Omega \times
      \mathcal{U}(t);\reali)} \int_\Omega \modulo{\partial_j
      u_\epsilon (t,x)} \d{x} \,,
  \end{align}
  while the last term on the right hand side of~\eqref{eq:12} gives
  \begin{align}
    \nonumber & \epsilon \int_\Omega \Delta \partial_j u_\epsilon
    (t,x) \, \sigma'_\eta\left(\partial_j u_\epsilon (t,x)\right)
    \d{x}
    \\
    \nonumber =\ & \epsilon \int_\Omega \div \grad \partial_j
    u_\epsilon (t,x) \, \sigma'_\eta\left(\partial_j u_\epsilon
      (t,x)\right) \d{x}
    \\
    \nonumber = \ & \epsilon \int_{\partial\Omega}
    \sigma'_\eta\left(\partial_j u_\epsilon (t,\xi)\right) \,
    \grad \partial_j u_\epsilon (t,\xi) \cdot \nu (\xi) \d\xi -
    \epsilon \int_\Omega \grad\partial_j u_\epsilon (t,x) \cdot \grad
    \sigma'_\eta (t,x) \d{x}
    \\
    \nonumber = \ & \epsilon \int_{\partial\Omega}
    \sigma'_\eta\left(\partial_j u_\epsilon (t,\xi)\right) \,
    \grad \partial_j u_\epsilon (t,\xi) \cdot \nu (\xi) \d\xi -
    \epsilon \int_\Omega \sigma''_\eta (t,x) \, \norma{\grad\partial_j
      u_\epsilon (t,x)} \d{x}
    \\
    \nonumber \leq \ & \epsilon \int_{\partial\Omega}
    \sigma'_\eta\left(\partial_j u_\epsilon (t,\xi)\right) \,
    \grad \partial_j u_\epsilon (t,\xi) \cdot \nu (\xi) \d\xi
    \\
    \label{eq:derX8}
    \leq \ & \epsilon \int_{\partial\Omega}
    \sigma'_\eta\left(\partial_j u_\epsilon (t,\xi)\right) \,
    \partial_\nu \left(\partial_j u_\epsilon (t,\xi)\right) \d\xi \,.
  \end{align}
  Integrate~\eqref{eq:12} in time over $[0,t]$,
  using~\eqref{eq:derX1}, \eqref{eq:derX2}--\eqref{eq:derX4},
  \eqref{eq:derX5bis}, \eqref{eq:derX6} and~\eqref{eq:derX8} to obtain
  \begin{align}
    \label{eq:intEta}
    & \int_\Omega \sigma_\eta \left(\partial_j u_\epsilon (t,x)\right)
    \d{x}
    \\
    \nonumber \leq \ & \int_\Omega \sigma_\eta \left(\partial_j
      u_\epsilon (0,x)\right) \d{x}
    \\
    \nonumber & + \mathcal{L}^n (\Omega) \, t \, \norma{\grad \div
      f}_{\L\infty ([0,t] \times \Omega \times \mathcal{U}
      (t);\reali^n)}
    \\
    \nonumber & + \int_0^t \norma{\grad \partial_u f}_{\L\infty ([0,t]
      \times \Omega \times \mathcal{U} (t);\reali^{n\times n})} \,
    \norma{\partial_j u_\epsilon (\tau)}_{\L1 (\Omega;\reali)}
    \d{\tau}
    \\
    \nonumber & - \int_0^t \int_\Omega \Div \left[ \partial_u
      f\left(\tau,x,u_\epsilon (\tau,x)\right)\right] \!  \left[
      \partial_j u_\epsilon (\tau,x) \sigma'_\eta \left(\partial_j
        u_\epsilon (\tau,x)\right) - \sigma_\eta \left(\partial_j
        u_\epsilon (\tau,x)\right) \right] \!  \d{x}\d\tau
    \\
    \nonumber & - \int_0^t \int_{\partial\Omega}
    \frac{\sigma_\eta\left(\partial_j u_\epsilon (\tau,\xi)\right)}
    {\partial_j u_\epsilon (\tau,\xi)} \left( \epsilon \Delta
      u_\epsilon (\tau,\xi) + F\left(\tau, \xi, u_\epsilon
        (\tau,\xi)\right) - \div f\left(\tau, \xi, u_\epsilon
        (\tau,\xi)\right) \right)\nu_j (\xi) \d\xi \d\tau
    \\
    \nonumber & + \mathcal{L}^n (\Omega) \, t \, \norma{\grad
      F}_{\L\infty ([0,t]\times\Omega \times \mathcal{U}(t);\reali^n)}
    \\
    \nonumber & + \int_0^t \norma{\partial_u F}_{\L\infty
      ([0,t]\times\Omega \times \mathcal{U}(t);\reali)} \,
    \norma{\partial_j u_\epsilon (\tau)}_{\L1 (\Omega;\reali)} \d\tau
    \\
    \nonumber & + \int_0^t \int_{\partial\Omega} \epsilon \,
    \sigma'_\eta\left(\partial_j u_\epsilon (\tau,\xi)\right) \,
    \partial_\nu \left(\partial_j u_\epsilon (\tau,\xi)\right) \d\xi
    \d\tau
    \\
    \nonumber \leq \ & \int_\Omega \sigma_\eta \left(\partial_j
      u_\epsilon (0,x)\right) \d{x}
    \\
    \nonumber & + \mathcal{L}^n (\Omega) \, t \left( \norma{\grad \div
        f}_{\L\infty ([0,t] \times \Omega \times \mathcal{U}
        (t);\reali^n)} + \norma{\grad F}_{\L\infty ([0,t]\times\Omega
        \times \mathcal{U}(t);\reali^n)} \right)
    \\
    \nonumber & + \left( \norma{\grad \partial_u f}_{\L\infty ([0,t]
        \times \Omega \times \mathcal{U} (t);\reali^{n\times n})} +
      \norma{\partial_u F}_{\L\infty ([0,t]\times\Omega \times
        \mathcal{U}(t);\reali)} \right) \int_0^t \norma{\partial_j
      u_\epsilon (\tau)}_{\L1 (\Omega;\reali)} \d\tau
    \\
    \label{eq:u1}
    & - \int_0^t \int_\Omega \Div \left[
      \partial_u f\left(\tau,x,u_\epsilon (\tau,x)\right)\right] \!
    \left[
      \partial_j u_\epsilon (\tau,x) \sigma'_\eta \left(\partial_j
        u_\epsilon (\tau,x)\right) - \sigma_\eta \left(\partial_j
        u_\epsilon (\tau,x)\right) \right] \!  \d{x}\d\tau
    \\
    \label{eq:u2}
    & + \epsilon \int_0^t \int_{\partial\Omega} \left(
      \sigma'_\eta\left(\partial_j u_\epsilon (\tau,\xi)\right) \,
      \partial_\nu \left(\partial_j u_\epsilon (\tau,\xi)\right) -
      \frac{\sigma_\eta\left(\partial_j u_\epsilon (\tau,\xi)\right)}
      {\partial_j u_\epsilon (\tau,\xi)} \; \Delta u_\epsilon
      (\tau,\xi) \, \nu_j (\xi) \right) \d\xi \d\tau
    \\
    \label{eq:u3}
    & - \int_0^t \int_{\partial\Omega}
    \frac{\sigma_\eta\left(\partial_j u_\epsilon (\tau,\xi)\right)}
    {\partial_j u_\epsilon (\tau,\xi)} \left( F\left(\tau, \xi,
        0\right) - \div f\left(\tau, \xi, 0\right) \right) \nu_j (\xi)
    \d\xi \d\tau \,.
  \end{align}
  To compute the limit $\eta \to 0$, consider first the latter three
  terms above separately:
  \begin{equation}
    \label{eq:u4}
    \lim_{\eta\to 0} \left[\eqref{eq:u1}\right] = 0
  \end{equation}
  Concerning~\eqref{eq:u2}, following~\cite[Proof of
  Lemma~6.9.5]{DafermosBook}, it is useful to recall the following
  relations, based on~\eqref{eq:13}:
  \begin{equation}
    \label{eq:14}
    \partial_\nu (\partial_j u_\epsilon)
    =
    \partial^2_{\nu\nu} u_\epsilon \; \nu_j + \O \partial_\nu u_\epsilon
    \quad \mbox{ and } \quad
    \Delta u_\epsilon
    =
    \partial^2_{\nu\nu} u_\epsilon + \O \partial_\nu u_\epsilon .
  \end{equation}
  Here and in what follows, by $\O$ we denote a constant dependent
  only on the geometry of $\Omega$. In particular, $\O$ is independent
  of the flow $f$, of the source $F$ and of the initial datum $u_o$.
  Then, using~\eqref{eq:14} and the boundedness of $\sigma_\eta (z)/z$
  and of $\sigma'_\eta$,
  \begin{align*}
    & \sigma'_\eta\left(\partial_j u_\epsilon (\tau,\xi)\right) \,
    \partial_\nu \left(\partial_j u_\epsilon (t,\xi)\right) -
    \frac{\sigma_\eta\left(\partial_j u_\epsilon (\tau,\xi)\right)}
    {\partial_j u_\epsilon (\tau,\xi)} \; \Delta u_\epsilon (\tau,\xi)
    \, \nu_j (\xi)
    \\
    = \ & \left(\sigma'_\eta\left(\partial_j u_\epsilon
        (\tau,\xi)\right) - \frac{\sigma_\eta\left(\partial_j
          u_\epsilon (\tau,\xi)\right)} {\partial_j u_\epsilon
        (\tau,\xi)} \right)
    \partial^2_{\nu\nu} u_\epsilon (\tau, \xi) \, \nu_j (\xi) + \O
    \, \partial_\nu u_\epsilon (\tau,\xi)\,,
  \end{align*}
  whence, by~\cite[Lemma~A.3]{BardosBrezisBrezis}, see
  also~\cite[Chapter~4]{GagneuxBook},
  \begin{align}
    \nonumber \lim_{\eta\to 0} \left[\eqref{eq:u2}\right] = \ & \O \,
    \epsilon \int_0^t \int_{\partial\Omega} \partial_\nu u_\epsilon
    (\tau,\xi) \d\xi \d\tau
    \\
    \nonumber \leq \ & \O \, \epsilon \int_0^t \int_{\partial\Omega}
    \norma{\grad u_\epsilon (\tau,\xi)} \d\xi \d\tau
    \\
    \nonumber \leq \ & \O \, \epsilon \int_0^t \int_\Omega
    \modulo{\Delta u_\epsilon (\tau,x)} \d{x} \d\tau
    \\
    \nonumber \leq \ & \O \int_0^t \int_\Omega \modulo{\partial_t
      u_\epsilon (\tau,x)} \d{x} \d\tau
    \\
    \label{eq:16}
    & + \O \, \norma{\partial_u f}_{\L\infty ([0,t]\times \Omega
      \times \mathcal{U}(t); \reali^n)} \int_0^t \int_\Omega
    \norma{\grad u_\epsilon (\tau,x)} \d{x} \d\tau
    \\
    \nonumber & + \O \, \mathcal{L}^n (\Omega) \, t \left( \norma{\div
        f}_{\L\infty ([0,t] \times \Omega \times \mathcal{U}
        (t);\reali)} + \norma{F}_{\L\infty ([0,t]\times\Omega \times
        \mathcal{U}(t);\reali)} \right) \,.
  \end{align}
  Passing to~\eqref{eq:u3}, we have the following estimate that holds
  uniformly in $\eta$:
  \begin{equation}
    \label{eq:15}
    \!\!\!\!
    \left[\eqref{eq:u3}\right]
    \leq
    \mathcal{H}^{n-1} (\partial\Omega) \, t
    \left(
      \norma{\div f (\cdot, \cdot, 0)}_{\L\infty ([0,t] \times \partial\Omega;\reali)}
      +
      \norma{F (\cdot, \cdot, 0)}_{\L\infty ([0,t]\times \partial\Omega;\reali)}
    \right) .
  \end{equation}
  Insert now~\eqref{eq:u4}, \eqref{eq:16} and~\eqref{eq:15}
  in~\eqref{eq:intEta}--\eqref{eq:u3} to obtain
  \begin{align*}
    & \int_\Omega \modulo{\partial_j u_\epsilon (t,x)} \d{x}
    \\
    \leq \ & \int_\Omega \modulo{\partial_j u_\epsilon (0,x)} \d{x}
    \\
    & + \mathcal{L}^n (\Omega) \, t \left( \norma{\grad \div
        f}_{\L\infty ([0,t] \times \Omega \times \mathcal{U} (t);
        \reali^n)} + \norma{\grad F}_{\L\infty ([0,t] \times \Omega
        \times \mathcal{U} (t); \reali^n)} \right)
    \\
    & + \O \, \mathcal{L}^n (\Omega) \, t \left( \norma{\div
        f}_{\L\infty ([0,t] \times \Omega \times \mathcal{U}
        (t);\reali)} + \norma{F}_{\L\infty ([0,t]\times\Omega \times
        \mathcal{U}(t);\reali)} \right)
    \\
    & + \mathcal{H}^{n-1} (\partial\Omega) \, t \left( \norma{\div f
        (\cdot, \cdot, 0)}_{\L\infty ([0,t]
        \times \partial\Omega;\reali)} + \norma{F (\cdot, \cdot,
        0)}_{\L\infty ([0,t]\times \partial\Omega;\reali)} \right)
    \\
    & + \left( \norma{\grad \partial_u f}_{\L\infty ([0,t] \times
        \Omega \times \mathcal{U} (t);\reali^{n\times n})} +
      \norma{\partial_u F}_{\L\infty ([0,t]\times\Omega \times
        \mathcal{U}(t);\reali)} \right) \int_0^t \norma{\partial_j
      u_\epsilon (\tau)}_{\L1 (\Omega;\reali)} \d\tau
    \\
    & + \O \int_0^t \norma{\partial_t u_\epsilon (\tau)}_{\L1
      (\Omega;\reali)} \d\tau
    \\
    & + \O \, \norma{\partial_u f}_{\L\infty ([0,t]\times \Omega
      \times \mathcal{U}(t); \reali^n)} \int_0^t \norma{\grad
      u_\epsilon (\tau)}_{\L1 (\Omega;\reali^n)} \d\tau.
  \end{align*}
  Summing over $j=1, \ldots, n$ and using the notation $\O$, we get
  \begin{eqnarray}
    \nonumber
    & &
    \norma{\grad u_\epsilon (t)}_{\L1 (\Omega;\reali^n)}
    \\
    \label{eq:17_i}
    & \leq &
    \sqrt{n} \, \norma{\grad u_\epsilon (0)}_{\L1 (\Omega;\reali^n)}
    \\
    \nonumber
    & &
    +
    \O \, t
    \Big[
    \norma{\grad \div f}_{\L\infty ([0,t] \times \Omega \times \mathcal{U} (t); \reali^n)}
    +
    \norma{\grad F}_{\L\infty ([0,t] \times \Omega \times \mathcal{U} (t); \reali^n)}
    \\
    \nonumber
    & &
    \qquad\qquad
    +
    \norma{\div f}_{\L\infty ([0,t] \times \Omega \times \mathcal{U}
      (t);\reali)}
    +
    \norma{F}_{\L\infty  ([0,t]\times\Omega \times \mathcal{U}(t);\reali)}
    \\
    \nonumber
    & &
    \qquad\qquad
    +
    \norma{\div f (\cdot, \cdot, 0)}_{\L\infty ([0,t] \times \partial\Omega;\reali)}
    +
    \norma{F (\cdot, \cdot, 0)}_{\L\infty ([0,t]\times \partial\Omega;\reali)}
    \Big]
    \\
    \nonumber
    & &
    +
    \O \int_0^t \norma{\grad u_\epsilon (\tau)}_{\L1 (\Omega;\reali^n)} \d\tau
    \\
    \nonumber
    & &
    \times
    \left[
      \norma{\grad \partial_u f}_{\L\infty ([0,t]   \times \Omega \times \mathcal{U} (t);\reali^{n\times n})}
      +
      \norma{\partial_u F}_{\L\infty ([0,t]\times\Omega \times
        \mathcal{U}(t);\reali)}
      +
      \norma{\partial_u f}_{\L\infty ([0,t]\times \Omega \times \mathcal{U}(t); \reali^n)}
    \right]
    \\
    \label{eq:17_ii}
    & &
    +
    \O \int_0^t \norma{\partial_t u_\epsilon (\tau)}_{\L1 (\Omega;\reali)} \d\tau.
  \end{eqnarray}
  Summing the inequalities~\eqref{eq:18}, \eqref{eq:19}
  and~\eqref{eq:17_i}--\eqref{eq:17_ii} we obtain the estimate
  \begin{eqnarray*}
    \norma{\partial_t u_\epsilon (t)}_{\L1 (\Omega;\reali)}
    +
    \norma{\grad u_\epsilon (t)}_{\L1 (\Omega; \reali^n)}
    \!\!\!& \leq &\!\!\!
    A_1
    +
    A_2 \, t
    +
    A_3 \norma{\grad u_o}_{\L1 (\Omega; \reali^n)}
    +
    \epsilon
    \norma{\Delta u_o}_{\L1 (\Omega; \reali)}
    \\
    \!\!\!& &\!\!\!
    \!\!\!
    +
    A_4
    \int_0^t
    \left[
      \norma{\partial_t u_\epsilon (\tau)}_{\L1(\Omega;\reali)}
      +
      \norma{\grad u_\epsilon (\tau)}_{\L1  (\Omega; \reali^n)}
    \right] \d\tau,
  \end{eqnarray*}
  where
  \begin{equation}
    \label{eq:20}
    \begin{array}{rcl}
      A_1 & = &
      \O \left(
        \norma{\div f}_{\L\infty ([0,t]\times \Omega \times \mathcal{U} (t); \reali)}
        +
        \norma{F}_{\L\infty ([0,t]\times \Omega \times \mathcal{U} (t); \reali)}
      \right)
      \\
      A_2 & = &
      \O
      \Big[
      \norma{\grad \div f}_{\L\infty ([0,t] \times \Omega \times \mathcal{U} (t); \reali^n)}
      +
      \norma{\grad F}_{\L\infty ([0,t] \times \Omega \times \mathcal{U} (t); \reali^n)}
      \\
      & &
      \qquad\qquad
      +
      \norma{\div f}_{\L\infty ([0,t] \times \Omega \times \mathcal{U}
        (t);\reali)}
      +
      \norma{F}_{\L\infty  ([0,t]\times\Omega \times \mathcal{U}(t);\reali)}
      \\
      & &
      \qquad\qquad
      +
      \norma{\div \partial_t f}_{\L\infty ([0,t] \times \Omega \times \mathcal{U}
        (t);\reali)}
      +
      \norma{\partial_t F}_{\L\infty  ([0,t]\times\Omega \times \mathcal{U}(t);\reali)}
      \\
      & &
      \qquad\qquad
      +
      \norma{\div f (\cdot, \cdot, 0)}_{\L\infty ([0,t] \times \partial\Omega;\reali)}
      +
      \norma{F (\cdot, \cdot, 0)}_{\L\infty ([0,t]\times \partial\Omega;\reali)}
      \Big]
      \\
      A_3 & = &
      \O
      +
      \norma{\partial_u f}_{\L\infty ([0,t]\times\Omega\times\mathcal{U} (t); \reali^n)}
      \\
      A_4 & = &
      \O
      \Big[
      1
      +
      \norma{\partial_t \partial_u f}_{\L\infty ([0,t]\times\Omega\times\mathcal{U} (t); \reali^n)}
      +
      \norma{\partial_u F}_{\L\infty ([0,t]\times\Omega\times\mathcal{U} (t); \reali)}
      \\
      & &
      +      \norma{\grad \partial_u f}_{\L\infty ([0,t]   \times \Omega \times \mathcal{U} (t);\reali^{n\times n})}
      +
      \norma{\partial_u f}_{\L\infty ([0,t]\times \Omega \times \mathcal{U}(t); \reali^n)}
      \Big] \, .
    \end{array}
  \end{equation}
  Note that the $A_i$ are increasing with $t$. Hence, an application
  of Gronwall Lemma yields
  \begin{eqnarray*}
    & &
    \norma{\partial_t u_\epsilon (t)}_{\L1 (\Omega;\reali)}
    +
    \norma{\grad u_\epsilon (t)}_{\L1 (\Omega; \reali^n)}
    \\
    & \leq &
    \left(
      A_1
      +
      A_2 \, t
      +
      A_3 \norma{\grad u_o}_{\L1 (\Omega; \reali^n)}
      +
      \epsilon
      \norma{\Delta u_o}_{\L1 (\Omega; \reali)}
    \right)
    e^{A_4 \, t}.
  \end{eqnarray*}
  From the inequality above, \eqref{eq:tvPar} follows easily,
  introducing the notation~\eqref{eq:30}. Noting that
  \begin{displaymath}
    \norma{u_\epsilon (t) - u_\epsilon (s)}_{\L1 (\Omega; \reali)} \leq
    \int_s^t \norma{\partial_t u_\epsilon (\tau)}_{\L1 (\Omega; \reali)},
  \end{displaymath}
  we obtain~\eqref{eq:dipTempoPar}, concluding the proof.
\end{proofof}

\section{Proofs Related to the Hyperbolic Problem}
\label{sec:PHP}

\begin{proofof}{Proposition~\ref{prop:limiteZero}}
  The family $u_\epsilon$ of solutions to~\eqref{eq:visco} as
  constructed in Lemma~\ref{lem:tvParZero} is uniformly bounded in
  $\L1 (I \times \Omega;\reali)$ by~\eqref{eq:Linfty}. It is also
  totally bounded in $\L1 (I \times \Omega;\reali)$ thanks
  to~\cite[Corollary~8]{HancheHolden}, which can be applied
  by~\eqref{eq:tvPar}.

  To prove that cluster point of the $u_\epsilon$ is a solution
  to~\eqref{eq:1} in the sense of Definition~\ref{def:sol}, we
  introduce $k \in \reali$ and a test function $\phi \in \Cc2
  (]-\infty, T[ \times \reali^n; \reali^+)$. We multiply
  equation~\eqref{eq:visco} by $\sigma'_\eta \left(u_\epsilon (t,x) -
    k\right) \phi (t,x)$, with $\eta > 0$ and $\sigma'_\eta$ as
  in~\eqref{eq:sigma}. Then, we integrate over $I \times \Omega$:
  \begin{equation}
    \label{eq:10}
    \begin{split}
      \!\int_I \!\!\int_\Omega \!\!\left(\partial_t u_\epsilon (t,x) +
        \Div f \left(t,x,u_\epsilon (t,x)\right)
        -F\left(t,x,u_\epsilon (t,x)\right)\right) \sigma'_\eta\!
      \left(u_\epsilon (t,x) - k\right) \phi (t,x) \d{x} \d{t} \\=
      \int_I\!\! \int_\Omega \epsilon \, \Delta u_\epsilon (t,x) \,\,
      \sigma'_\eta\! \left(u_\epsilon (t,x) - k\right) \phi (t,x)
      \d{x} \d{t}.
    \end{split}
  \end{equation}
  Consider each term in~\eqref{eq:10} separately. Integrate by part
  the first term:
  \begin{align}
    \nonumber & \int_I \!\!\int_\Omega
    \partial_t u_\epsilon (t,x) \,\, \sigma'_\eta\! \left(u_\epsilon
      (t,x) - k\right) \phi (t,x) \d{x} \d{t}
    \\
    \nonumber = \ & \int_I \!\!\int_\Omega \frac{\d{}}{\d{t}}
    \sigma_\eta\!  \left(u_\epsilon (t,x) - k\right) \, \phi (t,x)
    \d{x} \d{t}
    \\
    \label{eq:s1}
    = \ & - \int_\Omega \sigma_\eta\! \left(u_o (x) - k\right) \, \phi
    (0,x) \d{x} - \int_I \!\!\int_\Omega \sigma_\eta\!\left(u_\epsilon
      (t,x) - k\right) \,
    \partial_t \phi (t,x) \d{x} \d{t}.
  \end{align}
  Concerning the second term in the left hand side of~\eqref{eq:10},
  first integrate by part, then add and subtract $\displaystyle\int_I
  \int_\Omega f (t,x,k) \cdot \grad \left( \sigma'_\eta\!
    \left(u_\epsilon (t,x) - k\right) \phi (t,x) \right) \d{x}
  \d{t}$. After some rearrangements,
  \begin{align}
    \nonumber & \int_I \int_\Omega \Div f \left(t,x,u_\epsilon
      (t,x)\right) \,\, \sigma'_\eta\! \left(u_\epsilon (t,x) -
      k\right) \phi (t,x) \d{x} \d{t}
    \\
    \nonumber = & \int_I \int_{\partial\Omega} f \left(t,x,u_\epsilon
      (t,\xi)\right) \,\, \sigma'_\eta\!  \left(u_\epsilon (t,\xi) -
      k\right) \phi (t,\xi) \cdot \nu (\xi) \d{\xi} \d{t}
    \\
    \nonumber & - \int_I \int_\Omega f \left(t,x,u_\epsilon
      (t,x)\right) \cdot \grad \left( \sigma'_\eta\! \left(u_\epsilon
        (t,x) - k\right) \phi (t,x) \right) \d{x} \d{t}
    \\
    \nonumber = & \int_I \int_{\partial\Omega} f (t,x,0) \,\,
    \sigma'_\eta\!  \left(- k\right) \phi (t,\xi) \cdot \nu (\xi)
    \d{\xi} \d{t}
    \\
    \nonumber & - \int_I \int_\Omega \left(f \left(t,x,u_\epsilon
        (t,x)\right) - f (t,x,k) \right) \cdot \grad \left(
      \sigma'_\eta\!  \left(u_\epsilon (t,x) - k\right) \phi (t,x)
    \right) \d{x} \d{t}
    \\
    \nonumber & - \int_I \int_\Omega f (t,x,k) \cdot \grad \left(
      \sigma'_\eta\! \left(u_\epsilon (t,x) - k\right) \phi (t,x)
    \right) \d{x} \d{t}
    \\
    \label{eq:s2}
    \begin{split}
      = & \displaystyle \int_I \int_{\partial\Omega} f (t,x,0) \,\,
      \sigma'_\eta\! \left(- k\right) \phi (t,\xi) \cdot \nu (\xi)
      \d{\xi} \d{t}
      \\
      & \displaystyle - \int_I \int_\Omega \left(f
        \left(t,x,u_\epsilon (t,x)\right) - f (t,x,k) \right) \cdot
      \grad \left( \sigma'_\eta\!  \left(u_\epsilon (t,x) - k\right)
        \phi (t,x) \right) \d{x} \d{t}
      \\
      & \displaystyle - \int_I \int_{\partial\Omega} f (t,\xi,k) \,
      \sigma'_\eta \!  \left(- k\right) \phi (t,\xi) \cdot \nu (\xi)
      \d{\xi} \d{t}
      \\
      & \displaystyle + \int_I \int_\Omega \div f (t,x,k) \,\,
      \sigma'_\eta\!  \left(u_\epsilon (t,x) - k\right) \phi (t,x)
      \d{x} \d{t}.
    \end{split}
  \end{align}
  We do not modify the third term in the left hand side
  of~\eqref{eq:10}. Passing to the right hand side of~\eqref{eq:10},
  we have:
  \begin{align}
    \nonumber & \int_I \int_\Omega \epsilon \, \Delta u_\epsilon (t,x)
    \,\, \sigma'_\eta\! \left(u_\epsilon (t,x) - k\right) \phi (t,x)
    \d{x} \d{t}
    \\
    \nonumber = \ & \epsilon \int_I \int_{\partial\Omega} \grad
    u_\epsilon (t,\xi) \, \sigma'_\eta\! \left(- k \right) \phi
    (t,\xi) \cdot \nu (\xi) \d{\xi} \d{t}
    \\
    \nonumber & - \epsilon \int_I \int_\Omega \grad u_\epsilon (t,x)
    \cdot \grad\left( \sigma'_\eta\! \left(u_\epsilon (t,x) - k\right)
      \phi (t,x)\right) \d{x} \d{t}
    \\
    \label{eq:s3}
    \begin{split}
      = \ & \epsilon \int_I \int_{\partial\Omega} \sigma'_\eta\!
      \left(- k \right)\, \phi (t,\xi)\,\, \grad u_\epsilon (t,\xi)
      \cdot \nu (\xi) \d{\xi} \d{t}
      \\
      & - \epsilon \int_I \int_\Omega \sigma'_\eta\!  \left(u_\epsilon
        (t,x) - k\right) \grad u_\epsilon (t,x) \cdot \grad \phi (t,x)
      \d{x} \d{t}
      \\
      & - \epsilon \int_I \int_\Omega \norma{\grad u_\epsilon (t,x)}^2
      \, \sigma''_\eta\!  \left(u_\epsilon (t,x) - k\right) \phi (t,x)
      \d{x} \d{t}.
    \end{split}
  \end{align}
  Using~\eqref{eq:s1}, \eqref{eq:s2} and~\eqref{eq:s3},
  equation~\eqref{eq:10} becomes
  \begin{align}
    \label{eq:daQui}
    & \int_I \int_\Omega \sigma_\eta\!\left(u_\epsilon (t,x) -
      k\right) \,
    \partial_t \phi (t,x) \d{x} \d{t}
    \\
    \nonumber & + \int_I \int_\Omega \sigma'_\eta\!  \left(u_\epsilon
      (t,x) -k\right) \left(f \left(t,x,u_\epsilon (t,x)\right) - f
      (t,x,k) \right) \cdot \grad \phi (t,x) \d{x} \d{t}
    \\
    \nonumber & + \int_I \int_\Omega \sigma''_\eta\!  \left(u_\epsilon
      (t,x)- k\right) \left(f \left(t,x,u_\epsilon (t,x)\right) - f
      (t,x,k) \right) \cdot \grad u_\epsilon (t,x) \, \phi (t,x) \d{x}
    \d{t}
    \\
    \nonumber & + \int_I \int_\Omega \sigma'_\eta\!  \left(u_\epsilon
      (t,x) - k\right) \left( F \left(t,x,u_\epsilon (t,x)\right) -
      \div f (t,x,k) \right) \phi (t,x) \d{x} \d{t}
    \\
    \nonumber & + \int_\Omega \sigma_\eta\! \left(u_o (x) - k\right)
    \, \phi (0,x) \d{x}
    \\
    \nonumber & - \int_I \int_{\partial\Omega} \sigma'_\eta\!  \left(-
      k\right) \left(f(t,\xi,0) - f (t,\xi,k)\right)\phi (t,\xi) \cdot
    \nu (\xi) \d{\xi} \d{t}
    \\
    \nonumber = \ & \epsilon \int_I \int_\Omega \sigma'_\eta\!
    \left(u_\epsilon (t,x) - k\right) \grad u_\epsilon (t,x) \cdot
    \grad \phi (t,x) \d{x} \d{t}
    \\
    \nonumber & + \epsilon \int_I \int_\Omega \norma{\grad u_\epsilon
      (t,x)}^2 \, \sigma''_\eta\!  \left(u_\epsilon (t,x) - k\right)
    \phi (t,x) \d{x} \d{t}
    \\
    \label{eq:aQui}
    & - \epsilon \int_I \int_{\partial\Omega} \sigma'_\eta\!  \left(-
      k \right)\, \phi (t,\xi)\,\, \grad u_\epsilon (t,\xi) \cdot \nu
    (\xi) \d{\xi} \d{t} \,.
  \end{align}
  Choose now any sequence $\epsilon_m$, with $m \in \naturali$, and
  call $u_\infty$ the $\L1$ limit of a convergent subsequence. For the
  sake of readability, we write $u_\epsilon$ instead of
  $u_{\epsilon_m}$. The left hand side
  of~\eqref{eq:daQui}--\eqref{eq:aQui} converges to the same
  expression with $u_\epsilon$ replaced by $u_\infty$. The first term
  in the right hand side can be treated as follows:
  \begin{align*}
    & \epsilon_m \int_I\!\! \int_\Omega \sigma'_\eta \left(u_\epsilon
      (t,x) - k\right) \grad u_\epsilon (t,x) \cdot \grad \phi (t,x)
    \d{x} \d{t}
    \\
    \geq \ & - \modulo{ \epsilon_m \int_I\!\! \int_\Omega \sigma'_\eta
      \left(u_\epsilon (t,x) - k\right) \grad u_\epsilon (t,x) \cdot
      \grad \phi (t,x) \d{x} \d{t}}
    \\
    \geq \ & - \epsilon_m \, \norma{\grad \phi}_{\L\infty (I\times
      \Omega; \reali)} \, \norma{\grad u_\epsilon}_{\L1 (I\times
      \Omega; \reali^n)}
    \\
    \stackrel{m \to +\infty}{=} & 0\,,
  \end{align*}
  since $\epsilon_m$ is a multiplicative coefficient in the
  estimate~\eqref{eq:tvPar} of $\norma{\grad u_\epsilon}_{\L1 (I\times
    \Omega; \reali^n)}$, see~\eqref{eq:30}.

  The second term in the right hand side
  of~\eqref{eq:daQui}--\eqref{eq:aQui} is non negative.

  To compute the limit as $m \to +\infty$ of the third term in the
  right hand side of~\eqref{eq:daQui}--\eqref{eq:aQui}, introduce a
  function $\Phi_h \in \Cc2 (\reali^n; [0,1])$ with the following
  properties:
  \begin{equation}
    \label{eq:11}
    \begin{array}{cc}
      \begin{array}[b]{l}
        \Phi_h (\xi) = 1 \mbox{ for all } \xi \in \partial\Omega,
        \\
        \\
        \Phi_h (x) = 0 \mbox{ for all } x \in \Omega
        \mbox{ such that } B (x,h) \subseteq \Omega,
        \\
        \\
        \norma{\nabla \Phi_h}_{\L\infty (\Omega;\reali^n)} \leq 1/h.
        \\
        \\
        \\
      \end{array}
      &
      \includegraphics[width=0.33\textwidth]{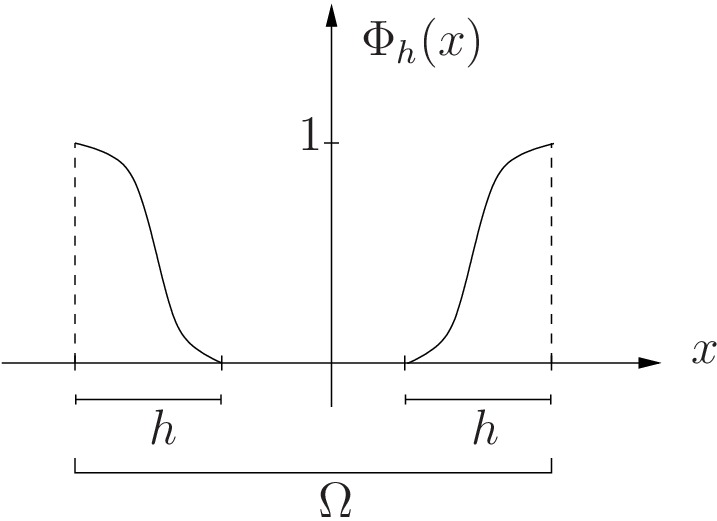}
    \end{array}
  \end{equation}
  Then, using equation~\eqref{eq:visco} and integration by parts,
  except for the constant $\sigma'_\eta (-k)$, the considered term
  becomes
  \begin{align}
    \nonumber & \epsilon_m \int_I \int_{\partial\Omega} \phi (t,\xi)
    \, \grad u_\epsilon (t,\xi) \cdot \nu (\xi) \d{\xi} \d{t}
    \\
    \nonumber = \ & \epsilon_m \int_I \int_{\partial\Omega} \phi
    (t,\xi)\, \Phi_h (\xi) \, \grad u_\epsilon (t,\xi) \cdot \nu (\xi)
    \d{\xi} \d{t}
    \\
    \nonumber = \ & \epsilon_m \int_I \int_{\Omega} \left( \Delta
      u_\epsilon (t,x) \, \phi (t,x)\, \Phi_h (x) + \grad u_\epsilon
      (t,x) \cdot \grad\left( \phi (t,x)\, \Phi_h (x) \right) \right)
    \d{x} \d{t}
    \\
    \nonumber =\ & \int_I \int_{\Omega} \left(
      \partial_t u_\epsilon + \Div f \left(t,x,u_\epsilon (t,x)\right)
      - F \left(t,x,u_\epsilon (t,x)\right) \right) \phi (t,x)\,
    \Phi_h (x) \d{x} \d{t}
    \\
    \nonumber & + \epsilon_m \int_I \int_{\Omega} \grad u_\epsilon
    (t,x) \cdot \grad \left( \phi (t,x)\, \Phi_h (x) \right) \d{x}
    \d{t}
    \\
    \nonumber = \ & \int_\Omega u_o (x) \, \phi (0,x) \, \Phi_h (x)
    \d{x} - \int_I \int_{\Omega} u_\epsilon (t,x) \, \partial_t \phi
    (t,x) \, \Phi_h (x) \d{x} \d{t}
    \\
    \nonumber & + \int_I \int_{\partial\Omega} f
    \left(t,\xi,u_\epsilon (t,\xi)\right) \,\phi (t,\xi) \, \Phi_h
    (\xi) \cdot \nu (\xi) \d\xi \d{t}
    \\
    \nonumber & - \int_I \int_{\Omega} f \left(t,x,u_\epsilon
      (t,x)\right) \cdot \grad\left(\phi (t,x)\,\Phi_h (x)\right)
    \d{x}\d{t}
    \\
    \nonumber & - \int_I \int_{\Omega} F \left(t,x,u_\epsilon
      (t,x)\right) \phi (t,x)\, \Phi_h (x) \d{x} \d{t}
    \\
    \nonumber & + \epsilon_m \int_I \int_{\Omega} \grad u_\epsilon
    (t,x) \cdot \grad \left(\phi (t,x)\, \Phi_h (x) \right) \d{x}
    \d{t} \\ \label{eq:25}
    \begin{split}
      = \ & \int_\Omega u_o (x) \, \phi (0,x) \, \Phi_h (x) \d{x}
      \\
      & - \int_I \int_{\Omega} \bigl( u_\epsilon (t,x) \, \partial_t
      \phi (t,x) + f \left(t,x,u_\epsilon (t,x)\right) \cdot \grad
      \phi (t,x)
      \\
      &\qquad + F \left(t,x,u_\epsilon (t,x)\right) \phi (t,x) -
      \epsilon_m \grad u_\epsilon (t,x) \cdot \grad \phi (t,x) \bigr)
      \Phi_h (x) \d{x} \d{t}
      \\
      & - \int_I \int_{\Omega} \left( f \left(t,x,u_\epsilon
          (t,x)\right) - \epsilon_m \grad u_\epsilon (t,x) \right)
      \cdot \grad \Phi_h (x) \, \phi (t,x) \d{x} \d{t}
      \\
      & + \int_I \int_{\partial\Omega} f (t,\xi,0) \,\phi (t,\xi) \,
      \Phi_h (\xi) \cdot \nu (\xi) \d\xi \d{t}.
    \end{split}
  \end{align}
  Let $m \to +\infty$:
  \begin{align}
    \nonumber & \lim_{m \to +\infty} \left[\eqref{eq:25}\right]
    \\
    \label{eq:26}
    \begin{split}
      = \ & - \int_I \int_{\Omega} \bigl( u_\infty (t,x) \, \partial_t
      \phi (t,x) + f \left(t,x,u_\infty (t,x)\right) \cdot \grad \phi
      (t,x)
      \\
      & \qquad\qquad\qquad\qquad\qquad + F \left(t,x,u_\infty
        (t,x)\right) \phi (t,x) \bigr) \Phi_h (x) \d{x} \d{t}
      \\
      & - \int_I \int_{\Omega} f \left(t,x,u_\infty (t,x)\right) \cdot
      \grad \Phi_h (x) \, \phi (t,x) \d{x} \d{t}
      \\
      & + \int_\Omega u_o (x) \, \phi (0,x) \, \Phi_h (x) \d{x} +
      \int_I \int_{\partial\Omega} f (t,\xi,0) \,\phi (t,\xi) \cdot
      \nu (\xi) \d\xi \d{t}.
    \end{split}
  \end{align}
  Now let $h \to 0$. Thanks to Lemma~\ref{lem:uffaMultiD} and
  Lemma~\ref{lem:tr3}, we obtain
  \begin{displaymath}
    \lim_{h \to 0} \left[\eqref{eq:26}\right] =
    \int_I \int_{\partial \Omega}
    \left(
      f (t,\xi,0) -  f \left(t,\xi,\tr u_\infty (t,\xi) \right)
    \right)
    \phi (t,\xi)  \cdot \nu (\xi) \d{\xi} \d{t}.
  \end{displaymath}
  Hence
  \begin{equation}
    \label{eq:27}
    \lim_{m \to +\infty} \left[\eqref{eq:aQui}\right] =
    -
    \int_I \int_{\partial \Omega}
    \sigma'_\eta ( - k)
    \left(
      f (t,\xi,0) -  f \left(t,\xi,\tr u_\infty (t,\xi) \right)
    \right)
    \phi (t,\xi)  \cdot \nu (\xi) \d{\xi} \d{t}.
  \end{equation}
  Therefore, in the limit $m \to +\infty$, we obtain that the
  equality~\eqref{eq:daQui}--\eqref{eq:aQui} implies the inequality
  \begin{eqnarray*}
    & &
    \int_I \int_\Omega \sigma_\eta\!\left(u_\infty (t,x) -
      k\right) \,
    \partial_t \phi (t,x) \d{x} \d{t}
    \\
    \nonumber
    & &
    +
    \int_I \int_\Omega \sigma'_\eta\!
    \left(u_\infty (t,x) -k\right) \left(f \left(t,x,u_\infty (t,x)\right) -
      f (t,x,k) \right) \cdot \grad \phi (t,x) \d{x} \d{t}
    \\
    \nonumber
    & &
    +
    \int_I \int_\Omega \sigma''_\eta\!
    \left(u_\infty (t,x)- k\right) \left(f \left(t,x,u_\infty (t,x)\right) -
      f (t,x,k) \right) \cdot \grad u_\infty (t,x) \, \phi (t,x)
    \d{x} \d{t}
    \\
    \nonumber
    & &
    +
    \int_I \int_\Omega \sigma'_\eta\!
    \left(u_\infty (t,x) - k\right) \left( F \left(t,x,u_\infty (t,x)\right)
      - \div f (t,x,k) \right) \phi (t,x) \d{x} \d{t}
    \\
    \nonumber
    & &
    +
    \int_\Omega \sigma_\eta\! \left(u_o (x) -
      k\right) \, \phi (0,x) \d{x}
    \\
    \nonumber
    & &
    -
    \int_I  \int_{\partial\Omega} \sigma'_\eta\!
    \left(- k\right) \left(f(t,\xi,0) - f (t,\xi,k)\right)\phi (t,\xi)
    \cdot \nu (\xi) \d{\xi} \d{t}
    \\
    & \geq &
    \int_I \int_{\partial \Omega}
    \sigma'_\eta ( - k)
    \left(
      f (t,\xi,0) -  f \left(t,\xi,\tr u_\infty (t,\xi) \right)
    \right)
    \phi (t,\xi)  \cdot \nu (\xi) \d{\xi} \d{t} \,.
  \end{eqnarray*}
  Let now $\eta \to 0$. Thanks to~\cite[Lemma 2]{BardosLerouxNedelec},
  to the choice~\eqref{eq:sigma} of $\sigma_\eta$ and of its
  derivative, we get
  \begin{eqnarray*}
    & &
    \int_I \int_\Omega \modulo{u_\infty (t,x) - k} \,
    \partial_t \phi (t,x) \d{x} \d{t}
    \\
    & &
    +
    \int_I \int_\Omega \sgn\left(u_\infty (t,x)
      -k\right) \left(f \left(t,x,u_\infty (t,x)\right) - f (t,x,k) \right)
    \cdot \grad \phi (t,x) \d{x} \d{t}
    \\
    & &
    +
    \int_I \int_\Omega \sgn \left(u_\infty (t,x) -
      k\right) \left( F \left(t,x,u_\infty (t,x)\right) - \div f (t,x,k)
    \right) \phi (t,x) \d{x} \d{t}
    \\
    & &
    + \int_\Omega \modulo{u_o (x) - k} \, \phi (0,x)
    \d{x}
    \\
    & &
    - \int_I  \int_{\partial\Omega} \sgn \left(- k\right)
    \left(f\left(t,\xi,\tr u_\infty (t, \xi)\right) - f (t,\xi,k)\right)\phi (t,\xi) \cdot \nu
    (\xi) \d{\xi} \d{t}
    \\
    & \geq &
    0 \,,
  \end{eqnarray*}
  that is~\eqref{eq:4} in the case $u_b = 0$. Hence $u_\infty$ is a
  solution to~\eqref{eq:1} in the sense of Definition~\ref{def:sol}.

  As a consequence of~\eqref{eq:Linfty} in Lemma~\ref{lem:stimaLinf},
  $u_\infty$ satisfies the $\L\infty$ estimate
  \begin{displaymath}
    \norma{u_\infty}_{\L\infty ([0,t]\times \Omega; \reali)}
    \leq
    \norma{u_o}_{\L\infty (\Omega; \reali)} e^{c_1 \, t}
    + \frac{c_2}{c_1}\left(e^{c_1 \, t} - 1 \right) \,,
  \end{displaymath}
  where $c_1, \, c_2$ are defined in~\eqref{eq:c1c2}. Thanks to the
  lower semicontinuity in $\L1$ of the total variation,
  see~\cite[Remark~3.5]{AmbrosioFuscoPallara}, the
  bound~\eqref{eq:tvPar} in Lemma~\ref{lem:tvParZero} gives
  \begin{displaymath}
    \tv \left(u_\infty (t)\right)
    \leq
    \liminf_{\epsilon \to 0} \tv \left(u_\epsilon (t)\right)
    \leq
    \liminf_{\epsilon \to 0} \mathcal{L}_\epsilon (t)
    =
    \mathcal{L} (t)
  \end{displaymath}
  with $\mathcal{L}_\epsilon (t)$ and $\mathcal{L} (t)$ as defined
  in~\eqref{eq:30} and~\eqref{eq:21}.

  From~\eqref{eq:dipTempoPar} in Lemma~\ref{lem:tvParZero}, we have
  for $t, \, s \in I$
  \begin{eqnarray*}
    \norma{u_\infty (t) - u_\infty (s)}_{\L1 (\Omega;\reali)}
    & = &
    \lim_{\epsilon \to 0}
    \norma{u_\epsilon (t) - u_\epsilon (s)}_{\L1 (\Omega;\reali)}
    \\
    & \leq &
    \lim_{\epsilon \to 0} \mathcal{L}_\epsilon (\max\{t,s\}) \; \modulo{t-s}
    \\
    & = &
    \mathcal{L} (\max\{t,s\}) \; \modulo{t-s} \,,
  \end{eqnarray*}
  concluding the proof.
\end{proofof}

\smallskip The following Lemma will be of use in the proof of
Theorem~\ref{thm:estConComp}.
\begin{lemma}
  \label{lem:elliptic}
  Let $k \in \naturali$ with $k \geq 2$, $\Omega$
  satisfy~{\bf($\boldsymbol{\Omega_{k,\alpha}}$)} and fix $\psi \in
  \C{k,\alpha} (I \times \partial\Omega; \reali)$. Then, the elliptic
  problem
  \begin{equation}
    \label{eq:24}
    \left\{
      \begin{array}{lr@{\;}c@{\;}l}
        \Delta z (t,x) = 0
        & (t,x) & \in & I \times \Omega
        \\
        z (t, \xi) = \psi (t,\xi)
        & (t,\xi) & \in & I \times \partial\Omega
      \end{array}
    \right.
  \end{equation}
  admits a unique solution $z \in \C{k,\alpha} (I \times \bar \Omega;
  \reali)$. Moreover,
  \begin{eqnarray}
    \label{eq:el1}
    \norma{z}_{\L\infty ([0,t]\times\Omega;\reali)}
    & \leq &
    \norma{\psi}_{\L\infty ([0,t]\times\partial\Omega;\reali)} \,,
    \\
    \label{eq:e11emezzo}
    \norma{\grad z}_{\L\infty ([0,t]\times\Omega;\reali^n)}
    & \leq &
    \norma{\psi}_{\C{2,\alpha} ([0,t]\times\partial\Omega;\reali)} \,,
    \\
    \label{eq:el2}
    (k\geq 3) \qquad\quad
    \norma{\partial_t z}_{\L\infty ([0,t]\times\Omega;\reali)}
    & \leq &
    \norma{\partial_t \psi}_{\L\infty ([0,t]\times\partial\Omega;\reali)} \,,
    \\
    \label{eq:el3}
    (k\geq 3) \qquad
    \norma{D z}_{\W1\infty ([0,t]\times\Omega;\reali)}
    & \leq &
    \O \, \norma{\psi}_{\C{3,\alpha} (\partial\Omega;\reali)} \,.
  \end{eqnarray}
\end{lemma}

\begin{proof}
  We verify that the assumptions of~\cite[Chapter~3, \S~8,
  Theorem~20]{FriedmanParabolic}, in the case $p=k+2$, hold.  With
  reference to the notation of~\cite[Chapter~3, \S~8,
  Theorem~20]{FriedmanParabolic}, for any $t \in I$ and for $i=0,
  \ldots, k$, consider the problem
  \begin{equation}
    \label{eq:22}
    \left\{
      \begin{array}{l}
        L z_i (t) = \f \mbox{ in } \Omega
        \\
        z_i (t) = \partial_t^i\psi (t) \mbox{ in } \partial\Omega
      \end{array}
    \right.
    \quad \mbox{ where } \quad
    \begin{array}{l}
      L = \Delta \mbox{ is an elliptic operator,}
      \\
      \f = 0 \mbox{ is of class } \C{k-2,\alpha},
      \\
      \partial\Omega \mbox{ is of class } \C{k,\alpha},
      \\
      \partial_t^i \psi (t) \mbox{ is of class } \C{k,\alpha} \mbox{ in }x.
    \end{array}
  \end{equation}
  Therefore, for any $t \in I$, \eqref{eq:24} admits a solution $z_i
  (t) \in \C{k,\alpha} (\bar\Omega; \reali)$.

  Thanks to the form of $L$ in~\eqref{eq:22} and to the continuity of
  $\partial_t^i \psi$, \cite[Chapter~2, \S~7,
  Theorem~20]{FriedmanParabolic} can be applied, ensuring the
  uniqueness of the solution to~\eqref{eq:24}.

  Concerning the regularity in $t$, remark that $\partial_t^i \psi$ is
  of class $\C{k-1,\alpha}$ in $t$. Hence, for $i=0, \ldots, k$, by
  the Maximum Principle~\cite[Chapter~2, \S~7,
  Theorem~19]{FriedmanParabolic} for any $x \in \Omega$, $t \in I$,
  $h$ sufficiently small such that $t+h \in I$, considering separately
  the cases $i < k$ and $i = k$,
  \begin{displaymath}
    \begin{array}{@{}l@{\qquad}r@{\,}c@{\,}l@{}}
      i<k: &
      \displaystyle
      \modulo{\frac{z_i (t+h,x) - z_i (t,x)}{h} - z_{i+1} (t,x)}
      & \leq &
      \displaystyle
      \sup_{\xi \in \partial\Omega}
      \modulo{
        \frac{\partial_t^i\psi (t+h,\xi)-\partial_t^i \psi(t,\xi)}{h}
        -
        \partial_t^{i+1}\psi (t,h)}
      \\
      & & \leq &
      \sup_{\xi \in \partial\Omega}
      \modulo{
        \partial_t^{i+1}\psi (t + \theta h, \xi)
        -
        \partial_t^{i+1}\psi (t,\xi)}
      \\
      & & \leq &
      \left\{
        \begin{array}{lr@{\,}c@{\,}l}
          \sup_{\xi\in \partial\Omega} \modulo{\partial_t^{i+2}\psi} \, h
          & i+1 & < &k \,,
          \\
          C \, h^\alpha & i+1 & = & k\,,
        \end{array}
      \right.
      \\
      i=k: &
      \modulo{z_k (t+h,x) - z_k (t,x)}
      & \leq &
      \displaystyle
      \sup_{\xi \in \partial \Omega}
      \modulo{
        \partial_t^k \psi (t+h, \xi)
        -
        \partial_t^k \psi (t,\xi)}
      \leq
      C \, h^\alpha \,.
    \end{array}
  \end{displaymath}
  where $C$ is the H\"older constant of $\partial_t^k \psi$. Hence, $z
  = z_0$ is of class $\C{k,ìa}$ in both $t$ and $x$.

  Concerning the bounds on $z$ and on its derivatives, note
  that~\eqref{eq:el1} immediately follow from the Maximum
  Principle~\cite[Formula~(7.5)]{FriedmanParabolic}. The same result
  applies also to $\partial_t z = \Delta \partial_t z$,
  yielding~\eqref{eq:el2}, whenever $k\geq 3$. The Boundary Schauder
  Estimate~\cite[Chapter~3, p.86]{FriedmanParabolic} provides the
  bound for $\grad z$, $\grad^2 z$, $\grad \partial_t z$ and
  $\partial^2_{tt} z$, proving~\eqref{eq:e11emezzo}
  and~\eqref{eq:el3}.
\end{proof}

We recall the following result from~\cite{Kruzkov}, to be used in the
proof below.

\begin{lemma}[{\cite[Lemma~2]{Kruzkov}}]
  \label{lem:kru2}
  Fix positive $r$ and choose $\rho \in [0, \min\{r,T\}]$. Let $w \in
  \L\infty (I \times B (0,r); \reali)$. For $h \in \left]0,
    \rho\right[$, define
  \begin{displaymath}
    A_h = \left\{
      (t,X,s,Y) \in \left(I \times \reali^N\right)^2
      \colon
      \begin{array}{r@{\;}c@{\;}l@{\qquad}r@{\;}c@{\;}l}
        \modulo{t-s} & \leq & h,\,
        &
        (t+s)/2 & \in & [\rho, T-\rho] ,\,
        \\
        \norma{X-Y} & \leq & h,\,
        &
        \norma{X+Y}/2 & \in & [0, r-\rho]
      \end{array}
    \right\}\,.
  \end{displaymath}
  Then, $\displaystyle \lim_{h\to 0+} \frac{1}{h^{1+N}} \int_{A_h}
  \modulo{w (t,X) - w (s,Y)} \d{t} \d{X} \d{s} \d{Y} = 0$.
\end{lemma}

\begin{proofof}{Theorem~\ref{thm:estConComp}}
  Define $z$ as the solution to~\eqref{eq:24} with $\psi (t,\xi) = u_b
  (t,\xi)$. Lemma~\ref{lem:elliptic} applies, ensuring the existence
  and uniqueness of a solution $z$ of class $\C{3,\alpha}$. Note that
  $z (0,x) = 0$ for all $x \in \bar\Omega$. For all $\check k \in
  \reali$ and for all $\check\phi \in \Cc2 (\left]-\infty,
    T\right[\times \reali^n; \reali^+)$ the following equality holds
  \begin{equation}
    \label{eq:solZ}
    \begin{array}{l}
      \displaystyle
      \int_I \int_\Omega
      \left(
        \modulo{z (s,y) - \check k} \, \partial_s \check\phi (s,y)
        +
        \sgn \left(z (s,y) - \check k\right) \, \partial_s z (s,y) \,
        \check\phi (s,y)
      \right)
      \d{y} \, \d{s}
      \\
      \displaystyle
      +
      \int_\Omega \modulo{\check k} \, \check\phi (0, y) \d{y} = 0\, .
    \end{array}
  \end{equation}

  We now apply Proposition~\ref{prop:limiteZero} to the problem
  \begin{equation}
    \label{eq:23}
    \left\{
      \begin{array}{@{}l}
        \partial_t v + \Div g (t,x,v) = G (t,x,v)
        \\
        \begin{array}{@{}lr@{\,}c@{\,}l@{}}
          v (0,x) = u_o (x)
          &
          x & \in & \Omega
          \\
          v (t,\xi) = 0
          &
          (t,\xi) & \in & I\times\partial\Omega
        \end{array}
      \end{array}
    \right.
    \mbox{ where } \quad
    \begin{array}{@{}r@{\,}c@{\,}l@{}}
      g (t,x,v)
      & = &
      f \left(t,x,v+z (t,x)\right) \,,
      \\[3pt]
      G (t,x,v)
      & = &
      F\left(t,x,v+z (t,x)\right) - \partial_t z (t,x) \,.
    \end{array}
  \end{equation}
  To this aim, we verify the necessary assumptions. Clearly,
  \textbf{($\boldsymbol{\Omega_{2,\delta}}$)} holds. By assumption,
  $u_o \in \C{2,\delta} (\bar\Omega;\reali)$ and $u_o (\xi) = 0$ for
  all $\xi \in \Omega$. By construction, the boundary data along
  $I\times\partial\Omega$ is zero. To verify that also~\textbf{(f)}
  and~\textbf{(F)} hold for $g$ and $G$, simply use the assumptions on
  $f$, $F$ and apply Lemma~\ref{lem:elliptic}. Call $v$ the solution
  to~\eqref{eq:23} as constructed in
  Proposition~\ref{prop:limiteZero}. By Definition~\ref{def:sol}, for
  all $\hat k \in \reali$ and for all $\hat\phi \in \Cc2
  (\left]-\infty, T\right[\times \reali^n; \reali^+)$,
  \begin{align}
    \nonumber 0 \leq \ & \int_I \!\int_\Omega \Bigl[ \modulo{v (t,x) -
      \hat k} \, \partial_t \hat\phi (t,x)
    \\
    \nonumber & \quad + \sgn\left[v (t,x) - \hat k\right] \left[
      g\left(t,x,v (t,x)\right) - g (t,x,\hat k) \right] \cdot \grad
    \hat\phi (t,x)
    \\
    \nonumber & \quad + \sgn\left[v (t,x) - \hat k\right] \left[
      G\left(t,x,v (t,x)\right) - \div g (t,x,\hat k) \right] \hat\phi
    (t,x) \Bigr] \d{x} \d{t}
    \\
    \nonumber & + \int_\Omega \modulo{u_o (x) - \hat k} \hat\phi (0,x)
    \d{x}
    \\
    \nonumber & - \int_I \int_{\partial\Omega} \sgn(-\hat k) \left[
      g\left(t,\xi, \tr v (t,\xi)\right) - g (t,\xi,\hat k) \right]
    \cdot \nu (\xi) \, \hat\phi (t,\xi) \d\xi \d{t}
    \\
    \label{eq:inizio}
    = \ & \int_I \!\int_\Omega \Bigl[ \modulo{v (t,x) - \hat k}
    \, \partial_t \hat\phi (t,x)
    \\
    \nonumber & \quad + \sgn\left[v (t,x) - \hat k\right] \left[
      f\left(t,x,v (t,x) + z (t,x)\right) - f \left(t,x,\hat k+ z
        (t,x)\right) \right] \cdot \grad \hat\phi (t,x)
    \\
    \nonumber & \quad + \sgn\left[v (t,x) - \hat k\right] \left[
      F\left(t,x,v (t,x)+z (t,x)\right) -
      \partial_t z (t,x) - \div f \left(t,x,\hat k+ z (t,x)\right)
    \right]
    \\
    \nonumber & \qquad\qquad \times \hat\phi (t,x) \Bigr] \d{x} \d{t}
    \\
    \nonumber & + \int_\Omega \modulo{u_o (x) - \hat k} \hat\phi (0,x)
    \d{x}
    \\
    \nonumber & - \int_I \int_{\partial\Omega} \sgn(-\hat k) \left[
      f\left(t,\xi, \tr v (t,\xi) + z (t,\xi)\right) - f
      \left(t,\xi,\hat k + z (t,\xi) \right) \right]
    \\
    \label{eq:fine}
    & \qquad\qquad \cdot \nu (\xi) \, \hat\phi (t,\xi) \d\xi \d{t}.
  \end{align}

  We now verify that the map
  \begin{displaymath}
    u (t,x) = v (t,x) + z (t,x)
  \end{displaymath}
  is a solution to~\eqref{eq:1} in the sense of
  Definition~\ref{def:sol}. To this aim, we suitably modify the
  doubling of variables technique by Kru\v zkov,
  see~\cite{Kruzkov}. Let $\check k = k - v (t,x)$ in~\eqref{eq:solZ}
  and $\hat k = k -z (s,y)$ in~\eqref{eq:inizio}--\eqref{eq:fine} for
  $k \in \reali$. Integrate~\eqref{eq:solZ} with respect to $t$ and
  $x$ over $I \times \Omega$,
  integrate~\eqref{eq:inizio}--\eqref{eq:fine} in $s$ and $y$ over $I
  \times \Omega$.  Add the resulting expressions, with as test
  function the map $\psi_h = \psi_h (t,x,s,y)$ defined by
  \begin{equation}
    \label{eq:31}
    \psi_h (t,x,s,y) =
    \phi \left(\frac{t+s}{2}, x\right) \,
    Y_h (t-s) \, \prod_{i=1}^n Y_h (x_i-y_i),
  \end{equation}
  with $\phi \in \Cc2 (\left]-\infty, T\right[ \times \reali^n;
  \reali^+)$ and $Y_h$ defined as follows.  Let $Y \in \Cc{\infty}
  (\reali; \reali^+)$ be such
  that\\
  \begin{minipage}{0.30\linewidth}
    \begin{center}
      \includegraphics[width=\textwidth]{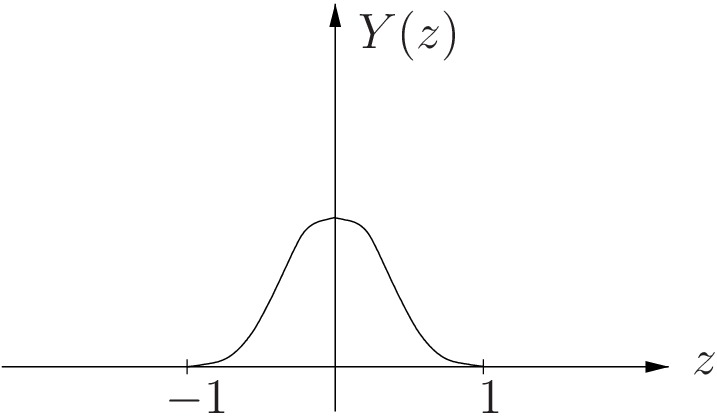}
    \end{center}
  \end{minipage}%
  \begin{minipage}{0.30\linewidth}
    \includegraphics[width=\textwidth]{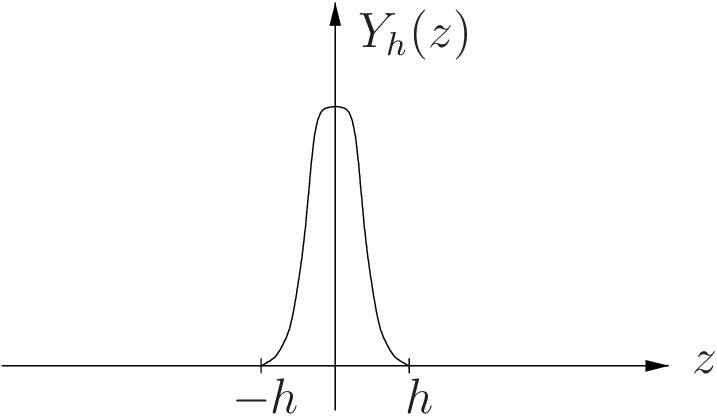}
  \end{minipage}%
  \begin{minipage}[b]{0.40\linewidth}
    \begin{equation}
      \label{eq:Yh}
      \begin{array}{@{}r@{\,}c@{\,}l@{}}
        Y (-z) & = & Y (z)
        \\[3pt]
        Y (z) & = & 0 \mbox{ for } \modulo{z} \geq 1
        \\[3pt]
        \int_{\reali} Y (z) \d{z} & = & 1 \,.
      \end{array}
    \end{equation}
  \end{minipage}\\

  \noindent and define $ Y_h (z) = \frac{1}{h} \, Y
  \left(\frac{z}{h}\right)$. Obviously, $Y_h \in \Cc{\infty} (\reali;
  \reali^+)$, $Y_h (-z) = Y_h (z)$, $Y_h (z) = 0$ for $\modulo{z} \geq
  h$, $\int_{\reali} Y_h (z) \d{z} = 1$ and $Y_h \to \delta_0$ as $h
  \to 0$, where $\delta_0$ is the Dirac delta in $0$.

  We temporarily require also that
  \begin{equation}
    \label{eq:hStar}
    h \in \left]0, h_*\right[
    \quad \mbox{ and } \quad
    \phi (t,x) = 0 \quad \mbox{ for all } \quad
    x \quad \mbox{ such that } \quad B (x,h_*) \cap (\reali^n \setminus \Omega) \neq \emptyset
  \end{equation}
  for a fixed positive $h_*$. We therefore obtain:
  \begin{align}
    \label{eq:ss0}
    0 \leq \ & \int_I \int_\Omega \int_I \int_\Omega \Bigl[ \modulo{v
      (t,x) + z(s,y) - k} \left(
      \partial_t \psi_h (t,x,s,y) +
      \partial_s \psi_h (t,x,s,y) \right)
    \\
    \label{eq:ss1}
    & \quad + \sgn\left[v (t,x) + z(s,y) - k\right]
    \\
    \label{eq:ss2}
    & \quad \times \left[ f\left(t,x,u (t,x)\right) - f
      \left(t,x,z(t,x) - z (s,y) + k\right) \right] \cdot \grad_x
    \psi_h (t,x,s,y)
    \\
    \label{eq:ss3}
    & \quad + \sgn\left[v (t,x) + z(s,y) - k\right] \left(\partial_s z
      (s,y) - \partial_t z (t,x)\right) \psi_h (t,x,s,y)
    \\
    \label{eq:ss4}
    & \quad - \sgn\left[v (t,x) + z(s,y) - k\right] \div f
    \left(t,x,z(t,x) - z (s,y) + k\right) \psi_h (t,x,s,y)
    \\
    \label{eq:ss3bis}
    & \quad + \sgn\left[v (t,x) + z(s,y) - k\right] F\left(t,x,u
      (t,x)\right) \psi_h (t,x,s,y) \Bigr] \d{x} \d{t} \d{y} \d{s}
    \\
    \label{eq:ss5}
    & + \int_I \int_\Omega \int_\Omega \modulo{u_o (x) + z(s,y) - k}
    \psi_h (0,x,s,y) \d{x} \d{y} \d{s}
    \\
    \label{eq:ss6}
    & + \int_I \int_\Omega \int_\Omega \modulo{v (t,x) - k} \psi_h
    (t,x,0,y) \d{x} \d{y} \d{t} \,.
  \end{align}
  To compute the limit as $h \to 0$, consider the terms above
  separately. First, proceeding as
  in~\cite[Formul\ae~(3.5)--(3.7)]{Kruzkov}, thanks to~\eqref{eq:31},
  we have
  \begin{equation}
    \label{eq:29}
    \lim_{h\to 0} \mbox{\eqref{eq:ss0}}
    =
    \int_I \int_\Omega \modulo{u (t,x) - k}
    \, \partial_t \phi (t,x) \, \d{x} \d{t} .
  \end{equation}

  To deal with~\eqref{eq:ss1}--\eqref{eq:ss2} we simplify the notation
  by introducing the map
  \begin{displaymath}
    \Upsilon (t,x,s,y)
    =
    \sgn\left[v (t,x) + z(s,y) - k\right]
    \left[
      f\left(t,x,u (t,x)\right)  - f \left(t,x,z(t,x) - z (s,y) + k\right)
    \right],
  \end{displaymath}
  so that
  \begin{align}
    & \Upsilon (t,x,s,y) \cdot \grad_x \psi_h (t,x,s,y)
    \\
    \label{eq:uffa1}
    = \ & \Upsilon (t,x,t,x) \cdot \grad\phi (t,x) \; Y_h (t-s) \,
    \prod_{j=1}^n Y_h(x_j-y_j)
    \\
    \label{eq:uffa2}
    & + \Upsilon (t,x,t,x) \cdot \left( \grad \phi\left(\frac{t+s}{2},
        x\right) - \grad\phi (t,x) \right) Y_h (t-s) \, \prod_{j=1}^n
    Y_h(x_j-y_j)
    \\
    \label{eq:uffa3}
    & + \left( \Upsilon (t,x,s,y) - \Upsilon (t,x,t,x) \right) \cdot
    \grad \phi\left(\frac{t+s}{2}, x\right) Y_h (t-s) \, \prod_{j=1}^n
    Y_h(x_j-y_j)
    \\
    \label{eq:uffa4}
    & + \sum_{i=1}^n \Upsilon_i (t,x,t,x) \, \phi\left(\frac{t+s}{2},
      x\right) \, Y_h (t-s) \, Y'_h (x_i-y_i) \, \prod_{j\neq i} Y_h
    (x_j-y_j)
    \\
    \label{eq:uffa5}
    & + \sum_{i=1}^n \left[ \Upsilon_i (t,x,s,y) - \Upsilon_i
      (t,x,t,x) \right] \!  \phi\!\left(\frac{t+s}{2}, x\right) \!
    Y_h (t-s) Y'_h (x_i-y_i) \prod_{j\neq i} \! Y_h (x_j-y_j).
  \end{align}
  Then,
  \begin{eqnarray*}
    \int_I \int_\Omega \int_I \int_\Omega
    \left[\eqref{eq:uffa1}\right]
    \d{y} \d{s} \d{x} \d{t}
    =
    \int_I \int_\Omega \Upsilon (t,x,t,x) \cdot \grad \phi (t,x) \d{x} \d{t} \,.
  \end{eqnarray*}
  To deal with~\eqref{eq:uffa2}, recall that $\modulo{Y_h} \leq
  \left(Y (0) /h\right) \chi_{\strut[-h,h]}$ and apply
  Lemma~\ref{lem:kru2} with
  \begin{displaymath}
    N = 2n+1\,,\qquad
    \begin{array}{r@{\,}c@{\,}l}
      X & = & (x,t,x),
      \\
      Y & = & (x,t,y),
    \end{array}
    \qquad
    w (s,Y)
    =
    \frac{Y (0)^{n+1}}{h^{n+1}} \;
    \Upsilon (t,x,t,x) \cdot
    \grad \phi\left(\frac{t+s}{2}, x\right),
  \end{displaymath}
  so that
  \begin{displaymath}
    \lim_{h\to 0}
    \int_I \int_\Omega \int_I \int_\Omega
    \modulo{\eqref{eq:uffa2}}
    \d{y} \d{s} \d{x} \d{t}
    =0 \,.
  \end{displaymath}
  Similarly, to deal with~\eqref{eq:uffa3}, apply Lemma~\ref{lem:kru2}
  with
  \begin{displaymath}
    N = 2n+1\,,\qquad
    \begin{array}{r@{\,}c@{\,}l}
      X & = & (x,t,x),
      \\
      Y & = & (x,t,y),
    \end{array}
    \qquad
    w (s,Y)
    =
    \frac{Y (0)^{n+1}\;
      \norma{\grad\phi}_{\L\infty (I\times\reali^n;\reali^n)}}{h^{n+1}}
    \;
    \Upsilon (t,x,s,y),
  \end{displaymath}
  so that
  \begin{displaymath}
    \lim_{h\to 0}
    \int_I \int_\Omega \int_I \int_\Omega
    \modulo{[\eqref{eq:uffa3}]}
    \d{y} \d{s} \d{x} \d{t}
    =0 \,.
  \end{displaymath}
  The term~\eqref{eq:uffa4} vanishes, since
  \begin{displaymath}
    \int_I \int_\Omega \int_I \int_\Omega
    [\eqref{eq:uffa4}]
    \d{y} \d{s} \d{x} \d{t}
    = \int \cdots \int_{x_i-h}^{x_i+h} Y'_h (x_i-y_i) \d{y_i} \cdots \d{x} \d{t}
    =
    0 \,.
  \end{displaymath}
  Finally, to estimate~\eqref{eq:uffa5}, recall that $\modulo{Y'_h}
  \leq \left(\norma{Y'}_{\L\infty (\reali;\reali)} / h^2\right)
  \chi_{\strut[-h,h]}$ and use Lemma~\ref{lem:kru2} with
  \begin{displaymath}
    N = 2n+1\,,\qquad
    \begin{array}{r@{\,}c@{\,}l}
      X & = & (x,t,x),
      \\
      Y & = & (x,t,y),
    \end{array}
    \qquad
    w (s,Y)
    =
    \frac{Y (0)^{n}
      \, \norma{Y'}_{\L\infty (\reali;\reali)}
      \,\norma{\phi}_{\L\infty (I\times\reali^n;\reali)}}{h^{n+2}} \;
    \Upsilon (t,x,s,y),
  \end{displaymath}
  so that, thanks to $2n+1 \geq n+2$,
  \begin{displaymath}
    \lim_{h \to 0} \int_I \int_\Omega \int_I \int_\Omega
    \modulo{\eqref{eq:uffa5}}
    \d{y} \d{s} \d{x} \d{t}
    =
    0\,.
  \end{displaymath}
  Hence
  \begin{eqnarray}
    \nonumber
    & &
    \lim_{h \to 0} \left[\eqref{eq:ss1} \times \eqref{eq:ss2}\right]
    \\
    \nonumber
    & = &
    \lim_{h \to 0}
    \int_I \int_{\Omega} \int_I \int_{\Omega}
    \Upsilon (t,x,s,y) \cdot
    \grad_x \psi_h (t,x,s,y)
    \d{y} \d{s} \d{x} \d{t}
    \\
    \nonumber
    & = &
    \int_I \int_{\Omega}
    \Upsilon (t,x,t,x) \cdot \grad \phi (t,x)
    \d{x} \d{t}
    \\
    & = &
    \label{eq:e1}
    \int_I \int_{\Omega}
    \sgn\left(u (t,x) - k\right)
    \left(
      f\left(t,x,u (t,x)\right)
      -
      f (t,x,k)\right)
    \cdot \grad \phi (t,x)
    \d{x} \d{t} \,.
  \end{eqnarray}
  Note that setting
  \begin{displaymath}
    N = n \,,\quad
    \begin{array}{@{}r@{}c@{}l@{}}
      X & = &x
      \\
      Y & = & y
    \end{array}
    \quad
    \mbox{ and } \quad
    w (s,y)
    =
    \frac{Y (0)^{n+1}}{h^{n+1}} \, \norma{\phi}_{\L\infty (I\times\reali^n; \reali)}
    \,
    \partial_s z (s,y)
  \end{displaymath}
  in Lemma~\ref{lem:kru2}, we obtain
  \begin{displaymath}
    \lim_{h\to 0} \left[\eqref{eq:ss3}\right]
    =
    \lim_{h\to 0}
    \int_I \int_\Omega \int_I \int_\Omega
    \modulo{ \partial_t z (t,x) - \partial_s z (s,y)}
    \psi_h (t,x,s,y)
    \d{y} \d{s} \d{x} \d{t}
    = 0.
  \end{displaymath}
  Omitting now the integrals in~\eqref{eq:ss4}, we have
  \begin{eqnarray*}
    \left[\eqref{eq:ss4}\right]
    & = &
    -
    \sgn\left[v (t,x) + z(s,y) - k\right]
    \div f \left(t,x,z(t,x) - z (s,y) + k\right)
    \psi_h (t,x,s,y)
    \\
    & = &
    -
    \sgn\left[u (t,x) - k\right] \;
    \div f (t,x,k) \;
    \phi (t,x) Y_h (t-s) \prod Y_h (x_i-y_i)
    \\
    & &
    -
    \sgn\left[u (t,x) - k\right]
    \div f (t,x,k)
    \left[
      \phi\left(\frac{t+s}{2}, x\right)
      -
      \phi (t,x)
    \right]
    Y_h (t-s) \prod Y_h (x_i-y_i)
    \\
    & &
    -
    \sgn\left[ u (t,x) - k\right]
    \left(
      \div f \left(t,x,z(t,x) - z (s,y) + k\right)
      -
      \div f \left(t,x,k\right)
    \right)
    \\
    & &
    \qquad
    \times
    \phi\left(\frac{t+s}{2}, x\right)
    Y_h (t-s) \prod Y_h (x_i-y_i)
    \\
    & &
    -
    \left(
      \sgn\left[ u (t,x) +z (s,y) - z (t,x) - k\right]
      -
      \sgn\left[ u (t,x) - k\right]
    \right)
    \\
    & &
    \qquad
    \times
    \div f \left(t,x,z(t,x) - z (s,y) + k\right)
    \phi\left(\frac{t+s}{2}, x\right)
    Y_h (t-s) \prod Y_h (x_i-y_i) \,.
  \end{eqnarray*}
  A repeated application of Lemma~\ref{lem:kru2}, together with
  standard estimates, yields
  \begin{equation}
    \label{eq:e2}
    \lim_{h\to0} \left[\eqref{eq:ss4}\right]
    =
    -
    \int_I \int_\Omega
    \sgn\left[u (t,x) - k\right] \;
    \div f (t,x,k) \;
    \phi (t,x)
    \d{x} \d{t} \,.
  \end{equation}
  The term~\eqref{eq:ss3bis} is treated similarly, since
  \begin{eqnarray*}
    \left[\eqref{eq:ss3bis}\right]
    & = &
    \sgn\left[v (t,x) + z(s,y) - k\right]
    F\left(t,x,u (t,x)\right)
    \psi_h (t,x,s,y)
    \\
    & = &
    \sgn\left[u (t,x) - k\right]
    F\left(t,x,u (t,x)\right)
    \phi (t,x)
    Y_h (t-s) \prod Y_h (x_i-y_i)
    \\
    & &
    +
    \sgn\left[u (t,x) - k\right]
    F\left(t,x,u (t,x)\right)
    \left[
      \phi\left(\frac{t+s}{2},x\right)
      -
      \phi (t,x)
    \right]
    Y_h (t-s) \prod Y_h (x_i-y_i)
    \\
    & &
    +
    \left(
      \sgn\left[u (t,x) +z (s,y) - z (t,x) - k\right]
      -
      \sgn\left[u (t,x) - k\right]
    \right)
    F\left(t,x,u (t,x)\right)
    \\
    & &
    \quad \times
    \phi\left(\frac{t+s}{2},x\right)
    Y_h (t-s) \prod Y_h (x_i-y_i) \,,
  \end{eqnarray*}
  so that further applications of Lemma~\ref{lem:kru2} lead to
  \begin{equation}
    \label{eq:e3}
    \lim_{h\to0} \left[\eqref{eq:ss3bis}\right]
    =
    \int_I \int_\Omega
    \sgn\left[u (t,x) - k\right] \;
    F \left(t,x,u (t,x)\right) \;
    \phi (t,x)
    \d{x} \d{t}  \,.
  \end{equation}
  To deal with~\eqref{eq:ss5} and~\eqref{eq:ss6}, introduce the
  function
  \begin{displaymath}
    \Upsilon (x,s,y)
    =
    \modulo{u_o (x) + z (s,y) - k}
    +
    \modulo{v (s,x) - k}
  \end{displaymath}
  and, exploiting the symmetry $Y (x) = Y (-x)$, we obtain
  \begin{align*}
    & \left[\eqref{eq:ss5} + \eqref{eq:ss6}\right]
    \\
    = \ & \int_I \int_\Omega \int_\Omega \Upsilon (x,s,y) \;\psi_h
    (0,x,s,y) \d{y} \d{x} \d{s}
    \\
    = \ & \int_I \int_\Omega \int_\Omega \Upsilon (x,0,x) \, \phi(0,
    x) \, Y_h (s) \, \prod Y_h (x_i-y_i) \d{y} \d{x} \d{s}
    \\
    & + \int_I \int_\Omega \int_\Omega \Upsilon (x,0,x) \left(
      \phi\left(\frac{s}{2}, x\right) - \phi (0,x) \right) Y_h (s) \,
    \prod Y_h (x_i-y_i) \d{y} \d{x} \d{s}
    \\
    & + \int_I \int_\Omega \int_\Omega \left( \Upsilon (x,s,y) -
      \Upsilon (x,0,x) \right) \phi\left(\frac{s}{2}, x\right) Y_h (s)
    \, \prod Y_h (x_i-y_i) \d{y} \d{x} \d{s}
    \\
    \leq \ & \frac{1}{2} \int_\Omega \Upsilon (x,0,x) \;\phi (0,x)
    \d{x}
    \\
    & + \frac{\norma{\partial_t\phi}_{\L\infty (\reali\times\reali^n;
        \reali)}}{2} \int_\Omega \Upsilon (x,0,x) \d{x} \int_0^h s \,
    Y_h (s) \d{s}
    \\
    & + \norma{\phi}_{\L\infty (\reali\times\reali^N; \reali)} \int_I
    \int_\Omega \int_\Omega \modulo{\Upsilon (x,s,y) - \Upsilon
      (x,0,x)} Y_h (s) \, \prod Y_h (x_i-y_i) \d{y} \d{x} \d{s}.
  \end{align*}
  Both the two latter terms vanish in the limit $h\to 0$. Indeed, by
  Lemma~\ref{lem:kru2}, for a.e.~$s \in [0,h]$, we have that
  $\int_\Omega \int_\Omega \modulo{\Upsilon (x,s,y) - \Upsilon
    (x,0,x)} \d{y} \d{x} \to 0$. Hence,
  \begin{equation}
    \label{eq:e4}
    \lim_{h\to 0}
    \left[\eqref{eq:ss5} + \eqref{eq:ss6}\right]
    =
    \int_\Omega \modulo{u_o (x) - k} \; \phi (0,x) \, \d{x} \,.
  \end{equation}

  We can now summarize the computations: thanks to~\eqref{eq:29},
  \eqref{eq:e1}, \eqref{eq:e2}, \eqref{eq:e3} and~\eqref{eq:e4}, in
  the limit $h \to 0$ \eqref{eq:ss0}--\eqref{eq:ss6} becomes
  \begin{align*}
    \int_I \int_\Omega \modulo{u (t,x) - k} \, \partial_t \phi (t,x)
    \, \d{x} \d{t} &
    \\
    + \int_I \int_{\Omega} \sgn\left(u (t,x) - k\right) \left(
      f\left(t,x,u (t,x)\right) - f (t,x,k)\right) \cdot \grad \phi
    (t,x) \d{x} \d{t} &
    \\
    + \int_I \int_\Omega \sgn\left(u (t,x) - k\right) \left( F
      \left(t,x,u (t,x)\right) - \div f (t,x,k) \right) \phi (t,x)
    \d{x} \d{t} &
    \\
    + \int_\Omega \modulo{u_o (x) - k} \; \phi (0,x) \, \d{x} & \geq
    \,0,
  \end{align*}
  which holds under the choice~\eqref{eq:hStar} of $\phi$. To pass to
  an arbitrary test function as in Definition~\ref{def:sol},
  substitute $\phi (t,x)$ with $\left(1-\Phi_h (x)\right)\phi (t,x)$,
  where $\phi \in \Cc2 (\left]-\infty, T\right] \times \reali^n;
  \reali^+)$ and $\Phi_h$ is as in~\eqref{eq:11}:
  \begin{align*}
    \int_I \int_\Omega \modulo{u (t,x) - k} \, \partial_t \phi (t,x)
    \left(1-\Phi_h (x)\right) \d{x} \d{t} %
    \\
    + \int_I \int_{\Omega} \sgn\left(u (t,x) - k\right) \left(
      f\left(t,x,u (t,x)\right) - f (t,x,k)\right) \cdot \grad \phi
    (t,x) \left(1-\Phi_h (x)\right) \d{x} \d{t} &
    \\
    + \int_I \int_\Omega \sgn\left(u (t,x) - k\right) \left( F
      \left(t,x,u (t,x)\right) - \div f (t,x,k) \right) \phi (t,x)
    \left(1-\Phi_h (x)\right) \d{x} \d{t} &
    \\
    + \int_\Omega \modulo{u_o (x) - k} \; \phi (0,x) \left(1-\Phi_h
      (x)\right)\d{x} &
    \\
    - \int_I \int_\Omega \sgn\left(u (t,x) - k\right) \left(
      f\left(t,x,u (t,x)\right) - f (t,x,k)\right) \cdot \grad \Phi_h
    (x) \; \phi (t,x) \d{x} \d{t} & \geq 0.
  \end{align*}
  In the limit $h \to 0$, the first 4 lines above converge to the
  first 3 lines in the left hand side in~\eqref{eq:4} of
  Definition~\ref{def:sol}, by the Dominated Convergence
  Theorem. Concerning the latter term, use Lemma~\ref{lem:uffaMultiD}
  and Lemma~\ref{lem:tr3}, which can be applied since the function
  $(w_1,w_2) \to \sgn (w_1-w_2) \left(f (t,x,w_1) - f
    (t,x,w_2)\right)$ is Lipschitz continuous,
  see~\cite[Lemma~3]{Kruzkov}. We therefore obtain that
  \begin{eqnarray*}
    & &
    -
    \lim_{h\to 0}
    \int_I \int_\Omega
    \sgn\left(u (t,x) - k\right)
    \left(
      f\left(t,x,u (t,x)\right)
      -
      f (t,x,k)\right)
    \cdot \grad \Phi_h (x) \; \phi (t,x)
    \d{x} \d{t}
    \\
    & = &
    -
    \int_I \int_{\partial\Omega}
    \sgn\left(\tr u (t,\xi) - k\right)
    \left(
      f\left(t,\xi,\tr u (t,\xi)\right)
      -
      f (t,\xi,k)\right)
    \cdot \nu (\xi) \, \phi (t,\xi)
    \d{x} \d{t}
    \\
    & = &
    -
    \int_I \int_{\partial\Omega}
    \sgn\left( u_b (t,\xi) - k\right)
    \left(
      f\left(t,\xi,\tr u (t,\xi)\right)
      -
      f (t,\xi,k)\right)
    \cdot \nu (\xi) \, \phi (t,\xi)
    \d{x} \d{t}
    \\
    & &
    -
    \int_I \int_{\partial\Omega}
    \left(
      \sgn\left(\tr u (t,\xi) - k\right)
      -
      \sgn\left( u_b (t,\xi) - k\right)
    \right)
    \\
    & &
    \qquad\qquad
    \times
    \left(
      f\left(t,\xi,\tr u (t,\xi)\right)
      -
      f (t,\xi,k)\right)
    \cdot \nu (\xi) \, \phi (t,\xi)
    \d{x} \d{t}
    \\
    & \leq &
    -
    \int_I \int_{\partial\Omega}
    \left(
      \sgn\left(\tr u (t,\xi) - k\right)
      -
      \sgn\left( u_b (t,\xi) - k\right)
    \right),
  \end{eqnarray*}
  where to get to the last line, we used the following fact:
  \begin{eqnarray*}
    & &
    -
    \int_I \int_{\partial\Omega}
    \left(
      \sgn\left(\tr u (t,\xi) - k\right)
      -
      \sgn\left( u_b (t,\xi) - k\right)
    \right)
    \\
    & &
    \qquad\qquad
    \times
    \left(
      f\left(t,\xi,\tr u (t,\xi)\right)
      -
      f (t,\xi,k)\right)
    \cdot \nu (\xi) \, \phi (t,\xi)
    \d{x} \d{t}
    \\
    & = &
    -
    \int_I \int_{\partial\Omega}
    \left(
      \sgn\left(\tr v (t,\xi) +z (t,\xi) - k\right)
      -
      \sgn\left( z (t,\xi) - k\right)
    \right)
    \\
    & &
    \qquad\qquad
    \times
    \left(
      f\left(t,\xi,\tr v (t,\xi) + z (t,\xi)\right)
      -
      f (t,\xi,k)\right)
    \cdot \nu (\xi) \, \phi (t,\xi)
    \d{x} \d{t}
    \\
    & = &
    -
    \int_I \int_{\partial\Omega}
    \left(
      \sgn\left(\tr v (t,\xi) - \left(k-z (t,\xi)\right)\right)
      -
      \sgn\left( - \left(k-z (t,\xi)\right) \right)
    \right)
    \\
    & &
    \qquad\qquad
    \times
    \left(
      g\left(t,\xi,\tr v (t,\xi)\right)
      -
      g (t,\xi,k-z (t,\xi))\right)
    \cdot \nu (\xi) \, \phi (t,\xi)
    \d{x} \d{t}
    \\
    & \leq &
    0,
  \end{eqnarray*}
  since $\phi\geq 0$ and by~\eqref{eq:6} in Proposition~\ref{prop:min}
  applied to $v$ as solution to~\eqref{eq:23}
  \begin{displaymath}
    \left(
      \sgn\left(\tr v (t,\xi) - \hat k\right)
      -
      \sgn\left( - \hat k\right) \right)
    \left(
      g\left(t,\xi,\tr v (t,\xi)\right)
      -
      g (t,\xi,\hat k)\right)
    \cdot \nu (\xi)
    \geq
    0
  \end{displaymath}
  for all $\hat k\in \reali$ and for a.e.~$(t,\xi) \in I
  \times \partial\Omega$. This completes the first part of the proof:
  the existence of a solution to~\eqref{eq:1} in the sense of
  Definition~\ref{def:sol}.

  Consider now the $\L\infty$ estimate. Recall~\eqref{eq:23},
  \eqref{eq:c1c2} and Proposition~\ref{prop:limiteZero}, so that
  \begin{eqnarray*}
    \norma{u (t)}_{\L\infty (\Omega; \reali)}
    & \leq &
    \norma{v (t)}_{\L\infty (\Omega;\reali)}
    +
    \norma{z (t)}_{\L\infty (\Omega;\reali)}
    \\
    & \leq &
    \norma{u_o}_{\L\infty (\Omega;\reali)} e^{c_1 t}
    +
    \frac{c_2 + \norma{\partial_t z}_{\L\infty ([0,t]\times\Omega;\reali)}}{c_1}
    \left(e^{c_1 t}-1\right)
    \\
    & &
    +
    \norma{z}_{\L\infty ([0,t]\times\Omega;\reali)}
    \left(e^{c_1 t}-1\right)
    +
    \norma{z (t)}_{\L\infty (\Omega;\reali)}
    \\
    & \leq & M_u (t).
  \end{eqnarray*}
  Using also Lemma~\ref{lem:elliptic}, we obtain
  \begin{displaymath}
    M_u (t)
    =
    \left(
      \norma{u_o}_{\L\infty (\Omega;\reali)}
      +
      \norma{u_b}_{\L\infty (0,t]\times\partial\Omega)}
    \right)
    e^{c_1 t}
    +
    \frac{c_2 + \norma{\partial_t u_b}_{\L\infty ([0,t]\times\partial\Omega;\reali)}}{c_1}
    \left(e^{c_1 t}-1\right),
  \end{displaymath}
  which proves the $\L\infty$ estimate~\eqref{eq:L_bound}.

  To obtain the $\tv$ bound, we use Proposition~\ref{prop:limiteZero}
  to estimate $\tv (v)$ and standard estimates on elliptic problems to
  bound $\tv (z)$. To this aim, we call $A_i (g)$, for $i=1, \ldots,
  4$, the quantities defined in~\eqref{eq:20}, but with norms of $g$
  and $G$ over $[0,t] \times \Omega \times \mathcal{V} (t)$, where
  $\mathcal{V} (t) = [-\mathcal{M}_v (t), \mathcal{M}_v (t)]$, with
  $\mathcal{M}_v (t)$ being an upper bound for $\norma{v}_{\L\infty
    ([0,t]\times\Omega;\reali)}$ as in~\eqref{eq:Ut}. Clearly,
  $\mathcal{V} (t) \subseteq \mathcal{U} (t) = [-M_u (t), M_u
  (t)]$. By~\eqref{eq:20}, and Lemma~\ref{lem:elliptic}, we have:
  \begin{eqnarray*}
    A_1 (g)
    \!\! \!\! & \leq & \!\! \!\!
    \O \Big[
    \norma{\div f}_{\L\infty ([0,t]\times\Omega \times \mathcal{U} (t); \reali^n)}
    +
    \norma{F}_{\L\infty ([0,t]\times\Omega \times \mathcal{U} (t); \reali)}
    \\
    & &
    \qquad\qquad
    +
    \norma{\partial_u f}_{\L\infty ([0,t]\times\Omega\times \mathcal{U} (t);\reali^n)}
    \,
    \norma{\grad z}_{\L\infty ([0,t]\times\Omega;\reali^n)}
    +
    \norma{\partial_t z}_{\L\infty ([0,t]\times\Omega;\reali)}
    \Big]
    \\
    \!\! \!\! & \leq & \!\! \!\!
    \O \Big[
    \norma{D f}_{\L\infty ([0,t]\times\Omega \times \mathcal{U} (t); \reali^{n\times (2+n)})}
    +
    \norma{F}_{\L\infty ([0,t]\times\Omega \times \mathcal{U} (t); \reali)}
    \\
    & &
    \qquad\qquad
    +
    \left(
      1
      +
      \norma{D f}_{\L\infty ([0,t]\times\Omega\times \mathcal{U} (t);\reali^{n\times (2+n)})}
    \right)
    \norma{u_b}_{\C{2,\alpha} ([0,t]\times\partial\Omega;\reali)}
    \Big]
    \qquad\qquad =:\mathcal{A}_1
    \\
    A_2 (g)
    \!\! \!\! & \leq & \!\! \!\!
    \O
    \Big[
    \norma{D f}_{\W1\infty ([0,t]\times\Omega \times \mathcal{U} (t); \reali^{n\times (2+n)})}
    +
    \norma{F}_{\W1\infty ([0,t]\times\Omega\times \mathcal{U} (t); \reali)}
    \\
    & &
    + \!
    \left[
      1
      +
      \norma{D f}_{\W1\infty ([0,t]\times\Omega\times \mathcal{U} (t);\reali^{n\times (2+n)}}
      +
      \norma{\partial_u F}_{\W1\infty ([0,t]\times\Omega\times \mathcal{U} (t);\reali)}
    \right] \!\!
    \norma{D z}_{\W1\infty ([0,t]\times\Omega;\reali)}
    \\
    & &
    +
    \norma{\partial^2_{uu} f}_{\L\infty ([0,t]\times \Omega \times \mathcal{U} (t); \reali^n)}
    \norma{D z}_{\W1\infty ([0,t]\times\Omega;\reali)}^2
    \Big]
    \\
    \!\! \!\! & \leq & \!\! \!\!
    \O
    \Big[
    \norma{D f}_{\W1\infty ([0,t]\times\Omega \times \mathcal{U} (t); \reali^{n\times (2+n)})}
    +
    \norma{F}_{\W1\infty ([0,t]\times\Omega\times \mathcal{U} (t); \reali)}
    \\
    & &
    +
    \left[
      1
      +
      \norma{D f}_{\W1\infty ([0,t]\times\Omega\times \mathcal{U} (t);\reali^{n\times (2+n)}}
      +
      \norma{\partial_u F}_{\W1\infty ([0,t]\times\Omega\times \mathcal{U} (t);\reali)}
    \right]
    \norma{u_b}_{\C{3,\alpha} ([0,t]\times\partial\Omega;\reali)}
    \\
    & &
    +
    \norma{\partial^2_{uu} f}_{\L\infty ([0,t]\times \Omega \times \mathcal{U} (t); \reali^n)}
    \norma{u_b}_{\C{3,\alpha} ([0,t]\times\partial\Omega;\reali)}^2
    \Big]
    \qquad\qquad\qquad\qquad\qquad\quad\qquad =:\mathcal{A}_2
    \\
    A_3 (g)
    \!\! \!\! & \leq & \!\! \!\!
    \O + \norma{\partial_u f}_{\L\infty ([0,t]\times\Omega \times \mathcal{U} (t); \reali^n)}
    \qquad\qquad\qquad\qquad\qquad\qquad\qquad\qquad\qquad\qquad\;\,
    =:\mathcal{A}_3
    \\
    A_4 (g)
    \!\! \!\! & \leq & \!\! \!\!
    \O \Big[
    1
    +
    \norma{Df}_{\W1\infty([0,t]\times\Omega\times \mathcal{U} (t);\reali^{n\times n})}
    +
    \norma{\partial_u F}_{\L\infty([0,t]\times\Omega\times \mathcal{U} (t);\reali)}
    \\
    & &
    \qquad
    +
    \norma{\partial^2_{uu}f}_{\L\infty([0,t]\times\Omega\times \mathcal{U} (t);\reali^n)}
    \left(
      \norma{\partial_t z}_{\L\infty ([0,t]\times \Omega)}
      +
      \norma{\grad z}_{\L\infty ([0,t]\times \Omega)}
    \right)
    \Big]
    \\
    \!\! \!\! & \leq & \!\! \!\!
    \O \Big[
    1
    +
    \norma{Df}_{\W1\infty([0,t]\times\Omega\times \mathcal{U} (t);\reali^{n\times n})}
    +
    \norma{\partial_u F}_{\L\infty([0,t]\times\Omega\times \mathcal{U} (t);\reali)}
    \\
    & &
    \qquad
    +
    \norma{\partial^2_{uu}f}_{\L\infty([0,t]\times\Omega\times \mathcal{U} (t);\reali^n)}
    \norma{u_b}_{\C{2,\alpha} ([0,t]\times\partial\Omega;\reali)}
    \Big]
    \qquad\qquad\qquad\qquad\qquad\quad =:\mathcal{A}_4
  \end{eqnarray*}
  which proves the bound
  \begin{equation}
    \label{eq:finita}
    \tv\left(v (t)\right)
    \leq
    \left(
      \mathcal{A}_1
      +
      \mathcal{A}_2 \, t
      +
      \mathcal{A}_3 \tv (u_o)
    \right)
    e^{\mathcal{A}_4 t} \,.
  \end{equation}
  Recall now that $\tv (u) \leq \tv (v) + \tv (z)$ and, by
  Lemma~\ref{lem:elliptic}, $\tv (z) \leq \mathcal{L}^n (\Omega)
  \norma{u_b}_{\C{2,\alpha} ([0,t]\times\partial\Omega; \reali)}$.

  The proof is completed.
\end{proofof}

\begin{proofof}{Theorem~\ref{thm:wp}}
  Assume preliminarily that $u_o \in \C2 (\bar{\Omega} ; \reali)$ and
  $u_{b} \in \C2 (I \times \partial \Omega; \reali)$.

  Let $\phi \in \Cc2 (\left]-\infty, T\right[ \times \reali^n;
  \reali^+)$ be a test function as in Definition~\ref{def:sol} with
  \begin{equation}
    \label{eq:8}
    \begin{array}{rclcr@{\;}c@{\;}l}
      \phi (0,x) & = & 0 & \mbox{ for all } &
      x & \in & \reali^n \,,
      \\
      \phi (t,\xi) & = & 0 & \mbox{ for all } &
      (t,\xi) & \in & I \times \partial\Omega \,.
    \end{array}
  \end{equation}
  Define
  \begin{equation}
    \label{eq:psi}
    \psi_h (t,x,s,y) =
    \phi \left(\frac{t+s}{2}, \frac{x+y}{2}\right) \,
    Y_h (t-s) \, \prod_{i=1}^n Y_h (x_i-y_i)
  \end{equation}
  where $Y_h$ is defined in~\eqref{eq:Yh}.  We now use the
  \emph{doubling of variables} method, see~\cite{Kruzkov}.  In
  inequality~\eqref{eq:4}, set $k = v (s,y)$ and use as test function
  the map $\psi_h = \psi_h (t,x,s,y)$ for a fixed point $(s,y)$ and
  integrate over $I \times \Omega$ with respect to $(s,y)$:
  \begin{eqnarray*}
    & &
    \int_I \int_\Omega \int_I \int_\Omega
    \bigl\{ \modulo{u (t,x) - v(s,y)} \partial_t \psi_h (t,x,s,y)
    \\
    & &
    + \sgn\left(u (t,x) - v(s,y)\right) \,
    \left[
      f\left(t,x,u (t,x)\right) - f \left(t,x,v (s,y)\right)
    \right]
    \cdot
    \grad_x \psi_h(t,x,s,y)
    \\
    & &
    +
    \sgn\left(u (t,x) - v(s,y)\right) \,
    \left[
      F \left(t,x,u (t,x)\right) - \div f \left(t,x,v (s,y)\right)
    \right] \,
    \psi_h(t,x,s,y)
    \bigr\} \d{x} \d{t} \d{y} \d{s}
    \\
    & &
    +
    \int_I \int_\Omega \int_\Omega
    \psi_h (0,x,s,y) \,
    \modulo{u_o (x) - v (s,y)} \d{x} \d{y} \d{s}
    \geq 0.
  \end{eqnarray*}
  In the same way, starting from the inequality~\eqref{eq:4} for the
  function $v = v (s,y)$, set $k = u(t,x)$, consider the same test
  function $\psi_h = \psi_h (t,x,s,y)$ and integrate over $I \times
  \Omega$ with respect to $(t,x)$:
  \begin{eqnarray*}
    & &
    \int_I \int_\Omega \int_I \int_\Omega
    \bigl\{ \modulo{v(s,y) - u (t,x)} \partial_s \psi_h(t,x,s,y)
    \\
    & &
    + \sgn\left(v(s,y) - u (t,x)\right) \,
    \left[
      f \left(s,y,v (s,y)\right) - f \left(s,y,u (t,x)\right)
    \right]
    \cdot
    \grad_y \psi_h(t,x,s,y)
    \\
    & &
    +
    \sgn\left(v(s,y) - u (t,x)\right) \,
    \left[
      F \left(s,y,v (s,y)\right) - \div f \left(s,y,u (t,x)\right)
    \right] \,
    \psi_h(t,x,s,y)
    \bigr\} \d{y} \d{s}  \d{x} \d{t}
    \\
    & &
    +
    \int_I \int_\Omega \int_\Omega
    \psi_h (t,x,0,y) \,
    \modulo{v_o (y) - u (t,x)} \d{y} \d{x} \d{t}
    \geq 0.
  \end{eqnarray*}
  Summing the last two inequalities above, we obtain:
  \begin{align}
    \label{eq:ante7} 0 \leq & \int_I \int_\Omega \int_I \int_\Omega
    \bigl\{ \modulo{u (t,x) - v(s,y)} \left( \partial_t \psi_h
      (t,x,s,y) + \partial_s \psi_h (t,x,s,y)\right)
    \\
    \nonumber & + \sgn\left(u (t,x) - v(s,y)\right) \, \left[
      f\left(t,x,u (t,x)\right) - f \left(t,x,v (s,y)\right) \right]
    \cdot \grad_x \psi_h (t,x,s,y)
    \\
    \nonumber & + \sgn\left(v(s,y) - u (t,x)\right) \, \left[ f
      \left(s,y,v (s,y)\right) - f \left(s,y,u (t,x)\right) \right]
    \cdot \grad_y \psi_h (t,x,s,y)
    \\
    \nonumber & + \sgn\left(u (t,x) - v(s,y)\right) \, \bigl[ F
    \left(t,x,u (t,x)\right) - F \left(s,y,v (s,y)\right)
    \\
    \label{eq:7}
    & \quad + \div f \left(s,y,u (t,x)\right) - \div f \left(t,x,v
      (s,y)\right) \bigr] \, \psi_h (t,x,s,y) \bigr\} \d{x} \d{t}
    \d{y} \d{s}
    \\
    \label{eq:via1} & + \int_I \int_\Omega \int_\Omega \psi_h
    (0,x,s,y) \, \modulo{u_o (x) - v (s,y)} \d{x} \d{y} \d{s}
    \\
    \label{eq:via2} & + \int_I \int_\Omega \int_\Omega \psi_h
    (t,x,0,y) \, \modulo{v_o (y) - u (t,x)} \d{y} \d{x} \d{t}.
  \end{align}
  We follow the proof of~\cite[Theorem~1]{Kruzkov}. As $h\to 0$, the
  first integral in the $5$
  lines~$\left[\eqref{eq:ante7}\cdots\eqref{eq:7}\right]$, can be
  treated exactly as in~\cite{Kruzkov}, leading to the following
  analog of~\cite[Formula~(3.12)]{Kruzkov}:
  \begin{eqnarray}
    \nonumber
    & &
    \lim_{h \to 0+}
    \left[\eqref{eq:ante7}\cdots\eqref{eq:7}\right]
    \\
    \nonumber
    & = &
    \int_I \int_\Omega
    \bigl\{ \modulo{u (t,x) - v(t,x)} \,
    \partial_t \phi (t,x)
    \\
    \label{eq:Kruz}
    & &
    \qquad
    +
    \sgn\left(u (t,x) - v(t,x)\right)
    \left[ f \left(t,x,u  (t,x)\right) - f \left(t,x,v (t,x)\right) \right]
    \cdot \grad \phi (t,x)
    \\
    \nonumber
    & &
    \qquad
    +
    \sgn\left(u (t,x) - v(t,x)\right) \, \left[ F
      \left(t,x,u (t,x)\right) - F \left(t,x,v (t,x)\right) \right] \,
    \phi (t,x) \bigr\} \d{x} \d{t}.
  \end{eqnarray}
  To compute the second integral~\eqref{eq:via1}, observe
  preliminarily that
  \begin{eqnarray*}
    \modulo{u_o (x) - v (s,y)}
    & \leq &
    \modulo{u_o (x) - v (s,y)}
    -
    \modulo{u_o (x) - v (s,x)}
    \\
    & &
    +
    \modulo{u_o (x) - v (s,x)}
    -
    \modulo{u_o (x) - v (0+,x)}
    \\
    & &
    +
    \modulo{u_o (x) - v (0+,x)}
    \\
    & \leq &
    \modulo{v (s,y) - v (s,x)}
    +
    \modulo{v (s,x) - v (0+,x)}
    +
    \modulo{u_o (x) - v (0+,x)}\,.
  \end{eqnarray*}
  Hence:
  \begin{eqnarray}
    \label{eq:a}
    \left[\mbox{\eqref{eq:via1}}\right]
    & \leq &
    \int_I \int_\Omega \int_\Omega
    \psi_h (0,x,s,y) \modulo{v (s,y) - v (s,x)} \d{x} \d{y} \d{s}
    \\
    \label{eq:b}
    & &
    +
    \int_I \int_\Omega \int_\Omega
    \psi_h (0,x,s,y) \modulo{v (s,x) - v (0+,x)} \d{x} \d{y} \d{s}
    \\
    \label{eq:c}
    & &
    +
    \int_I \int_\Omega \int_\Omega
    \psi_h (0,x,s,y) \modulo{u_o (x) - v (0+,x)} \d{x} \d{y} \d{s} \,.
  \end{eqnarray}
  Compute the limit as $h \to 0+$ of the three lines
  separately. First, apply Lemma~\ref{lem:kru2} in the case of a
  function $w$ depending only on the space variable to obtain
  \begin{displaymath}
    \lim_{h\to 0+}[\eqref{eq:a}]
    =
    \lim_{h\to 0+}
    \int_I \int_\Omega \int_\Omega
    \psi_h (0,x,s,y) \modulo{v (s,y) - v (s,x)} \d{x} \d{y} \d{s}
    =
    0\,.
  \end{displaymath}
  Second, by Lemma~\ref{lem:tr2},
  \begin{displaymath}
    \lim_{h\to 0+} [\eqref{eq:b}]
    =
    \int_\Omega
    \int_\Omega
    \left(
      \lim_{h\to 0+}
      \int_I
      \psi_h (0,x,s,y) \modulo{v (s,x) - v (0+,x)} \d{s}
    \right)
    \d{x} \d{y}
    =
    0 \,.
  \end{displaymath}
  Third, by the choice of the function $Y_h$ and~\eqref{eq:8}
  \begin{displaymath}
    \lim_{h\to 0+} [\eqref{eq:c}]
    =
    \int_\Omega \phi (0,x) \modulo{u_o (x) - v (0+,x)} \d{x}
    =
    0\,,
  \end{displaymath}
  proving that $\lim_{h\to 0+} \mbox{\eqref{eq:via1}} = 0$. The
  term~\eqref{eq:via2} is treated exactly in the same way.  Hence,
  \begin{align*}
    \lim_{h \to 0+} \left[\eqref{eq:ante7} \cdots \eqref{eq:via2}
    \right] = {} & \left[\eqref{eq:Kruz}\right],
  \end{align*}
  so that
  \begin{align}
    \nonumber 0 \leq {} & \int_I \int_\Omega \bigl\{ \modulo{u (t,x) -
      v(t,x)} \,
    \partial_t \phi (t,x)
    \\
    \label{eq:okphi}
    & \qquad + \sgn\left(u (t,x) - v(t,x)\right) \left[ f \left(t,x,u
        (t,x)\right) - f \left(t,x,v (t,x)\right) \right] \!\!\cdot
    \grad \phi (t,x)
    \\
    \nonumber & \qquad + \sgn\left(u (t,x) - v(t,x)\right) \, \left[ F
      \left(t,x,u (t,x)\right) - F \left(t,x,v (t,x)\right) \right] \,
    \phi (t,x) \bigr\} \d{x} \d{t}.
  \end{align}
  For $h> 0$, recall the function $\Phi_h \in \Cc2 (\reali^n; [0,1])$
  defined in~\eqref{eq:11}. Let $\Psi \in \Cc2 (]0,T[; \reali^+)$ with
  $\Psi (0)=0$. Note that for any $h>0$ sufficiently small, the map
  \begin{displaymath}
    \phi_h (t,x) = \Psi (t) \, \left( 1 - \Phi_h(x) \right)
    \quad \mbox{for} \quad (t,x) \in \left]-\infty,T\right[ \times \reali^n
  \end{displaymath}
  satisfies~\eqref{eq:8}.  Introduce this test function $\phi_h$
  in~\eqref{eq:okphi} and pass to the limit $h \to 0$ to obtain:
  \begin{align}
    \nonumber 0 \leq {} & \int_I \! \int_\Omega \bigl\{ \modulo{u
      (t,x) - v(t,x)} \, \Psi' (t)
    \\
    \label{eq:2}
    & + \sgn\left(u (t,x) - v(t,x)\right) \, \left[ F \left(t,x,u
        (t,x)\right) - F \left(t,x,v (t,x)\right) \right] \, \Psi (t)
    \bigr\} \d{x} \d{t}
    \\
    \nonumber & - \int_I \! \int_{\partial\Omega} \sgn\left(\tr u
      (t,\xi) - \tr v(t,\xi)\right) \left[ f \left(t,\xi,\tr u
        (t,\xi)\right) - f \left(t,\xi,\tr v (t,\xi)\right) \right]
    \cdot \nu (\xi) \, \Psi(t) \d\xi \d{t},
  \end{align}
  where we used Lemma~\ref{lem:uffaMultiD} and Lemma~\ref{lem:tr3},
  which can be applied since the function $(u,v) \to \sgn(u-v)
  \left(f(t,x,u)-f(t,x,v) \right)$ is Lipschitz continuous,
  see~\cite[Lemma~3]{Kruzkov}.

  \medskip

  To ease readability, we now omit the dependence on $(t,\xi)$ of $f,
  \tr u, \tr v, u_b, v_b, \nu$.  Apply~\eqref{eq:6} to $u$ choosing $k
  = \tr v$ and to $v$ choosing $k= \tr u$:
  \begin{align*}
    - \sgn\left( \tr u - \tr v \right) \left[ f \left( \tr u \right) -
      f \left( \tr v \right) \right]\cdot \nu \leq {} & - \sgn\left(
      u_b - \tr v \right) \left[ f \left( \tr u \right) - f \left( \tr
        v \right) \right]\cdot \nu,
    \\
    - \sgn\left( \tr u - \tr v \right) \left[ f \left( \tr u \right) -
      f \left( \tr v \right) \right]\cdot \nu \leq {} & - \sgn\left(
      v_b - \tr u \right) \left[ f \left( \tr v \right) - f \left( \tr
        u \right) \right]\cdot \nu \, .
  \end{align*}
  Hence,
  \begin{equation}
    \label{eq:5}
    \begin{array}{ll}
      &
      - \sgn\left( \tr u  - \tr v  \right)
      \left[ f \left(\tr u  \right) - f \left(\tr v \right) \right]\cdot \nu
      \\[6pt]
      \leq
      &
      \displaystyle
      \frac12 \left[ \sgn\left( v_b  - \tr u  \right) -
        \sgn\left( u_b  - \tr v  \right) \right]
      \left[ f \left(\tr u  \right) - f \left(\tr v \right) \right]\cdot \nu\,.
    \end{array}
  \end{equation}
  The second line in~\eqref{eq:5} attains the following values:
  \renewcommand{\arraystretch}{1.5}{\begin{center}
      \begin{tabular}{@{}c|c|c|c@{}}
        & $u_b -\tr v > 0$
        & $u_b -\tr v = 0$
        & $u_b -\tr v < 0$ \\
        \hline
        $v_b -\tr u > 0$
        & $0$
        & $ \frac12 \left(f (\tr u) - f (\tr v)\right) \cdot \nu $
        & \fbox{$\boldsymbol{\left(f (\tr u) - f (\tr v)\right) \cdot \nu}$}$\vphantom{\Big|^|_|}$ \\
        \hline
        $v_b -\tr u = 0$
        & $\frac12 \left(f (\tr v) - f (\tr u)\right) \cdot \nu$
        & $0$
        &  $ \frac12\left(f (\tr u) - f (\tr v)\right) \cdot \nu $ \\
        \hline
        $v_b -\tr u < 0$
        & \fbox{$\boldsymbol{\left(f (\tr v) - f (\tr u)\right) \cdot \nu}$}$\vphantom{\Big|^|_|}$
        & $\frac12 \left(f (\tr v) - f (\tr u)\right) \cdot \nu$
        & $0$
      \end{tabular}
    \end{center}} \vspace{0.2cm}
  \noindent Clearly, we can reduce our study to two cases highlighted
  in the table above.  Applying~\eqref{eq:alive} to $u$ with $k = u_b$
  and to $v$ with $k=v_b$ leads to
  \begin{align}
    \label{eq:A}
    \sgn\left( \tr u - u_b \right) \left[ f\left(\tr u \right) -
      f\left(u_b \right) \right] \cdot \nu \geq {} & 0 \,,
    \\
    \label{eq:B}
    \sgn\left( \tr v - v_b \right) \left[ f\left(\tr v \right) -
      f\left(v_b\right) \right] \cdot \nu \geq {} & 0 \,.
  \end{align}
  Let $J = J (t,\xi) =\left\{k \in \reali \colon \left(u_b (t,\xi)
      -k\right) \left(k-v_b (t,\xi)\right) \geq 0\right\}$.  Focus on
  each case separately.
  \begin{description}
  \item[Case I:] $\boldsymbol{u_b -\tr v \leq 0}$ and $\boldsymbol{v_b
      -\tr u \geq 0}$.
    \begin{enumerate}
    \item If $u_b \leq \tr v \leq \tr u \leq v_b$ or $u_b \leq \tr u
      \leq \tr v \leq v_b$, then
      \begin{displaymath}
        \left(f (\tr u) - f (\tr v)\right) \cdot \nu
        \leq
        \sup_{s,r \in J}
        \norma{f (s) - f (r)},
      \end{displaymath}
      where $\norma{\cdot}$ represents the Euclidean norm in
      $\reali^n$.
    \item If $\tr u \leq v_b \leq u_b \leq \tr v$ or $\tr u \leq u_b
      \leq v_b \leq \tr v$, then
      \begin{align*}
        \mbox{by~\eqref{eq:A}} & \Rightarrow f(\tr u) \cdot \nu \leq f
        (u_b) \cdot \nu, & & & \mbox{by~\eqref{eq:B}} & \Rightarrow
        f(\tr v) \cdot \nu \geq f (v_b) \cdot \nu.
      \end{align*}
      Hence we have
      \begin{displaymath}
        \left(f (\tr u) - f (\tr v)\right) \cdot \nu
        \leq
        \left(f (u_b) - f (v_b)\right) \cdot \nu
        \leq
        \sup_{s,r \in J}
        \norma{f (s) - f (r)}.
      \end{displaymath}
    \item If $u_b\leq \tr u \leq v_b \leq \tr v$, by~\eqref{eq:B}
      $f(\tr v) \cdot \nu \geq f (v_b) \cdot \nu$, and using the fact
      that $\tr u \in [u_b,v_b]$, we get
      \begin{displaymath}
        \left(f (\tr u) - f (\tr v)\right) \cdot \nu
        \leq
        \left(f (\tr u) - f (v_b)\right) \cdot \nu
        \leq
        \sup_{s,r \in J}
        \norma{f (s) - f (r)}.
      \end{displaymath}
    \item If $\tr u \leq u_b\leq \tr v \leq v_b$, by~\eqref{eq:A}
      $f(\tr u) \cdot \nu \leq f (u_b) \cdot \nu$, and using the fact
      that $\tr v \in [u_b,v_b]$, we obtain
      \begin{displaymath}
        \left(f (\tr u) - f (\tr v)\right) \cdot \nu
        \leq
        \left(f (u_b) - f (\tr v)\right) \cdot \nu
        \leq
        \sup_{s,r \in J}
        \norma{f (s) - f (r)}.
      \end{displaymath}
    \end{enumerate}
  \item[Case II:] $\boldsymbol{u_b -\tr v \geq 0}$ and
    $\boldsymbol{u_b -\tr v \leq 0}$.
    \begin{enumerate}
    \item If $v_b \leq \tr v \leq \tr u \leq u_b$ or $v_b \leq \tr u
      \leq \tr v \leq u_b$, then
      \begin{displaymath}
        \left(f (\tr v) - f (\tr u)\right) \cdot \nu
        \leq
        \sup_{s,r \in J}
        \norma{f (s) - f (r)}.
      \end{displaymath}
    \item If $\tr v \leq v_b \leq u_b \leq \tr u$ or $\tr v \leq u_b
      \leq v_b \leq \tr u$, then
      \begin{align*}
        \mbox{by~\eqref{eq:A}} & \Rightarrow f(\tr u) \cdot \nu \geq f
        (u_b) \cdot \nu, & & & \mbox{by~\eqref{eq:B}} & \Rightarrow
        f(\tr v) \cdot \nu \leq f (v_b) \cdot \nu.
      \end{align*}
      Hence we have
      \begin{displaymath}
        \left(f (\tr v) - f (\tr u)\right) \cdot \nu
        \leq
        \left(f (v_b) - f (u_b)\right) \cdot \nu
        \leq
        \sup_{s,r \in J}
        \norma{f (s) - f (r)}.
      \end{displaymath}
    \item If $v_b \leq \tr v \leq u_b \leq \tr u$, by~\eqref{eq:A}
      $f(\tr u) \cdot \nu \geq f (u_b) \cdot \nu$ and using the fact
      that $\tr v \in [v_b,u_b]$, we get
      \begin{displaymath}
        \left(f (\tr v) - f (\tr u)\right) \cdot \nu
        \leq
        \left(f (\tr v) - f (u_b)\right) \cdot \nu
        \leq
        \sup_{s,r \in J}
        \norma{f (s) - f (r)}.
      \end{displaymath}
    \item If $\tr v \leq v_b \leq \tr u \leq u_b$, by~\eqref{eq:B}
      $f(\tr v) \cdot \nu \leq f (v_b) \cdot \nu$, and using the fact
      that $\tr u \in [v_b,u_b]$, we obtain
      \begin{displaymath}
        \left(f (\tr v) - f (\tr u)\right) \cdot \nu
        \leq
        \left(f (v_b) - f (\tr u)\right) \cdot \nu
        \leq
        \sup_{s,r \in J}
        \norma{f (s) - f (r)}.
      \end{displaymath}
    \end{enumerate}
  \end{description}
  Hence, \eqref{eq:5} can be estimated as follows:
  \begin{align*}
    & - \sgn\left( \tr u (t,\xi) - \tr v (t,\xi) \right) \left[ f
      \left( t,\xi,\tr u (t,\xi) \right) - f \left(t,\xi,\tr v (t,\xi)
      \right) \right]\cdot \nu (\xi)
    \\
    \leq {} &
    \begin{aligned}[t]
      \frac12 \left[ \sgn\left( v_b (t,\xi) - \tr u (t,\xi) \right) -
        \sgn\left( u_b (t,\xi) - \tr v (t,\xi) \right) \right] &
      \\
      \times \left[ f \left(t,\xi,\tr u (t,\xi) \right) - f
        \left(t,\xi,\tr v (t,\xi) \right) \right]\cdot \nu (\xi) &
    \end{aligned}
    \\
    \leq {} & \sup_{s,r \in J (t,\xi)} \norma{f (t,\xi,s) - f
      (t,\xi,r)}
    \\
    \leq {} & \norma{\partial_u f}_{\L\infty (\Sigma; \reali^n)} \,
    \modulo{u_b (t,\xi) - v_b (t,\xi)}.
  \end{align*}
  Since $\Psi$ assumes only positive values, we can
  estimate~\eqref{eq:2} by
  \begin{equation}
    \label{eq:3}
    \begin{array}{@{}r@{\,}l@{}}
      0 \leq \left[\eqref{eq:2}\right]
      \leq
      &
      \displaystyle
      \int_I \! \int_\Omega \!
      \left\{
        \modulo{u (t,x) - v(t,x)} \, \Psi' (t)
        +
        \norma{\partial_u F}_{\L\infty (\Sigma; \reali)}
        \modulo{u (t,x) - v(t,x)} \Psi (t)
      \right\}
      \d{x}
      \d{t}
      \\
      &
      \displaystyle
      \qquad
      +
      \norma{\partial_u f}_{\L\infty (\Sigma; \reali^n)}
      \int_I \! \int_{\partial\Omega}
      \modulo{u_b (t,\xi) - v_b (t,\xi)}
      \, \Psi (t) \d\xi \d{t}.
    \end{array}
  \end{equation}
  Introduce $\tau, t$ such that $0<\tau<t<T$. Note that the map $s
  \to \Psi_h (s)$ defined by\\
  \begin{minipage}{0.5\linewidth}
    \begin{displaymath}
      \begin{array}{l}
        \displaystyle
        \Psi_h (s) =  \alpha_h (s-\tau-h) - \alpha_h (s-t-h),
        \\
        \displaystyle
        \mbox{where }
        \alpha_h (z) = \int_{-\infty}^z Y_h (\zeta) \d\zeta
        \\
        \mbox{and } Y_h \mbox{ as in \eqref{eq:Yh}},
      \end{array}
    \end{displaymath}
  \end{minipage}%
  \begin{minipage}{0.5\linewidth}
    \begin{center}
      \includegraphics[width=0.8\textwidth]{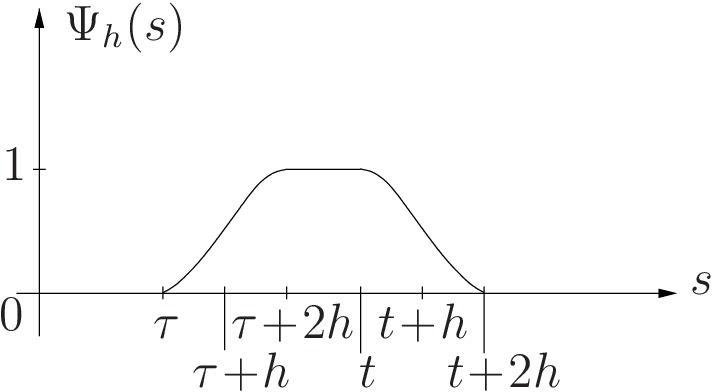}
    \end{center}
  \end{minipage}\\ \smallskip satisfies~\eqref{eq:8}. Hence, we
  substitute $\Psi_h$ for $\Psi$ in~\eqref{eq:3}. Observe that
  $\Psi_h \to \caratt{[\tau,t]}$ and $\Psi_h' \to \delta_\tau
  -\delta_t$ as $h$ tends to $0$.  At the limit we obtain
  \begin{align*}
    0 \leq {} & \int_\Omega \modulo{u (\tau,x) - v (\tau,x)} \d{x} -
    \int_\Omega \modulo{u (t,x) - v (t,x)} \d{x}
    \\
    & + \norma{\partial_u F}_{\L\infty (\Sigma; \reali)} \int_\tau^t
    \!\!\! \int_\Omega \modulo{u (s,x) - v (s,x)} \d{x} \d{s}
    \\
    & + \norma{\partial_u f}_{\L\infty (\Sigma; \reali^n)} \int_\tau^t
    \!\!\! \int_{\partial\Omega} \modulo{u_b (s,\xi) - v_b (s,\xi)}
    \d\xi \d{s} \,.
  \end{align*}
  A Gronwall type argument yields
  \begin{align}
    \nonumber & \int_\Omega \modulo{u (t,x) - v (t,x)} \d{x}
    \\
    \label{eq:9}
    \leq {} & e^{\norma{\partial_u F}_{\L\infty (\Sigma; \reali)}
      (t-\tau)} \int_\Omega \modulo{u (\tau,x) - v (\tau,x)} \d{x}
    \\
    \nonumber & + \norma{\partial_u f}_{\L\infty (\Sigma; \reali^n)}
    \int_\tau^t e^{(t-\tau-s) \, \norma{\partial_u F}_{\L\infty
        (\Sigma; \reali)}} \int_{\partial\Omega} \modulo{u_b (s,\xi) -
      v_b (s,\xi)} \d\xi \d{s} \,.
  \end{align}
  In the limit $\tau \to 0$ for a.e.~$\tau$, an application of
  Proposition~\ref{prop:sol0} completes the proof when $u_o \in \C2
  (\bar{\Omega} ; \reali)$ and $u_{b} \in \C2 (I \times \partial
  \Omega; \reali)$. The general case now follows by a straightforward
  regularization argument.
\end{proofof}

\section{Proofs Related to Section~\ref{sec:NDA}}
\label{sec:PNDA}

\begin{proofof}{Proposition~\ref{prop:sol0}}
  Let $M = \max \{ \norma{u}_{\L\infty (I\times \Omega; \reali)}, \,
  \norma{u_o}_{\L\infty (\Omega; \reali)}, \, \norma{u_b}_{\L\infty (I
    \times \partial\Omega; \reali)} \}$. We first prove that choosing
  $k \in \left]-\infty, -M \right[ \cap \left]M, +\infty\right[$, the
  terms containing $k$ in the left hand side in~\eqref{eq:4}
  vanish. Indeed, assuming $k < - M$, observe that
  \begin{equation}
    \label{eq:segni}
    \begin{aligned}[c]
      \modulo{u (t,x) - k} = {} & u (t,x) - k & & & \sgn\left(u (t,x)
        - k\right) = {} & 1
      \\
      \modulo{u_o (x) - k} = {} & u_o (x) - k & & & \sgn\left(u_b
        (t,\xi) - k\right) = {} & 1 \,.
    \end{aligned}
  \end{equation}
  Therefore, the terms containing $k$ in the left hand side
  in~\eqref{eq:4} are:
  \begin{align*}
    \int_I \int_\Omega \left( - k \, \partial_t \phi (t,x) - f (t,x,k)
      \cdot \nabla \phi (t,x) - \div f (t,x,k) \, \phi (t,x) \right)
    \d{x} \d{t} &
    \\
    -\int_\Omega k \, \phi (0,x)\d{x} +\int_I \int_{\partial\Omega} f
    (t,\xi,k) \cdot \nu (\xi) \phi (t,\xi) \d{\xi} \d{t} &
    \\
    = \int_I \int_\Omega - \div \left(f (t,x,k) \, \phi (t,x)\right)
    \d{x} \d{t} + \int_I \int_{\partial\Omega} f (t,\xi,k) \cdot \nu
    (\xi) \phi (t,\xi) \d{\xi} \d{t} &
    \\
    =0 \,.&
  \end{align*}
  The inequality~\eqref{eq:4} now reads
  \begin{align*}
    0 \leq {} & \int_I \int_\Omega \Big\{ u (t,x) \, \partial_t\phi
    (t,x) + f (t,x,u) \cdot \nabla \phi (t,x) + F (t,x,u) \, \phi
    (t,x) \Big\} \d{x} \d{t}
    \\
    & + \int_\Omega \phi (0,x) \, u_o (x) \d{x} - \int_I
    \int_{\partial\Omega} f\left(t,\xi,\left(\tr u\right)
      (t,\xi)\right) \cdot \nu (\xi) \, \phi (t,\xi) \, \d{\xi} \d{t}
    \\
    = {} & \int_I \int_\Omega \Big\{ u (t,x) \, \partial_t\phi (t,x) +
    f (t,x,u) \cdot \nabla \phi (t,x) + F (t,x,u) \, \phi (t,x) \Big\}
    \d{x} \d{t}
    \\
    & + \int_\Omega \phi (0,x) \, u_o (x) \d{x} - \int_I
    \int_{\partial\Omega} \tr f\left(t,\xi, u (t,\xi)\right) \cdot \nu
    (\xi) \, \phi (t,\xi) \, \d{\xi} \d{t}
    \\
    = {} & \int_I \int_\Omega \Big\{ u (t,x) \, \partial_t\phi (t,x) -
    \phi (t,x) \, \div f (t,x,u) + F (t,x,u) \, \phi (t,x) \Big\}
    \d{x} \d{t}
    \\
    & + \int_\Omega \phi (0,x) \, u_o (x) \d{x},
  \end{align*}
  where we apply Lemma~\ref{lem:tr3} and the Divergence Theorem.
  Choose $\phi_k \in \Cc2 (]-\infty,T[ \times \reali^n; \reali^+)$ as
  \begin{displaymath}
    \phi_k (t,x) = \theta_k (t) \, \psi (x),
  \end{displaymath}
  where $\theta_k \in \Cc2 ([0,T[; [0,1])$ is such that\\
  \begin{minipage}[c]{.6\linewidth}
    \begin{align*}
      &\theta_k (0) = 1, \phantom{\sum}
      \\
      & \theta_k (t) = 0 \mbox{ for all } t \geq 1/\modulo{k},
      \\
      & \sup_k \frac{1}{\modulo{k}} \norma{\theta'_k}_{\C0} < +
      \infty,
    \end{align*}
  \end{minipage}%
  \begin{minipage}[c]{.4\linewidth}
    \begin{center}
      \includegraphics[width=0.8\textwidth]{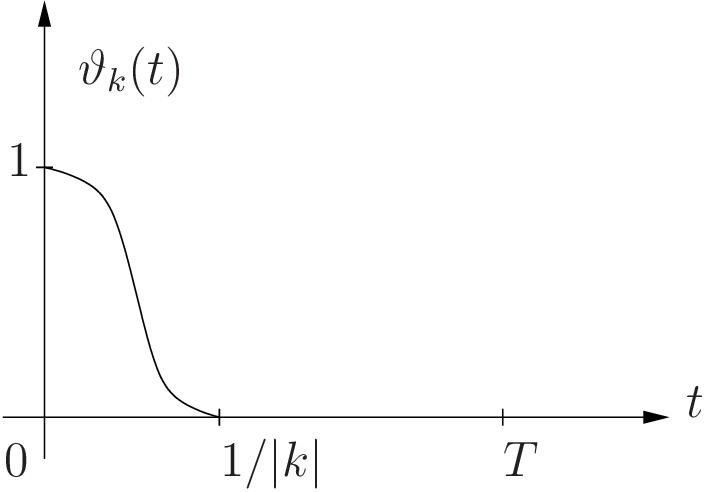}
    \end{center}
  \end{minipage}\\
  while $\psi \in \Cc2 (\reali^n; \reali^+)$.  Hence,
  \begin{align*}
    0 \leq {} & \int_I \int_\Omega \Big\{ u (t,x) \, \theta_k' (t) \,
    \psi (x) + \left( F (t,x,u) -\div f (t,x,u) \right) \theta_k (t)
    \, \psi (x) \Big\} \d{x} \d{t}
    \\
    & + \int_\Omega \psi (x) \, u_o (x) \d{x}.
  \end{align*}
  Pass now to the limit for $k \to - \infty$. Observe that, by the
  Dominated Convergence Theorem,
  \begin{displaymath}
    \lim_{k \to + \infty} \int_I \int_\Omega
    \left( F (t,x,u) -\div f (t,x,u) \right)  \theta_k (t) \, \psi (x)
    \d{x} \d{t}
    =
    0.
  \end{displaymath}
  Thanks to Lemma~\ref{lem:uffa1} and to the Dominated Convergence
  Theorem we also have
  \begin{displaymath}
    \lim_{k \to + \infty} \int_I \int_\Omega
    u (t,x) \, \theta_k' (t) \, \psi (x)
    \d{x} \d{t}
    =
    - \int_\Omega u (0+,x) \, \psi (x) \d{x}.
  \end{displaymath}
  Then, in the case $k < -M$, \eqref{eq:4} reduces to
  \begin{displaymath}
    0 \leq
    \int_\Omega
    \psi (x) \, u_o (x) \d{x}
    -
    \int_\Omega u (0+,x) \, \psi (x) \d{x}.
  \end{displaymath}
  If $k>M$ the signs in~\eqref{eq:segni} are opposite and analogous
  computations show that~\eqref{eq:4} reduces to
  \begin{displaymath}
    0 \geq
    \int_\Omega
    \psi (x) \, u_o (x) \d{x}
    -
    \int_\Omega u (0+,x) \, \psi (x) \d{x}.
  \end{displaymath}
  Hence,
  \begin{displaymath}
    \int_\Omega
    \psi (x) \left( u_o (x) - u (0+,x) \right) \d{x}
    = 0
    \quad \mbox{ for all } \quad
    \psi \in \Cc2 (\reali^n; \reali^+).
  \end{displaymath}
  We then obtain that
  \begin{eqnarray*}
    0
    & = &
    \int_\Omega
    \psi (x) \left(u_o (x) - u (0+,x)\right)  \d{x}
    \\
    & = &
    \int_\Omega
    \psi (x)
    \left(
      u_o (x) - \lim_{t \to 0+} \frac{1}{t}\int_0^t u (\tau,x)
      \d{\tau}
    \right)  \d{x}
    \\
    & = &
    \lim_{t \to 0+} \frac{1}{t}
    \int_\Omega
    \int_0^t
    \psi (x)
    \left(
      u_o (x) -  u (\tau,x)
    \right) \d{\tau} \d{x}
    \\
    & = &
    \lim_{t \to 0+} \frac{1}{t}
    \int_0^t
    \int_\Omega
    \psi (x)
    \left(
      u_o (x) -  u (\tau,x)
    \right) \d{x} \d{\tau}  \,.
  \end{eqnarray*}
  Therefore, there exists a set $\mathcal{E} \subset I$ with measure
  $0$ such that
  \begin{displaymath}
    \lim_{t \to 0+,\, t \in I\setminus \mathcal{E}}
    \int_\Omega
    \psi (x)
    \left(
      u_o (x) -  u (t,x)
    \right)  \d{x}
    =
    0
  \end{displaymath}
  and, by the arbitrariness of $\psi$, the proof is completed.
\end{proofof}

\begin{proofof}{Proposition~\ref{prop:min}}
  Let $\Psi \in \Cc2 (]0,T[ \times \reali^n; \reali^+)$ and $\Phi_h$
  as in~\eqref{eq:11}.  Write~\eqref{eq:4} with $\phi (t,x) = \Psi
  (t,x) \, \Phi_h (x) $ and take the limit as $h \to 0$. For all $k
  \in \reali$:
  \begin{eqnarray*}
    \lim_{h\to 0}
    \int_I \!\!\int_\Omega
    \modulo{u (t,x) -k} \, \partial_t\Psi (t,x) \, \Phi_h (x)
    \d{x} \d{t}
    & = &
    0;
    \\
    \lim_{h\to 0}
    \int_I \!\!\int_\Omega
    \sgn (u (t,x) -k) \,
    \left(f (t,x,u) - f (t,x,k)\right) \cdot \nabla \Psi (t,x) \Phi_h (x)
    \d{x} \d{t}
    & = &
    0;
    \\
    \lim_{h\to 0}
    \int_I \!\!\int_\Omega
    \sgn (u (t,x) -k) \,
    \left(f (t,x,u) - f (t,x,k)\right) \cdot \nabla \Phi_h (x) \, \Psi (t,x)
    \d{x} \d{t}
    & = &
    \\
    \int_I \!\! \int_{\partial\Omega}
    \sgn\left(\tr u (t,\xi) - k\right)
    \left(f \left(t,\xi,\tr u (t,\xi)\right) - f (t,\xi,k)\right)
    \cdot \nu(\xi) \, \Psi (t,\xi)
    \d{\xi} \d{t}
    \\
    \lim_{h\to 0}
    \int_I \!\! \int_{\Omega}
    \sgn (u (t,x) -k) \,
    \left( F (t,x,u) - \div f (t,x,k) \right)\, \Psi (t,x) \, \Phi_h (x)
    \d{x} \d{t}
    & = &
    0;
    \\
    \lim_{h\to 0}
    \int_\Omega
    \Psi (0,x) \Phi_h (x) \modulo{u_o (x) - k} \d{x}
    & = &
    0;
    \\
    \lim_{h\to 0}
    \int_I \!\!\int_{\partial\Omega}  \sgn (u_b (t,\xi) - k)
    \left[
      f\left(t,\xi,\tr u (t,\xi)\right) - f (t,\xi,k)
    \right] \cdot \nu (\xi) \,
    \Psi (t,\xi) \, \Phi_h (\xi)
    \d{\xi} \d{t}
    \\
    =
    \int_I \!\!\int_{\partial\Omega}  \sgn (u_b (t,\xi) - k)
    \left(
      f\left(t,\xi,\tr u (t,\xi)\right) - f (t,\xi,k)
    \right) \cdot \nu (\xi) \,
    \Psi (t,\xi)
    \d{\xi} \d{t}\,,
  \end{eqnarray*}
  where we used the Dominated Convergence Theorem,
  Lemma~\ref{lem:uffaMultiD} and Lemma~\ref{lem:tr3}. The latter Lemma
  can be used since the function $(u,k) \to \sgn (u-k) \left(f (t,x,u)
    - f (t,x,k)\right)$ is Lipschitz continuous
  by~\cite[Lemma~3]{Kruzkov}. Therefore,
  \begin{align*}
    \int_I \! \int_{\partial\Omega} \left[ \sgn\left( \tr u(t,\xi) - k
      \right) - \sgn\left( u_b (t,\xi) - k \right) \right] \left[
      f\left(t,\xi, \tr u (t,\xi)\right) - f\left(t,\xi,k\right)
    \right] &
    \\
    \cdot \nu (\xi) \, \Psi (t,x) \d\xi \d{t} & \geq 0.
  \end{align*}
  Hence,
  \begin{equation}
    \label{eq:viva}
    \left[
      \sgn\left( \tr u(t,\xi) - k \right) -
      \sgn\left( u_b (t,\xi) - k \right)
    \right]
    \left[
      f\left(t,\xi, \tr u (t,\xi)\right) - f\left(t,\xi,k\right)
    \right] \cdot \nu (\xi) \geq 0
  \end{equation}
  almost everywhere on $]0,T[ \times \partial\Omega$ for all $k \in
  \reali$.  Inequality~\eqref{eq:viva} is reduced to~\eqref{eq:alive}
  by taking $k$ in the interval $\mathcal{I} (t,\xi)$.
\end{proofof}

\begin{proofof}{Proposition~\ref{prop:equi}}
  Let $u$ satisfy Definition~\ref{def:e_sol}. Then, choose for
  instance $\mathcal{E}_m (u) = \sqrt{\frac{1}{m} + (u-k)^2}$ for $k
  \in \reali$ and $m \in \naturali$. The entropy flux is then defined
  by~3.~in Definition~\ref{def:pair}. A standard limiting procedure
  allows to obtain~\eqref{eq:4} in the limit $m \to +\infty$.

  Conversely, let $u$ solve~\eqref{eq:1} in the sense of
  Definition~\ref{def:sol} and assume that $\norma{u}_{\L\infty
    (I\times\Omega;\reali)} \leq M$. Then, clearly, $u$
  satisfies~\eqref{eq:dis_entro} with $\mathcal{E} (u) = \alpha
  \modulo{u-k} + \beta$, for any $\alpha>0$ and $k,\beta \in
  \reali$. Further, note that if $u$ satisfies~\eqref{eq:dis_entro}
  with two pairs $(\mathcal{E}_1,\mathcal{F}_1)$ and $(\mathcal{E}_2,
  \mathcal{F}_2)$ (for continuous maps $\mathcal{E}_1, \mathcal{E}_2,
  \mathcal{F}_1, \mathcal{F}_2)$, then it satisfies the same
  inequality also with respect to $(\mathcal{E}_1+\mathcal{E}_2,
  \mathcal{F}_1+\mathcal{F}_2)$. Inductively, $u$
  satisfies~\eqref{eq:dis_entro} for any pair
  $(\mathcal{E},\mathcal{F})$ with $\mathcal{E}$ piecewise linear and
  continuous on $[-M,M]$. Remark also that if $u$
  satisfies~\eqref{eq:dis_entro} with respect to the continuous pairs
  $(\mathcal{E}_n, \mathcal{F}_n)$ and the $\mathcal{E}_n$ are
  uniformly convergent to $\mathcal{E}$ on $[-M,M]$, then $u$
  satisfies~\eqref{eq:dis_entro} also with respect to the pair
  $(\mathcal{E},\mathcal{F}$), where $\mathcal{F}$ is given by~3.~in
  Definition~\ref{def:e_sol}. Finally, since any convex entropy
  $\mathcal{E}$ is the uniform limit on $[-M,M]$ of piecewise linear
  and continuous functions, we obtain the proof.
\end{proofof}

\begin{proofof}{Theorem~\ref{thm:Dream}}
  The proof consists in regularizing the initial datum through a
  sequence $u_o^m$. Applying Theorem~\ref{thm:estConComp}, we have a
  sequence of solutions $u_m$. Theorem~\ref{thm:wp} allows to prove
  that $u_m$ satisfies the Cauchy condition, hence converges to a map
  $u$, which is proved to solve~\eqref{eq:1}.

  To approximate the initial datum, using~\cite[Formula~(1.8) and
  Proposition~1.15]{GiustiBook} introduce a sequence $\tilde u_m \in
  \C\infty (\Omega;\reali)$ such that
  \begin{displaymath}
    \lim_{m\to +\infty} \norma{u_o - \tilde u_m}_{\L1 (\Omega;\reali)}=0
    \,, \quad
    \norma{\tilde u_m}_{\L\infty (\Omega;\reali)}
    \leq
    \norma{u_o}_{\L\infty (\Omega;\reali)}
    \quad \mbox{ and }
    \lim_{m\to+\infty} \tv (\tilde u_m) = \tv (u_o) \,.
  \end{displaymath}
  Define now $\Psi_m = 1-\Phi_{1/m}$, with $\Phi_{1/m}$ as
  in~\eqref{eq:11}. Let
  \begin{equation}
    \label{eq:def_um}
    u_o^m (x) = \Psi_m (x) \; \tilde u_m (x)
    \qquad \mbox{ for all } \quad x \in \Omega\,.
  \end{equation}
  By construction, $\lim_{m \to +\infty} \norma{u_o^m-u_o}_{\L1
    (\Omega;\reali)}=0$, so that $u_o^m$ is a Cauchy sequence in $\L1
  (\Omega;\reali)$. We have also the uniform bounds
  \begin{eqnarray}
    \label{eq:qq1}
    \norma{u_o^m}_{\L\infty (\Omega;\reali)}
    & \leq &
    \norma{u_o}_{\L\infty (\Omega;\reali)} \,;
    \\
    \nonumber
    \tv (u_o^m)
    & \leq &
    \norma{\grad \Psi_m}_{\L1 (\Omega;\reali^n)}
    \;
    \norma{\tilde u_m}_{\L\infty (\Omega;\reali)}
    +
    \norma{\Psi_m}_{\L\infty (\Omega;\reali)}
    \;
    \norma{\grad \tilde u_m}_{\L1 (\Omega;\reali^n)}
    \\
    \label{eq:qq2}
    & \leq &
    \O \, \norma{u_o}_{\L\infty (\Omega;\reali)}
    +
    \tv (u_o) \,.
  \end{eqnarray}

  Since for any $m \in \naturali \setminus \{0\}$ we have that $u_o^m
  (\xi) = u_b (0,\xi) = 0$, Theorem~\ref{thm:estConComp} applies
  to~\eqref{eq:1} with initial datum $u_o^m$ and boundary datum $u_b$,
  yielding the existence of a solution $u_m$ to~\eqref{eq:1} in the
  sense of Definition~\ref{def:sol}, which satisfies the
  estimates~\eqref{eq:L_bound}, \eqref{eq:tv_bound}
  and~\eqref{eq:t_Lip}. Theorem~\ref{thm:wp} then implies that
  \begin{eqnarray*}
    \int_{\Omega} \modulo{u_{m'}(t,x) - u_{m''}(t,x)} \d{x}
    \leq
    e^{L_F \, t} \int_{\Omega} \modulo{u_o^{m'}(x) - u_o^{m''}(x)} \d{x} \,,
  \end{eqnarray*}
  proving that the sequence $u_m (t)$ satisfies the Cauchy condition
  in $\L1 (\Omega; \reali)$ uniformly in $t \in I$.

  Call $u = \lim_{m\to\infty} u_m$. We now verify that $u$
  solves~\eqref{eq:1}. By Proposition~\ref{prop:equi}, each $u_m$
  satisfies~\eqref{eq:dis_entro} for any $\C2$ entropy $\mathcal{E}_l
  (u) = \sqrt{\frac{1}{l} + (u-k)^2}$ and for any $\phi \in \Cc2
  (\left]-\infty, T\right[\times \reali^n; \reali^+)$:
  \begin{displaymath}
    \!\!\!\!\!\!
    \begin{array}{@{}r@{\,}l@{}}
      \displaystyle
      \int_I \int_\Omega
      \Big\{
      \mathcal{E}_l\left(u_m (t,x)\right) \partial_t \phi (t,x)
      +
      \mathcal{F}_l\left(t,x,u_m (t,x)\right) \cdot \grad \phi (t,x)
      \\
      +
      \left[
        \mathcal{E}_l'\left(u_m (t,x)\right)
        \left(
          F\left(t,x,u_m (t,x)\right)
          -
          \div f \left(t,x,u_m (t,x)\right)
        \right)
        +
        \div\mathcal{F}_l \left(t,x,u_m (t,x)\right)
      \right]
      \phi (t,x)
      \Big\}
      \d{x} \d{t}
      \\
      \displaystyle
      +
      \int_\Omega \mathcal{E}_l\left(u_o^m (x)\right) \; \phi (0,x)\d{x}
      \\
      \displaystyle
      -
      \int_I \int_{\partial\Omega}
      \left[
        \mathcal{F}_l\left(t, \xi, u_b (t,\xi)\right)
        -
        \mathcal{E}_l'\left(u_b (t,\xi)\right)
        \left(
          f\left(t,\xi,u_b (t,\xi)\right)-
          f\left(t, \xi, \tr u_m (t,\xi)\right)
        \right)
      \right]
      \!\!\cdot \nu (\xi) \phi (t,\xi)
      \d\xi \d{t}
      & \geq 0.
    \end{array}
    \!\!\!\!\!\!\!\!\!
  \end{displaymath}
  In the limit $m\to+\infty$, since $u_m$ converges in $\L1$ to $u$,
  $\tr u_m$ converges to $\tr u$ by Lemma~\ref{lem:tr2}, which can be
  applied thanks to the estimate~\eqref{eq:tv_bound}. Hence, we have
  \begin{displaymath}
    \!\!\!\!\!\!
    \begin{array}{@{}r@{\,}l@{}}
      \displaystyle
      \int_I \int_\Omega
      \Big\{
      \mathcal{E}_l\left(u (t,x)\right) \partial_t \phi (t,x)
      +
      \mathcal{F}_l\left(t,x,u (t,x)\right) \cdot \grad \phi (t,x)
      \\
      +
      \left[
        \mathcal{E}_l'\left(u (t,x)\right)
        \left(
          F\left(t,x,u (t,x)\right)
          -
          \div f \left(t,x,u (t,x)\right)
        \right)
        +\div \mathcal{F}_l\left(t,x,u (t,x)\right)
      \right]
      \phi (t,x)
      \Big\}
      \d{x} \d{t}
      \\
      \displaystyle
      -
      \int_I \int_{\partial\Omega}
      \left[
        \mathcal{F}_l\left(t, \xi, u_b (t,\xi)\right)
        -
        \mathcal{E}_l'\left(u_b (t,\xi)\right)
        \left(
          f\left(t,\xi,u_b (t,\xi)\right)-
          f\left(t, \xi, \tr u (t,\xi)\right)
        \right)
      \right]
      \!\!\cdot \nu (\xi) \phi (t,\xi)
      \d\xi \d{t}
      & \geq 0.
    \end{array}
    \!\!\!\!\!\!\!\!\!
  \end{displaymath}
  In the limit $l \to +\infty$, we have the convergences
  $\mathcal{E}_l \to \mathcal{E}$ and $\mathcal{F}_l\to \mathcal{F}$,
  with $\mathcal{E} (u) = \modulo{u-k}$ and $\mathcal{F} (t,x,u) =
  \sgn (u-k) \left(f (t,x,u) - f (t,x,k)\right)$, so that
  \begin{displaymath}
    \!\!\!\!
    \begin{array}{@{}rcl@{}}
      \displaystyle
      \int_I \!\int_\Omega
      \Big\{
      \modulo{u (t,x) -k} \, \partial_t\phi (t,x)
      +
      \sgn (u (t,x) -k) \,
      \left(f (t,x,u) - f (t,x,k)\right) \cdot \grad \phi (t,x)
      \\
      \displaystyle
      +
      \sgn (u (t,x) -k) \,
      \left( F (t,x,u) - \div f (t,x,k) \right)\, \phi (t,x)
      \Big\} \d{x} \d{t}
      \\
      \displaystyle
      +
      \int_\Omega
      \modulo{u (0,x) - k} \; \phi (0,x) \d{x}
      \\
      \displaystyle
      -
      \int_I \!\int_{\partial\Omega}  \sgn (u_b (t,\xi) - k)
      \left(
        f\left(t,\xi,\left(\tr u\right) (t,\xi)\right) - f (t,\xi,k)
      \right) \cdot \nu (\xi) \,
      \phi (t,\xi) \,
      \d{\xi} \d{t}
      & \geq & 0.
    \end{array}
    \!\!\!\!\!\!\!\!\!
  \end{displaymath}
  Finally, observe that
  \begin{displaymath}
    \begin{array}{@{}r@{\;}c@{\;}ll@{}}
      \displaystyle
      \int_\Omega
      \modulo{u (0,x) - k} \; \phi (0,x) \d{x}
      & \leq &
      \displaystyle
      \int_\Omega
      \modulo{u_o (x) - k} \; \phi (0,x) \d{x}
      \\
      & &
      \displaystyle
      +
      \int_\Omega
      \modulo{u_o (x) -  u_o^m (x)} \; \phi (0,x) \d{x}
      & \stackrel{m\to+\infty}{\to} 0 \quad
      \mbox{by~\eqref{eq:def_um}}
      \\[8pt]
      & &
      \displaystyle
      +
      \int_\Omega
      \modulo{u_o^m (x) -  u_m (0,x)} \; \phi (0,x) \d{x}
      & \stackrel{m\to+\infty}{\to} 0 \quad
      \mbox{by Proposition~\ref{prop:sol0}}
      \\[8pt]
      & &
      \displaystyle
      +
      \int_\Omega
      \modulo{u_m (0,x) -  u(0,x)} \; \phi (0,x) \d{x}
      & \stackrel{m\to+\infty}{\to} 0 \quad
      \mbox{since } u_m \to u \mbox{ in } \L1,
    \end{array}
  \end{displaymath}
  concluding the existence proof.

  The bounds directly follow from~\eqref{eq:L_bound},
  \eqref{eq:tv_bound} and~\eqref{eq:t_Lip}, thanks to the
  properties~\eqref{eq:qq1} and~\eqref{eq:qq2} of the sequence
  $u_o^m$.
\end{proofof}

\appendix
\section{Appendix: The Trace Operator}

A relevant role is played by the trace operator which we recall here
from~\cite[Paragraph~5.3]{EvansGariepy}.

\begin{definition}
  \label{def:tr}
  Let $A \subset \reali^n$ be bounded with Lipschitz boundary. The
  trace operator is the map $\tr_A \colon \BV (A; \reali) \to \L1
  (\partial A; \reali)$ such that for all $\phi \in \C1 (\reali^n;
  \reali^n)$ and for all $w \in \BV (A; \reali)$,
  \begin{displaymath}
    \int_{\partial A}
    \left((\tr_A w )(\xi)\right) \phi (\xi) \cdot \nu (\xi) \d\xi
    =
    \int_A w (x) \, \div \phi (x) \d{x}
    +
    \int_A \phi (x) \d{\left(\nabla w (x)\right)} \,.
  \end{displaymath}
\end{definition}

Below, when no misunderstanding arises, we omit the dependence of the
trace operator from the set. First of all, we recall without proof the
following two lemmas.

\begin{lemma}[{\cite[Paragraph~5.3, Theorem~1 and
    Theorem~2]{EvansGariepy}}]
  \label{lem:tr2}
  Let $A \subset \reali^n$ be bounded with Lipschitz boundary. Fix $w
  \in \BV (A; \reali)$. Then, the trace operator is a bounded linear
  operator and for ${\cal H}^{n-1}$-a.e.~$\xi \in \partial A$,
  \begin{displaymath}
    \lim_{r \to 0+}
    \frac{1}{{\cal L}^n\left(B (\xi,r) \cap A\right)}
    \int_{B (\xi,r) \cap A} \modulo{w (x) - (\tr w) (\xi)} \d{x} = 0 \,.
  \end{displaymath}
\end{lemma}

\begin{lemma}[{\cite[Paragraph~5.3, Remark to
    Theorem~2]{EvansGariepy}}]
  \label{lem:tr1}
  Let $A \subset \reali^n$ be bounded with Lipschitz boundary. Fix $w
  \in \BV (A; \reali) \cap \C0 (\bar{A}; \reali)$. Then, $(\tr w)
  (\xi) = w (\xi)$ for ${\cal H}^{n-1}$-a.e.~$\xi \in \partial A$.
\end{lemma}

\noindent Recall also the following property.
\begin{lemma}
  \label{lem:tr3}
  Let $A \subset \reali^n$ be bounded with Lipschitz boundary. Fix $w
  \in \BV (A; \reali)$ and $h \in \C{0,1} (\reali; \reali)$. Then,
  $\tr (h \circ w) = h \circ (\tr w)$ for ${\cal H}^{n-1}$-a.e.~$\xi
  \in \partial A$.
\end{lemma}

\begin{proof}
  For any $\xi \in \partial A$ and for $r>0$ sufficiently small,
  compute:
  \begin{eqnarray*}
    & &
    \modulo{\left(\tr (h \circ w)\right) (\xi)
      -
      \left(h \circ (\tr w)\right) (\xi)}
    \\
    & \leq &
    \frac{1}{{\cal L}^n\left(B (\xi,r) \cap A\right)}
    \int_{B (\xi,r) \cap A}
    \modulo{(h \circ w) (x) - \left(\tr (h\circ w) (\xi)\right)}
    \d{x}
    \\
    & &
    +
    \frac{1}{{\cal L}^n\left(B (\xi,r) \cap A\right)}
    \int_{B (\xi,r) \cap A}
    \modulo{(h \circ w) (x) - h\left((\tr w) (\xi)\right)} \d{x}
    \\
    & \leq &
    \frac{1}{{\cal L}^n\left(B (\xi,r) \cap A\right)}
    \int_{B (\xi,r) \cap A}
    \modulo{(h \circ w) (x) - \left(\tr (h\circ w) (\xi)\right)}
    \d{x}
    \\
    & &
    +
    \frac{\Lip (h)}{{\cal L}^n\left(B (\xi,r) \cap A\right)}
    \int_{B (\xi,r) \cap A}
    \modulo{w (x) - (\tr w) (\xi)} \d{x}
  \end{eqnarray*}
  and both addends in the right hand side above vanish as $r \to 0+$
  by Lemma~\ref{lem:tr2}.
\end{proof}

The next two Lemmas relate the values attained by the trace of $u$
with limits at the boundary of integrals of $u$.

\begin{lemma}
  \label{lem:uffa1}
  Let $T>0$ and $u \in \BV ([0,T]; \reali)$. Choose a sequence $\phi_k
  \in \Cc1 ([0,T]; [0,1])$ such that $\phi_k (0) = 1$, $\phi_k (t) =
  0$ for all $t \geq 1/k$ and $\sup_k \frac{1}{k}
  \norma{\phi'_k}_{\L\infty ([0,T];\reali)} < +\infty$.  Then,
  $\displaystyle \lim_{k \to +\infty } \int_I u (t) \, \phi'_k (t)
  \d{t} = - u (0+)$.
\end{lemma}

\noindent Above, we used the standard notation $u (0+) = \tr_{[0,T]} u
(0)$.

\begin{proof}
  Denote $c = \sup_k \frac{1}{k} \, \norma{\phi'_k}_{\L\infty
    ([0,T];\reali^n)}$. Compute:
  \begin{eqnarray*}
    \modulo{\int_I u (t) \, \phi'_k (t) \d{t}
      +
      u (0+)}
    & = &
    \modulo{\int_0^{1/k} u (t) \, \phi'_k (t) \d{t}
      +
      u (0+)}
    \\
    & \leq &
    \int_0^{1/k} \!\! \modulo{u (t) - u (0+)} \modulo{\phi'_k (t)} \d{t}
    +
    \modulo{u (0+) \int_0^{1/k} \!\! \phi'_k (t) \d{t} + u (0+)}
    \\
    & \leq &
    \frac{c}{1/k} \int_0^{1/k} \modulo{u (t) - u (0+)} \d{t}
  \end{eqnarray*}
  which vanishes as $k \to +\infty$ by Lemma~\ref{lem:tr2}.
\end{proof}

\begin{lemma}
  \label{lem:uffaMultiD}
  Let $\Omega$ satisfy~{\bf($\boldsymbol{\Omega_{2,0}}$)} and $u \in
  \BV (\Omega; \reali)$. Choose a sequence $\chi_k \in \Cc1
  (\reali^n;[0,1])$ such that $\chi_k (\xi) = 1$ for all $\xi
  \in \partial\Omega$, $\chi_k (x) = 0$ for all $x \in \Omega$ with $B
  (x,1/k) \subseteq \Omega$ and moreover $\sup_k \frac{1}{k }
  \norma{\nabla \chi_k}_{\L\infty (\Omega;\reali^n)} < +\infty$.
  Then, $\displaystyle \lim_{k \to +\infty} \int_\Omega u (x) \grad
  \chi_k (x) \d{x} = \int_{\partial\Omega} \tr_\Omega u (\xi) \; \nu
  (\xi) \d\xi$. 
\end{lemma}

\begin{proof}
  Let $e_i$ be the $i$-th vector of the standard basis in
  $\reali^n$. Set $\Omega_k = \left\{x \in \Omega \colon d
    (x, \partial\Omega) \leq 1/k\right\}.$ Then:
  \begin{eqnarray*}
    & &
    \left(
      \int_\Omega u (x) \grad \chi_k (x) \d{x}
      -
      \int_{\partial\Omega}
      \tr_\Omega u (\xi) \; \nu (\xi) \d\xi
    \right)
    \cdot e_i
    \\
    & = &
    \int_\Omega u (x) \, \partial_i\chi_k (x) \d{x}
    -
    \int_{\partial\Omega}
    \tr u (\xi) \; \chi_k (\xi) \, e_i \cdot \nu(\xi) \d\xi
    \\
    & = &
    \int_{\Omega_k} u (x) \, \partial_i\chi_k (x) \d{x}
    -
    \int_{\partial\Omega_k}
    \tr u (\xi) \; \chi_k (\xi) \, e_i \cdot \nu(\xi) \d\xi
    \\
    & = &
    - \int_{\Omega_k} \chi_k (x) \d{(\nabla u)_i (x)}
  \end{eqnarray*}
  where Definition~\ref{def:tr} was used to obtain the last
  expression. By the Dominated Convergence Theorem, $\lim_{k \to
    +\infty} \int_{\Omega_k} \chi_k (x) \d{(\nabla u)_i (x)} =0$,
  completing the proof.
\end{proof}

\noindent\textbf{Acknowledgment:} Both authors thank Boris Andreianov and Piotr Gwiazda for useful discussions. The present work was supported by
the PRIN~2012 project \emph{Nonlinear Hyperbolic Partial Differential
  Equations, Dispersive and Transport Equations: Theoretical and
  Applicative Aspects} and by the INDAM--GNAMPA~2014 project
\emph{Conservation Laws in the Modeling of Collective Phenomena}.

\small{

  \bibliography{ColomboRossi}

  \bibliographystyle{abbrv}

}
\end{document}